\documentclass[a4paper,10pt,reqno]{article}
\usepackage{amsthm}
\usepackage{amsmath,amssymb,amsfonts,amsxtra}
\usepackage{graphicx}
\usepackage{epic}
\usepackage{verbatim}
\usepackage[french,english]{babel}
\usepackage[T1]{fontenc}
\usepackage{lmodern}
\usepackage[mathscr]{euscript}
\usepackage{enumerate}
\usepackage{xcolor}
\definecolor{rouge}{rgb}{0.7,0,0}
\definecolor{bleu}{rgb}{0,0,0.7}
\usepackage[colorlinks=true,linkcolor=bleu,urlcolor=violet,citecolor=rouge]{hyperref}
\usepackage{breakurl}
\usepackage{xspace}

\usepackage[all]{xy}
\makeatletter
\long\def\nnfoottext#1{\insert\footins{\footnotesize
    \interlinepenalty\interfootnotelinepenalty
    \splittopskip\footnotesep
    \splitmaxdepth \dp\strutbox \floatingpenalty \@MM
    \hsize\columnwidth \@parboxrestore
   \edef\@thefnmark{}
   \edef\@currentlabel{}\@makefntext
    {\rule{\z@}{\footnotesep}\ignorespaces
      #1\strut}}}
\def\multirowsetup{\raggedright} 
\def\multirow#1{\relax\@ifnextchar
  [{\@multirow{#1}}{\@multirow{#1}[0]}}
\def\@multirow#1[#2]#3{\@ifnextchar [{\@xmultirow{#1}[#2]{#3}}%
  {\@xmultirow{#1}[#2]{#3}[0pt]}}
\def\@xmultirow#1[#2]#3[#4]#5{\@tempcnta=#1%
  \@tempdima\@tempcnta\ht\@arstrutbox
  \advance\@tempdima\@tempcnta\dp\@arstrutbox
  \ifnum\@tempcnta<0\@tempdima=-\@tempdima\fi
  \advance\@tempdima#2\bigstrutjot
  \if*#3\setbox0\vtop to \@tempdima{\vfill\multirowsetup
    \hbox{\strut#5\strut}\vfill}%
  \else
      \setbox0\vtop to \@tempdima{\hsize#3\@parboxrestore
                \vfill \multirowsetup \strut#5\strut\par\vfill}%
  \fi
  \ht0\z@\dp0\z@
  \ifnum\@tempcnta<0\advance\@tempdima-\dp\@arstrutbox
  \else\@tempdima=\ht\@arstrutbox
    \ifnum#2>0 \advance\@tempdima\bigstrutjot \fi
  \fi
  \advance\@tempdima#4\relax\leavevmode\vtop{\vskip-\@tempdima\box0\vss}}
\@ifundefined{bigstrutjot}{\newdimen\bigstrutjot \bigstrutjot\jot}{}
\makeatother
\usepackage{hyperref}

\numberwithin{equation}{section}

\newtheorem{thm}{Theorem}[subsection]

\newtheorem{prop}[thm]{Proposition}
\newtheorem{lm}[thm]{Lemma}
\newtheorem{cor}[thm]{Corollary}
\newtheorem{claim}[thm]{Claim}

\theoremstyle{definition}
\newtheorem{defi}[thm]{Definition}

\newtheorem*{cons}{Consequence}

\newtheorem{Rq}[thm]{Remark}
\newtheorem{Rqs}[thm]{Remarks}

\theoremstyle{remark}
\newtheorem*{remark}{Remark}

\newtheorem*{exemple}{Example}

\newtheorem*{mercis}{Acknowledgments}

\theoremstyle{remark}

\theoremstyle{plain}

\def\ad{\operatorname {ad}}

\DeclareMathOperator{\im}{im}

\DeclareMathOperator{\rk}{rk}
\def\Hom{\operatorname {Hom}}
\def\Aut{\operatorname {Aut}}

\def\GL{\operatorname {GL}}

\def\SO{\operatorname {SO}}
\def\Ort{\operatorname {O}}
\def\Sp{\operatorname {Sp}}
\DeclareMathOperator{\spec}{\mathsf{sp}}
\def\Lie{\operatorname {Lie}}

\DeclareMathAlphabet{\calptmx}{OMS}{ztmcm}{m}{n}

\newcommand{\K}{{\Bbbk}}
\newcommand{\Z}{\mathbb{Z}}
\newcommand{\N}{\mathbb{N}}
\newcommand{\LL}{\mathbb{L}}

\newcommand{\Q}{\mathbb{Q}}

\newcommand{\g}{\mathfrak{g}}
\newcommand{\h}{\mathfrak{h}}

\newcommand{\kk}{\mathfrak{k}}
\newcommand{\af}{\mathfrak{a}}
\newcommand{\df}{\mathfrak{d}}
\newcommand{\ff}{\mathfrak{f}}
\newcommand{\pp}{\mathfrak{p}}
\newcommand{\qq}{\mathfrak{q}}

\newcommand{\bb}{\mathfrak{b}}
\newcommand{\lf}{\mathfrak{l}}

\newcommand{\rf}{\mathfrak{r}}
\newcommand{\nf}{\mathfrak{n}}
\newcommand{\cc}{\mathfrak{c}}

\newcommand{\wfr}{\mathfrak{w}}
\newcommand{\tf}{\mathfrak{t}}

\newcommand{\zz}{\mathfrak{z}}

\newcommand{\gl}{\mathfrak{gl}}
\newcommand{\sld}{\mathfrak{sl}_2}
\newcommand{\fsl}{\mathfrak{sl}}
\newcommand{\sln}{\mathfrak{sl}}
\newcommand{\so}{\mathfrak{so}}
\newcommand{\spn}{\mathfrak{sp}}

\newcommand{\Sf}{\mathfrak{S}}

\newcommand{\Striplet}{$\mathfrak{sl}_{2}$-triple\xspace}

\newcommand{\Od}{\mathcal{O}}

\newcommand{\sS}{\mathcal S}
\newcommand{\pr}{\mathrm{pr}}
\newcommand{\Id}{\mathrm{Id}}

\newcommand{\lnq}{<}
\newcommand{\gnq}{>}

\newcommand{\Kt}{G^{\theta}}
\newcommand{\bolda}{\boldsymbol{\lambda}}
\newcommand{\cpps}{\cc_{\pp}(\pp^s)^{\bullet}}
\newcommand{\cggs}{\cc_{\g}(\g^s)^{\bullet}}
\newcommand{\cpp}[1]{\cc_{\pp}(\pp^{#1})^{\bullet}}

\newcommand{\cppnb}[1]{\cc_{\pp}(\pp^{#1})}
\newcommand{\cggnb}[1]{\cc_{\g}(\g^{#1})}
\newcommand{\cppsnb}{\cc_{\pp}(\pp^s)}
\newcommand{\cggsnb}{\cc_{\g}(\g^s)}
\newcommand{\XP}{X_{\pp}}
\newcommand{\nonv}{\neq\emptyset}

\newcommand{\Sppm}{\spec_{\pm}}

\newcommand{\qmod}{{/\!\!/}}
\newcommand{\sminus}{\smallsetminus}
\newcommand{\nempty}{{\neq \emptyset}}
\newcommand{\ttheta}{\tilde{\theta}}

\newcommand{\rtt}{\mathtt{r}}

\newcommand{\sscolumn}[2]{\mbox{$\begin{smallmatrix} {#1} \\
{#2} \end{smallmatrix}$}}

\newcommand{\imply}{\Rightarrow}

\def\preisomto{\vbox{\hbox to
               14pt{\hfill$\sim$\hfill}\nointerlineskip\vskip -0.2pt
               \hbox to 14pt{\rightarrowfill}}}
\def\isomto{\mathop{\preisomto}}

\def\prelongisomto{\vbox{\hbox to
                17pt{\hfill$\sim$\hfill}\nointerlineskip\vskip -0.2pt
                \hbox to 17pt{\rightarrowfill}}}
\def\longisomto{\mathop{\prelongisomto}}

\begin{document}


\title{Sheets of Symmetric Lie Algebras and Slodowy Slices}
\author{{\sc Micha\"el~Bulois}
\thanks{{\url{Michael.Bulois@univ-brest.fr}}}}
\date{}
\maketitle


\nnfoottext{Universit\'e de Brest, CNRS, UMR 6205, Laboratoire de Math\'ematiques de Brest, 6 avenue Victor Le Gorgeu, CS 93837, F-29238 BREST Cedex 3}
\nnfoottext{Universit\'e Europ\'eenne de Bretagne, France}


\vspace{1cm}

\begin{abstract}
  Let $\theta$ be an involution of the finite dimmensional
  reductive Lie algebra $\g$ and $\g=\kk\oplus\pp$ be the
  associated Cartan decomposition.  Denote by $K\subset G$ the connected
  subgroup having $\kk$ as Lie algebra. The $K$-module $\pp$ is the union
  of the subsets $\pp^{(m)}:=\{x \mid \dim K.x =m\}$, $m \in
  \N$, and the $K$-sheets of $(\g,\theta)$ are the
  irreducible components of the $\pp^{(m)}$.  The sheets can
  be, in turn, written as a union of so-called Jordan
  $K$-classes. We introduce conditions in order to describe
  the sheets and Jordan classes in terms of Slodowy
  slices. When $\g$ is of classical type, the $K$-sheets are
  shown to be smooth; if $\g=\gl_N$ a complete description
  of sheets and Jordan classes is then obtained.
\end{abstract}

\vspace{1cm}

\section*{Introduction\markright{INTRODUCTION}}
\addcontentsline{toc}{section}{\protect\numberline{}Introduction}

Let $\g$ be a finite dimensional reductive Lie algebra over
an algebraically closed field $\K$ of characteristic zero.
Fix an involutive automorphism $\theta$ of $\g$; it yields
an eigenspace decomposition $\g=\kk\oplus\pp$ associated to
respective eigenvalues $+1$ and $-1$. One then says that
$(\g,\theta)$, or $(\g,\kk)$, is a symmetric Lie algebra, or
a symmetric pair.  Denote by $G$ the adjoint group of $\g$
and by $K\subset G$ the connected subgroup with Lie algebra
$\kk\cap[\g,\g]$. The adjoint action of $g \in G$ on $x \in
\g$ is denoted by $g.x$. Recall that a \emph{$G$-sheet} of
$\g$ is an irreducible component of $\g^{(m)}:=\{x\in\g\mid
\dim G.x=m\}$ for some $m\in\N$.  This notion can be
obviously generalized to $(\g,\theta)$: the
\emph{$K$-sheets} of $\pp$ are the irreducible components of
the $\pp^{(m)}:=\{x\in\pp\mid \dim K.x=m\}$, $m \in \N$.
The study of these varieties is related to various geometric
problems occuring in Lie theory.  For example, the study of
the irreducibility of the commuting variety in \cite{Ri1}
and of its symmetric analogue in \cite{Pa4, SY2, PY} is
based on some results about $G$-sheets and $K$-sheets.

\smallskip

Let us first recall some results about $G$-sheets.  The
$G$-sheets containing a semisimple element are called
Dixmier sheets; they were introduced by Dixmier in
\cite{Dix1, Dix2}.  Any $G$-sheet is Dixmier when
$\g=\gl_{N}$; in \cite{Kr}, Kraft gave a parametrization of
conjugacy classes of sheets.  Borho and Kraft
introduced in \cite{BK} the notion of a sheet for an arbitrary
representation, which includes the above definitions of
$G$-sheets and $K$-sheets.  They also generalized in
\cite{Boh,BK} some of the results of \cite{Kr} to any
semisimple $\g$. In particular, they give a parametrization of $G$-sheets which relies on the
\emph{induction of nilpotent orbits}, defined by
Lusztig-Spaltenstein~\cite{LS}, and the notion of
 \emph{decomposition
  classes} or \emph{Zerlegungsklassen}.
Following~\cite[39.1]{TY}, a decomposition class will be called a \emph{Jordan $G$-class} here.  
The Jordan $G$-class
of an element $x\in\g$ can be defined by
\[
J_{G}(x):=\{y \in \g \, \mid \, \exists \, g\in G, \
g.\g^x=\g^y\}
\]
(where $\g^x$ is the centralizer of $x$ in $\g$).  Clearly,
Jordan $G$-classes are equivalence classes and one can show
that $\g$ is a finite disjoint union of these classes.
Then, it is easily seen that a $G$-sheet is the union of
Jordan $G$-classes.  A significant part of the work made in
\cite{Boh,BK} consists in characterizing a $G$-sheet by the
Jordan $G$-classes it contains.  Basic results on Jordan
classes (finiteness, smoothness, description of
closures,\dots) can be found in \cite[Chapter 39]{TY} and
one can refer to Broer \cite{Bro} for more advanced properties
(geometric quotients, normalisation of closure,\dots).

An important example of a $G$-sheet is the \emph{set of
  regular elements}:
$$
\g^{\mathit{reg}} := \{x\in \g \, \mid \, \dim\g^x \leqslant
\dim\g^y \ \text{for all $y\in\g$}\}.
$$ 
Kostant \cite{Ko} has shown that the geometric quotient
$\g^{\mathit{reg}}/G$ exists and is isomorphic to an affine
space.  This has been generalized to the so-called
\emph{admissible} $G$-sheets in \cite{Ru}.  Then, Katsylo
proved in \cite{Kat} the existence of a \emph{geometric
  quotient} $S/G$ for any $G$-sheet $S$.  More recently,
Im Hof \cite{IH} showed that the $G$-sheets are \emph{smooth} when
$\g$ is of classical type.

The parametrization of sheets used in \cite{Ko,Ru,Kat,IH}
differs from the one given in \cite{Kr, Boh,BK} by the use
of ``Slodowy slices''.  More precisely, let $S$ be a sheet
containing the nilpotent element $e$ and embed $e$ into an
\Striplet $(e,h,f)$. Following the work of Slodowy
\cite[\S7.4]{Sl}, the associated \emph{Slodowy slice} $e+X$
of $S$ is defined by
$$
e+X:=(e+\g^f)\cap S.
$$
Then, one has $S=G.(e+X)$ and $S/G$ is isomorphic to the
quotient of $e+X$ by a finite group \cite{Kat}.
Furthermore, since the morphism $G\times (e+X) \rightarrow
S$ is smooth \cite{IH}, the geometry of $S$ is closely
related to that of $e+X$.
We give a more detailed presentation of these results in the
first section.

\smallskip

In the symmetric case, much less properties of sheets are
known. The first important one was obtained in \cite{KR}
where the \emph{regular sheet} $\pp^{\mathit{reg}}$ of $\pp$
is studied. In particular, similarly to \cite{Ko}, it is
shown that $\pp^{\mathit{reg}}=\Kt.(e^{\mathit{reg}}+\pp^f)$
where $\Kt:=\{g\in G \mid g\circ\theta=\theta\circ g\}$.
Another interesting result is obtained in \cite{Pa4, SY2,
  PY} (where the symmetric commuting variety is studied):
each even nilpotent element of $\pp$ belongs to some
$K$-sheet containing a semisimple element.  More advanced
results can be found in \cite[\S39]{TY}.
The \emph{Jordan $K$-class} of $x \in \pp$ is defined by
$$
J_{K}(x):=\{y \in \pp \,\mid \, \exists \, k\in K, \,
k.\pp^x=\pp^y\}.
$$
One can find in \cite{TY} some properties of Jordan
$K$-classes (finiteness, dimension, \dots) and it is shown
that a $K$-sheet is a finite disjoint union of such classes.

Unfortunately, the key notion of ``orbit induction'' does
not seem to be well adapted to the symmetric case.  For
instance, the definition introduced by Ohta in \cite{Oh3} does not
leave invariant the orbit dimension anymore.

\smallskip

We now turn to the results of this paper.  The inclusion
$\pp^{(m)}\subset \g^{(2m)}$ is the starting point for
studying the intersection of $G$-sheets, or Jordan classes,
with $\pp$ in order to get some information about
$K$-sheets.

We first consider the case of symmetric pairs {of type}~0 in
section~\ref{type0}.  A symmetric pair is said to be of
type~0 if it is isomorphic to a pair $(\g'\times\g',\theta)$
with $\theta(x,y)=(y,x)$. This case, often called the
``group case'', is the symmetric analogue of the Lie algebra
$\g'$.  

In the general case we study the intersection $J\cap\pp$
when $J$ is a Jordan $G$-class.  Using the results obtained
in sections~\ref{basis} to~\ref{JKclass}, we show (see
Theorem~\ref{compirr}) that $J\cap\pp$ is smooth,
equidimensional, and that its irreducible components are
exactly the Jordan $K$-classes it contains.

We study the $K$-sheets, for a general symmetric pair, in
section~\ref{Ksheet}. After proving the smoothness of
$K$-sheets in classical cases (Remark~\ref{slsheets}), we
try to obtain a parametrization similar to the Lie algebra
case by using generalized ``Slodowy slices'' of the form
$e+X\cap\pp$, where $e \in \pp$ is a nilpotent element
contained in the $G$-sheet $S$.  To get this parametrization
we need to introduce three conditions (labelled by
\eqref{heart}, \eqref{diamond} and \eqref{club}) on the
sheet $S$.  Under these assumptions, we obtain the
parametrization result in Theorem~\ref{equidim}; it gives in
particular the equidimensionality of $S\cap\pp$.

In the third section we show that the conditions
\eqref{heart}, \eqref{diamond}, \eqref{club} hold when
$\g=\gl_{N}$ or $\mathfrak{sl}_N$ (type~A). In this case, up
to conjugacy, three types of irreducible symmetric pairs
exist (AI, AII, AIII in the notation of \cite{He1}) and have
to be analyzed in details. The most difficult one being
type~AIII, i.e.~$(\g,\kk) \cong (\gl_N, \gl_p \times
\gl_{N-p})$.

In Section~4 we prove the main result in type~A
(Theorem~\ref{final}), which gives a complete description of
the $K$-sheets and of the intersections of $G$-sheets with
$\pp$.
In particular, we give the dimension of a $K$-sheet in terms
of the dimension of the nilpotent $K$-orbits contained in
the sheet.  One can also determine the sheets which contain
semisimple elements ({i.e.}~the \emph{Dixmier $K$-sheets})
and characterize nilpotent orbits which are $K$-sheets
({i.e.}~the \emph{rigid nilpotent $K$-orbits}).

\begin{mercis}
  I would like to thank {Micha\"el Le Barbier}, {Oksana Yakimova}  and {Anne Moreau}
   for useful conversations. 
  I also thank the referees of my thesis {Dmitri Panyushev} and {Michel Brion} for
  their valuable comments which helped to improve significantly the quality of this article.
  I am grateful to {Michel Brion} (and {Thierry Levasseur}) for
  pointing out the relevance of Theorem~\ref{Iv} to the
  situation.
\end{mercis}

\section{Generalities}
\label{general}

\subsection{Notation and basics}
\label{somenotation}

We fix an algebraically closed field $\K$ of characteristic
zero and we set $\K^{\times} := \K \smallsetminus\{0\}$. If
$V, V'$ are $\K$-vector spaces, $\Hom(V,V')$ is the vector
space of $\K$-linear maps from $V$ to $V'$ and the dual of
$V$ is $V^*:=\Hom(V,\K)$. The space $\gl(V):=\Hom(V,V)$
is equipped with a natural Lie algebra structure by setting
$[x,y]=x\circ y - y\circ x$ for $x,y\in \gl(V)$.  The action
of $x \in \gl(V)$ on $v \in V$ is written $x.v= x(v)$ and
${}^tx$ is the transpose linear map of $x$.  If $M$ is a
subset of $\Hom(V,V')$ we set $\ker M:=\bigcap_{\alpha\in M}
\ker \alpha$.
 
If $\mathbf{v} =(v_{1},\dots,v_{N})$ is a basis of $V$, the
algebra $\gl(V)$ can be identified with $\gl(\mathbf{v})
:=\gl_N=\mathrm{M}_{N}(\K)$ (the algebra of $N\times N$
matrices). When $\mathbf{v}'=(v_{i_{1}},\dots,v_{i_{k}})$ is
a sub-basis of $\mathbf{v}$,
we may identify $\gl(\mathbf{v}')$ with a subalgebra of
$\gl(V)$ by extending $x\in\gl(\mathbf{v'})$ as follows:
$x.v_{i}:=x.v_{i_{j}}$ if $i=i_{j}$ for some $j \in
[\![1,k]\!]$, $x.v_{i}:=0$ otherwise.

All the varieties considered will be algebraic over $\K$ and
we (mostly) adopt notations and conventions of \cite{Ha} or
\cite{TY} for relevant algebraic and topological notions.
In particular, $\K[X]$ is the ring of globally defined
algebraic functions on an algebraic variety $X$. Recall that
when $V$ is a finite dimensional vector space one has
$\K[V]=S(V^*)$, the symmetric algebra of $V^*$.

We will refer to \cite{TY} for most of the classical results
concerning Lie algebras.  As said in the introduction, $\g$
denotes a finite dimensional \emph{reductive} Lie
$\K$-algebra.  We write $\g=[\g,\g] \oplus \zz(\g)$ where
$\zz(\g)$ is the centre of $\g$ and we denote by $\ad_\g(x)
: y \mapsto [x,y]$ the adjoint action of $x \in \g$ on $\g$. 
Let $G$ be the connected algebraic subgroup of
$\mathrm{GL}(\g)$ with Lie algebra $\Lie G =
\mathrm{ad}_{\g}(\g) \cong [\g,\g]$.  The group $G$ is
called the \emph{adjoint group} of $\g$.  The adjoint action
of $g \in G$ on $y \in \g$ is denoted by $g.y =
\mathrm{Ad}(g).y$; thus, $G.y$ is the (adjoint) orbit of
$y$.
 
We will generally denote Lie subalgebras of $\g$ by small
german letters (e.g.~$\lf$) and the smallest algebraic
subgroup of $G$ whose Lie algebra contains
$\mathrm{ad}_{\g}(\lf)$ by the corresponding capital roman
letter (e.g.~$L$). When $\lf$ is an algebraic subalgebra of
$\g$ the subgroup $L$ acts on $\lf$ as its adjoint algebraic
group, cf.~\cite[24.8.5]{TY}. We denote by $H^\circ$ the
identity component of an algebraic group $H$.

Let $E \subset \g$ be an arbitray subset. If $\lf$,
resp.~$L$, is a subalgebra of $\g$, resp.~algebraic subgroup
of $G$, we define the associated centralizers and
normalizers by:
\begin{gather*}
  \lf^{E}=\cc_{\lf}(E) := \{x \in \lf \mid [x,E]= (0) \},
  \quad L^{E} = Z_{L}(E)= C_L(E) := \{g \in L \mid g.x= x \
  \text{for all $x \in E$}\},
  \\
  N_{L}(E) := \{g\in L \mid g.E \subset E \}.
\end{gather*}
When $E=\{x\}$ we simply write $\lf^x$, $L^x$, etc. Recall
from \cite[24.3.6]{TY} that $\Lie L^{E}=\lf^{E}$.  As in
\cite{TY}, the set of ``regular'' elements in $E$ is denoted
by:
\begin{equation} \label{Ebullet} E^{\bullet} := \bigl\{x \in
  E : \dim \g^x = \textstyle{\min_{y \in E}} \dim \g^y
  \bigr\} = \bigl\{x \in E : \dim G.x = \max_{y \in E}\dim
  G.y \bigr\}.
\end{equation}

Any $x\in\g$ has a \emph{Jordan decomposition} in $\g$, that
we will very often write $x=s+n$ (cf.~\cite[20.4.5,
20.5.9]{TY}). Thus $s$ is semisimple, i.e.~$\ad_\g(s) \in
\gl(\g)$ is semisimple, $n$ is nilpotent, i.e.~$\ad_\g(n)$
is nilpotent, and $[s,n] =0$. The element $s$, resp.~$n$, is
called the semisimple, resp.~nilpotent, part (or component)
of~$x$. An \emph{\Striplet} is a triple $(e,h,f)$ of
elements of $\g$ satisfying the relations
$$[h,e]=2e, \quad [h,f]=-2f, \quad [e,f]=h.$$


Let $\h$ be a Cartan subalgebra of $\g$; then, $\h =
([\g,\g] \cap \h) \oplus \zz(\g)$ and the \emph{rank} of
$\g$ is $\rk \g := \dim \h$.  We denote by $R=R(\g,\h) =
R([\g,\g], [\g,\g] \cap \h) \subset \h^*$ the associated
\emph{root system}. Recall that the \emph{Weyl group}
$W=W(\g,\h)$ of $R$ can be naturally identified with
$N_{G}(\h)/Z_{G}(\h)\subset \GL(\h)$ (see, for example,
\cite[30.6.5]{TY}).  The type of the root system $R$, as
well as the type of the reflection group $W$, will be
indicated by capital roman letters, frequently indexed by
the rank of $[\g,\g]$, e.g.~E$_{8}$. If $\alpha\in
R(\g,\h)$, $\g^{\alpha} :=\{x\in \g \mid [h,x] = \alpha(h)x
\; \, \text{for all $h\in \h$} \}$ is the \emph{root
  subspace} associated to $\alpha$.  If $M$ is a subset of
$R(\g,\h)$, we denote by $\langle M \rangle$ the root
subsystem $\bigl(\sum_{\alpha\in M} \Q\alpha\bigr)\cap
R(\g,\h)$.

We use the notation $\lfloor \phantom{\mu} \rfloor$,
resp.~$\lceil \phantom{\mu} \rceil$, for the floor,
resp.~ceiling, function on $\Q$; thus $\lfloor \lambda
\rfloor$, resp.~$\lceil \lambda \rceil$, is the largest,
resp.~smallest, integer $\le \lambda$, resp.~$\ge \lambda$.
If $i,j$ are two integers, the set $[\![i,j]\!]$ stands for $\{k\in\Z \mid i\leqslant k\leqslant j\}$.

\label{reductions}

\medskip

Let $\g=\prod_i \g_i = \bigoplus_i \g_{i}$ be a
decomposition of $\g$ as a direct sum of reductive Lie
(sub)algebras.  Let $G_{i}$ be the adjoint group of
$\g_{i}$, thus $G=\prod_{i} G_{i}$. 
Under these notations, it is not difficult to prove the following lemma:

\begin{lm}\label{pppmmm}
  The $G$-sheets of $\g$ are of the form $\prod_{i} S_{i}$
  where each $S_{i}$ is a $G_{i}$-sheet of $\g_{i}$.
\end{lm}


Recall that, since $\g$ is reductive, there exists a
decomposition $\g=\zz \times \prod_{i}\g_{i}$ where $\zz$ is
the centre of $\g$ and $\g_{i}$ is a simple Lie algebra for
all $i$. So lemma~\ref{pppmmm} provides the following.

\begin{cor}
  \label{sheetdecomposition}
  \label{simpleprod}
  The $G$-sheets of $\g$ are the sets of the form $\zz
  \times \prod_{i} S_{i}$ where each $S_{i}$ is a
  $G_{i}$-sheet of~$\g_{i}$.
\end{cor}


The previous corollary allows us to restrict to the case
when $\g$ is simple.  Furthermore, it shows that the study
of sheets of $\g$ and of $[\g,\g]$ are obviously related by
adding the centre.  Therefore, we may for instance work with
$\g=\gl_{n}$ to study the $\fsl_n$-case.

\subsection{Levi factors}
\label{Levi}
We start by recalling the definition of Levi factors:
 
\begin{defi} \label{deflevifactor} A \emph{Levi factor} of
  $\g$ is a subalgebra of the form $\lf=\g^s$ where $s \in
  \g$ is semisimple.  The connected algebraic subgroup
  $L\subset G$ associated to a Levi factor
  $\lf$ is called a \emph{Levi factor} of $G$.
\end{defi}

Observe that the previous definition of a Levi factor of
$\g$ is equivalent to the definition given
in~\cite[29.5.6]{TY}, see, for example, \cite[Exercice~10,
p.~223]{Bou}.  Recall that a Levi factor $\lf = \g^s$ is
reductive \cite[20.5.13]{TY} and $L=G^s$,
cf.~\cite[Corollary~3.11]{St} and \cite[24.3.6]{TY}.

Let $\h$ be a Cartan subalgebra and $\lf$ be a Levi factor
containing $\h$. By \cite[20.8.6]{TY} there exists a subset
$M=M_{\lf}\subset R(\g,\h)$ such that $M=\langle M\rangle$
and
\begin{gather} \label{lM} \lf = \lf_M := \h \oplus
  \bigoplus_{\alpha\in M} \g^{\alpha}
  \\
  \cc_{\g}(\lf) =\zz(\lf)=\{t\in\h \mid \alpha(t)=0 \; \,
  \text{for all $\alpha \in M$}\} \ \text{and} \
  \cc_{\g}(\cc_{\g}(\lf))=\lf.
\end{gather}
Conversely, if $M\subset R(\g,\h)$ is a subset such that $M=
\langle M\rangle$, define $\lf=\lf_{M}$ as in
equation~\eqref{lM}; then $\lf_{M}$ is a Levi factor and:
\begin{equation}\label{centerlevi}
  \h \supseteq \{s\in \g \mid \lf_{M}=\g^s \}
  = \ker M \setminus \left( \,\bigcup_{\alpha\notin
      M}\ker\alpha\right) \neq \emptyset.
\end{equation}
This construction gives a bijective correspondence
$\lf=\lf_M \leftrightarrow M=M_\lf$ between Levi factors containing $\h$ and
subsets of $R(\g,\h)$ satisfying the above property. 
Then the action of the Weyl group $W = W(\g,\h)$ on $R(\g,\h)$ induces an action
on the set of Levi factors containing $\h$.  
In other words, if $g\in N_{G}(\h)$ and $\lf$ is a Levi factor containing $\h$, one has
$g.\lf=(gZ_G(\h)).\lf$ and if $w\in W$ is the class of $\g$, we define $w.\lf:=g.\lf$.  
Let $x,y \in \h$; we will say that the Levi factors
$\g^x, \g^y$ are $W$-conjugate if there exists $w\in W$ such
that $w.M_{\g^x}=M_{\g^y}$. From~\eqref{centerlevi} one
deduces that this definition is equivalent to
$w.\cggnb{x}=\cggnb{y}$ for some $w \in W$.

Assume that $\g$ is semisimple and denote by $\kappa$ the
isomorphism $\h \isomto \h^*$ induced by the restriction of
the Killing form of $\g$. Define a $\Q$-form of $\h$, or
$\h^*$, by $\h_{\Q} \stackrel{\kappa}{\cong}\h_{\Q}^*
:=\Q. R(\g,\h)$.  Fix the Cartan subalgebra $\h$ and a
fundamental system ({i.e.}~a basis) $B$ of $R(\g,\h)$.  We
say that a Levi factor $\lf$ is \emph{standard} if $\lf =
\g^s$ with $s\in\h_{\Q}$ in the positive Weyl chamber of
associated to $B$.  In this case, one can write
$M_{\lf}=\langle I_{\lf}\rangle =\Z I_\lf \cap R(\g,\h)$
where $I_{\lf}\subset B$.  The following proposition is
consequence of the definition of a Levi factor and
\eqref{centerlevi}.

\begin{prop}
  Any Levi factor of $\g$ is $G$-conjugate to a standard
  Levi factor.
\end{prop}

Let $\lf \subset \g$ be a Levi factor and $L$ be the
associated Levi factor of $G$. There exists a unique
decomposition $\lf=\zz(\lf)\oplus \bigoplus_{i} \lf_{i}$,
where $\zz(\lf)$ is the centre and the $\lf_i$ are simple
subalgebras.  Let $L_{i} \subset G$ be the connected
subgroup with Lie algebra $\lf_{i}$
(cf.~\cite[24.7.2]{TY}). Under this notation we have:

\begin{prop} \label{Levisimple} The subgroup $L\subset G$ is
  generated by $C_{G}(\lf)$ and the subgroups $L_{i}$.
\end{prop}

\begin{proof}
  Recall that $\Lie L_{i}=\lf_{i}$ and $\Lie
  Z_{G}(\lf)=\zz(\lf)$. By \cite[24.5.9]{TY} one gets that
  $L$ is generated by the connected subgroups $L_{i}$ and
  $C_{G}(\lf)^{\circ}$. Writing $\lf=\g^s$ with $s$
  semisimple, we have already observed that $L=G^s$, hence
  $C_{G}(\g^s)\subset G^s$ and the result follows.
\end{proof}

\subsection{Jordan $G$-classes}
\label{Gclass}

The description of $G$-sheets is closely related to the
study of Jordan $G$-classes, also called decomposition
classes. We now recall some facts about these classes (see,
for example, \cite{BK, Boh, Bro, TY}).

Recall from~\S\ref{somenotation} that any element $x \in \g$
has a unique Jordan decomposition $x=s+n$. We then say that
the pair $(\g^s,n)$ is the \emph{datum} of $x$.

\begin{defi}
  Let $x=s+n$ be the Jordan decomposition of $x\in\g$. The
  Jordan $G$-class of $x$, or $J_{G}$-class of $x$, is the
  set $J_{G}(x):=G.(\cggs+n)$.  Two elements are Jordan
  $G$-equivalent if they have the same $J_G$-class.
\end{defi}

Let $L$ be a Levi factor of $G$ with Lie algebra $\lf$, and
$L.n \subset \lf$ be a nilpotent orbit.  If $J$ is a
$J_{G}$-class, the pair $(\lf, L.n)$, or $(\lf, n)$, is
called a \emph{datum of $J$} if $(\lf,n)$ is the datum of an
element $x\in J$. Setting $\tf:=\g^{\lf}$ it is then easy to
see that $J=G.(\tf^{\bullet}+n)$.  From this result one can
deduce that Jordan $G$-classes are locally closed
\cite[39.1.7]{TY}, and smooth \cite{Bro}.  Furthermore, two
elements of $\g$ are Jordan $G$-equivalent if and only if
their data are conjugate under the diagonal action of $G$
\cite[39.1]{TY}.  Then, $\g$ is the finite disjoint union of
its Jordan $G$-classes (cf.~\cite[39.1.8]{TY}). 
The following result is taken from \cite{BK} (see also \cite[39.3.4]{TY}).

\begin{prop} \label{sheetjordan} A $G$-sheet of $\g$ is a
  finite (disjoint) union of Jordan $G$-classes.
\end{prop}

An immediate consequence of this proposition is that each
$G$-sheet $S$ contains a unique dense (open) Jordan
$G$-class $J$. It follows that we can define a \emph{datum
  of $S$} to be any datum $(\lf, L.n)$, or $(\lf,n)$, of
this dense class $J$.  For instance, if $S$ is a $G$-sheet
containing a semisimple element, i.e.~$S$ is a \emph{Dixmier
  sheet}, then $J$ is the class of semisimple elements of
$S$ and $(\lf,0)$ is a datum of $S$, see~\cite[39.4.5]{TY}.

\subsection{Slodowy slices}
\label{Katsylo}
We recall in this subsection some of the important results
obtained by Katsylo \cite{Kat}.  One of the first
fundamental properties of the sheets in $\g$ was obtained
by Borho-Kraft~\cite[Korollar 5.8]{BK}
(cf.~also~\cite[39.3.5]{TY}):

\begin{prop}
  \label{unilp}
  Each $G$-sheet contains a unique nilpotent orbit.
\end{prop}

Fix a $G$-sheet $S_G$, a datum $(\lf, L.n)$ of $S_G$,
cf.~\ref{Gclass}, and a Cartan subalgebra $\h\subset
\lf$. Set $\tf:=\g^{\lf}$ (thus $\tf\subset\h$).  Then,
following \cite{BK}, one can construct a parabolic
subalgebra $\mathfrak{j}$ of $\g$ and a nilpotent ideal
$\nf$ of $\mathfrak{j}$ such that $\rf=\nf\oplus\tf$
satisfies $S_G=G.\rf^{\bullet}$ (and
$\overline{S_G}=G.\rf$).  This is done as follows.  Recall,
see for example \cite[\S 5.7]{Ca}, that there exists a
grading $\lf = \bigoplus_{i \in \Z} \lf_i$ such that
$\mathfrak{j}_2 := \bigoplus_{i \ge 0} \lf_i$ is a parabolic
subalgebra of $\lf$, $\nf_2 := \bigoplus_{i \ge 2} \lf_i$ is
a nilpotent ideal of $\mathfrak{j}_2$ such that
$[\mathfrak{j}_2,n]=\nf_2$.  If $\nf_1$ is the nilradical of
any parabolic subalgebra with $\lf$ as Levi factor, one then
takes $\mathfrak{j}:= \mathfrak{j}_2+\nf_1$ and
$\nf:=\nf_1+\nf_2$.

Note here that when $S_G$ is \emph{Dixmier}, i.e.
contains semisimple elements, then $n=0$ and $\mathfrak{j} =
\lf + \nf$ has~$\lf$ as Levi factor and $\nf$ as nilradical.
This will be the case when $S_G$ is regular in section
\ref{regularsheet} or when $\g$ is of type A in~\ref{glN}.

Under the previous notation, the following result is proved
in~\cite[Lemma~3.2]{Kat} (cf.~also \cite[Proposition
2.6]{IH}).

\begin{prop}
  \label{axiomK}
  Let $(e,h,f)$ be an \Striplet such that $e
  \in\nf^{\bullet}$ and $h \in \h$, then
  \[
  S_G=G.(e+\tf).
  \]
\end{prop}

From \cite[Lemma~3.1]{Kat} one knows that there exists an
\Striplet $\sS:=(e,h,f)$ such that $e \in \nf^{\bullet}$ and
$h \in \h$.  We fix $\sS=(e,h,f)$ for the rest of the
subsection. Note that $e\in S_G$. The adjoint action of $h$
on $\g$ yields a grading
\[
\g=\bigoplus_{i\in\Z} \g(i,h), \quad \g(i,h) := \{v \in \g :
[h,v]= iv\}.
\]
One of the main constructions in \cite{Kat} consists in
deforming the ``section'' $e+\tf$ into another ``section''
having nice properties. The construction goes as
follows. First, define a subset $e+X(S_G,\sS) \subset S_G$,
depending only on the sheet and the choice of the \Striplet,
by:
$$
e+X(S_G,\sS) := S_G\cap(e+\g^f).
$$
Then, the deformation is made by using a map
$\varepsilon_{S_{G},\sS}^{\g}: e+\tf\rightarrow
e+X(S_G,\sS)$, whose definition is recalled below, see
Remark~\ref{ggprime}.  Before going into the details, note
that when there is no ambiguity on the context, we write $X$
instead of $X(S_G,\sS)$ and $\varepsilon^{\g}$, or
$\varepsilon$, instead of $\varepsilon^{\g}_{S_{G},\sS}$.

\begin{Rq}\label{XsXss}
  When $\g$ is of type~A, there is a unique sheet containing
  a fixed nilpotent orbit (cf.~\cite[\S2]{Kr}). In this case
  we can therefore set $X(\sS):=X(S_G,\sS)$ where $S_G$ is
 the sheet containing the nilpotent element $e$ of $\sS$.
\end{Rq}

Define a one parameter subgroup
$(F_{t})_{t\in\K^\times}\subset \GL(\g)$ by setting
$F_{t}.y:=t^{(i-2)}y$ for $y\in\g(i,h)$.  One can show as in \cite{Kat} that
$F_t.e=e$, $F_t.S_G =S_G$, $F_t.X=X$ and $\lim_{t \to
  0}F_{t}.y = e$ for all $y \in e+X$.
 
One can slightly modify \cite[Lemma~5.1]{Kat} to obtain the
following result:

\begin{lm}\label{epsilon}
  There exists a polynomial map
  \[
  \epsilon 
  : e+\bigoplus_{i\leqslant 0}\g(2i,h)\longrightarrow e+(\g^f\cap\bigoplus_{i\leqslant0}\g(2i,h))
  \]
  such that:
\item[{\rm (i)}] $e+z\in G.\epsilon(e+z)$ for all
  $z\in\bigoplus_{i\leqslant 0}\g(2i,h)$;
\item[{\rm (ii)}] let $j\leqslant 0$ and set
  $P_{j}:=(\pi_{2j}\circ\epsilon)_{\mid e+\g(0,h)}$ where
  $\pi_{2j}$ is the canonical projection from
  $\bigoplus_{i\leqslant0}\g(2i,h)$ onto $\g(2j,h)$, then
  $P_j$ is either $0$ or a homogeneous polynomial of degree
  $-j+1$.
\end{lm}

\begin{proof}
  We set $\g_{i}:=\g(i,h)$ for $i\leqslant 1$. One can then
  define affine subspaces $L_{2i}$ and $M_{2i}$ by:
$$
L_{2i}:=\g^f\cap\g_{2i}, \quad
M_{2i}:=e+L_{2}+L_{0}+L_{-2}+\cdots+L_{2i}+\g_{2i-2}+
\g_{2i-4}+\cdots
$$
It is clear that $L_{2}=\{0\}$,
$M_{2}=e+\bigoplus_{i\leqslant0}\g_{2i}$ and
$M_{-2k}=e+(\g^f\cap\bigoplus_{i\leqslant0}\g(2i,h))$ for $k$ large enough. We fix such a $k$.
Now, define maps $\epsilon_{i}: M_{2i}\rightarrow M_{2i-2}$
as follows.
 
Denote the projections associated to the decomposition
$\g_{2i-2} = [e,\g_{2i-4}] \oplus L_{2i-2}$ by $\pr_{1}:
\g_{2i-2}\rightarrow [e,\g_{2i-4}]$ and $\pr_{2}:
\g_{2i-2}\rightarrow L_{2i-2}$ (hence
$\pr_{1}+\pr_{2}=\Id_{\g_{2i-2}}$). Next, define
$\eta_{2i-2}: \g_{2i-2}\rightarrow\g_{2i-4}$ to be the
linear map $(\ad e)^{-1}\circ\pr_{1}$. It satisfies
$[\eta_{2i-2}(x),e]+x\in L_{2i-2}$ for all $x\in\g_{2i-2}$.
If $e+z=e+\sum_{j=i}^0 z_{2j}+\sum_{j=k}^{i-1} w_{2j}\in
M_{2i}$, where $z_{2j}\in L_{2j}$, $w_{2j}\in\g_{2j}$, set:
$$
\epsilon_{i}(e+z):=\exp(\ad \eta_{2i-2}(w_{2i-2})) (e+z).
$$ 
Then, $\epsilon_{i}$ is a polynomial map such that
$\epsilon_{i}(e+z)\in M_{2i-2}$.  Now, set:
\[
\epsilon'_{i}:=\epsilon_{i}\circ\dots\circ\epsilon_{-1}\circ
\epsilon_{0}\circ\epsilon_{1}, \quad
\epsilon:=\epsilon'_{-k}.
\]
Clearly, $\epsilon$ is a polynomial map which satisfies (i).

To get (ii), we now show, by decreasing induction on
$i\leqslant 2$, that $P_j=(\pi_{2j}\circ\epsilon'_{i})_{\mid
  e+\g_{0}}$ is either $0$ or a homogeneous polynomial of
degree $-j+1$.  Set $\epsilon'_{2}:=\Id$ so that $P_1=0$ and the claim
is obviously true for $j=1$.  Assume that the assertion is true for a
given integer $i_{0}=i+1\leqslant 2$. Remark that the
construction of $\epsilon_i,\epsilon'_{i}$ gives
$$
\epsilon'_{i}(e+t)=\epsilon_{i}\circ\epsilon'_{i_{0}}(e+t)=\exp(
\ad\eta_{2i-2}(\pi_{2i-2}\circ\epsilon'_{i_{0}}(e+t))).
\epsilon'_{i_{0}}(e+t)
$$ 
for all $e+t\in e+\g_{0}$. By induction,
$u_{i}:=\eta_{2i-2}(\pi_{2i-2}\circ\epsilon'_{i_{0}}):e+\g_{0}
\rightarrow \g_{2i-4}$ is $0$ or homogeneous of degree
$-i+2$; thus
$$
\pi_{2j}\circ\epsilon'_{i}(e+t)=\sum_{l\geqslant0}
\frac{(\ad u_{i}(e+t))^l}{l!}\circ
\pi_{2j+l(-2i+4)}\circ\epsilon'_{i+1}(e+t)
$$ 
is either $0$ or homogeneous of degree
$l(-i+2)+(-j-l(-i+2)+1)=-j+1$, as desired.
\end{proof}

\begin{Rq}\label{ggprime}
  The polynomial map $\epsilon$ constructed in the proof of
  Lemma~\ref{epsilon} will be denoted by
  $\epsilon^\g=\epsilon^{\g}_{\sS}$.  By restriction, it
  induces a map $\varepsilon =
  \varepsilon^{\g}=\varepsilon^{\g}_{\sS}$ from $e+\h$ to
  $e+\g^f$ and Lemma~\ref{epsilon}(i) implies that
  $\epsilon$ maps $e+\tf$ into $e+X$.  One can therefore
  define $\varepsilon^{\g}_{S_{G},\sS}$ to be the polynomial
  map $(\varepsilon^{\g}_{\sS})_{\mid e+\tf}$.
   
  Furthermore, one may observe that the construction of
  $\epsilon^{\g}$ made in the proof of the previous
  proposition yields that $\varepsilon^{\g}$ does not depend
  on $\g$ in the following sense: if $\g'$ is a reductive
  Lie subalgebra of $\g$ containing $\sS$, then
  $\varepsilon^{\g'}=\varepsilon^{\g}_{\mid \h\cap\g'}$.  In
  the sequel, we will often write $\varepsilon$ when the
  subscript is obvious from the context.
\end{Rq}

The next lemma is due to Katsylo \cite{Kat}, see \cite{IH}
for a purely algebraic proof.

\begin{lm}
  \label{Gamma}
  Under the previous notation:
  \\
  {\rm (i)} $S_G= G.(e+X)$;
  \\
  {\rm (ii)} The group $(G^{e,h,f})^{\circ}$ acts trivially on $e+X$ so the action of $G$ on $\g$ induces an action of
  $A:=~G^{e}/(G^{e})^{\circ}\cong G^{e,h,f}/(G^{e,h,f})^{\circ}$ on $e+X$;
  \\
  {\rm (iii)} for all $x\in e+X$, one has $A.x=G.x\cap
  (e+X)$.
\end{lm}

These results enable us to define a quotient map (of sets)
by:
$$
\psi = \psi_{S_G,\sS}: S_{G}\longrightarrow (e+X)/A, \quad
\psi(x):=A.y \ \; \text{if $G.y=G.x$.}
$$
Since $e+X$ is an affine algebraic variety \cite[Lemma~4.1]{Kat} on which the
finite group $A$ acts rationally, it follows from
\cite[25.5.2]{TY} that $(e+X)/A$ can be endowed (in a
canonical way) with a structure of algebraic variety and
that the quotient map
\begin{equation}\label{mapgamma}
  \gamma:e+X \longrightarrow (e+X)/A 
\end{equation}
is the geometric quotient of $e+X$ under the action of $A$.
Using Lemma~\ref{epsilon}(i) and Lemma~\ref{Gamma} one
obtains:
\[
\psi = \gamma \circ \varepsilon \ \; \text{on $e + \tf$.}
\]
The following theorem is the main result in~\cite{Kat}:

\begin{thm}\label{psiquotient}
  The map $\psi: S_{G}\rightarrow (e+X)/A$ is a morphism of
  algebraic varieties and gives a geometric quotient $S_G/G$
  of the sheet $S_{G}$.
\end{thm}

\begin{Rq}
  \label{epsilonfinite}
  One has $\dim S_G/G = \dim X = \dim \tf$, see
  \cite[\S5]{Boh}.  It is shown in \cite[Corollary~4.6]{IH}
  that, when $\g$ is classical, the map $\varepsilon : e +
  \tf \to e +X$ is quasi-finite (it is actually finite by
  \cite[Chaps.~5~\&~6]{IH}).
\end{Rq}

The variety $e+X$ will be called a \emph{Slodowy slice} of
$S_{G}$. One of the main results of \cite{IH} is that $e+X$
is smooth when $\g$ is of classical type,
cf.~Theorem~\ref{slicesmoothness}.  This result relies on
some properties of $e+\g^f$ that we now recall
(see~\cite[7.4]{Sl}).

\begin{prop}\label{slodowy}
  {\rm (i)} The intersection of $G.x$ with $e+\g^f$ is
  transverse for any $x \in e+X$ (i.e.
  $T_x(e+\g^f)\oplus T_x(G.x)=T_x(\g)$.)
  \\
  {\rm (ii)} The morphism $\delta: G\times (e+\g^f)
  \rightarrow\g$, $(g,x)\mapsto g.x$, is smooth.
  \\
  {\rm (iii)} Let $Y$ be a $G$-stable subvariety of $\g$ and
  set $Z := Y\cap(e+\g^f)$. Then the restricted morphism
  $\delta':G\times Z\rightarrow Y$ is smooth.  In
  particular, when $Y=G.Z$, $Y$ is smooth if and only if $Z$
  is smooth.
\end{prop}

\begin{proof}
  Claims (i) and (ii) are essentially contained in
  \cite[7.4, Corollary~1]{Sl}.
  \\
  (iii) We merely repeat the argument given in \cite{IH}.
  Let
  $\hat{Z}=Y\cap_{\mathrm{sch}}(e+\g^f):=Y\times_{\g}(e+\g^f)$
  be the schematic intersection of $Y$ and $(e+\g^f)$
  (cf.~\cite[p.~87]{Ha}).  Writing $(G\times
  (e+\g^f))\times_{\g}Y\cong G\times((e+\g^f)\times_{\g} Y)
  = G\times \hat{Z}$, the base extension $Y \to \g$ gives
  the following diagram:
$$
\begin{array}{c c c} G\times (e+\g^f)&
  \stackrel{\delta}{\longrightarrow}&\g \\
  \uparrow & &\cup\\
  G\times \hat{Z}&\stackrel{\delta''}{\longrightarrow}&
  Y. \end{array}
$$
By \cite[III, Theorem 10.1]{Ha} $\delta''$ is smooth. Thus,
as $Y$ is reduced, \cite[VII, Theorem 4.9]{AK} implies that
$\hat{Z}$ is reduced.  Since $Y$ is $G$-stable, it is easy
to see that $\delta'$ factorizes through $\delta''$, hence
$\delta'=\delta''$.  When $Y=G.Z$, the morphism $\delta'$ is
surjective and \cite[VII, Theorem 4.9]{AK} then implies that
$Z$ is smooth if, and only if, $Y$ is smooth.
\end{proof}

Applying Proposition~\ref{slodowy}(iii) to a sheet
$Y=S_{G}$, one deduces that $S_{G}$ is smooth if and only if
the Slodowy slice $e+X$ is smooth. Using this method, the
following general result was obtained by Im~Hof:

\begin{thm}[\cite{IH}]
  \label{slicesmoothness}
  The sheets of a classical Lie algebra are smooth.
\end{thm}

Recall that the smoothness of sheets for $\fsl_N$ is due to
Kraft and Luna \cite{Kr} and, independentely, Peterson
\cite{Pe}. It is known that when $\g$ is of type
$\text{G}_2$, a subregular sheet of $\g$ is not normal
(hence is singular), see~\cite[8.11]{Sl}, \cite[6.4]{Boh} or
\cite{Pe}. It seems to be the only known example of non
smooothness of sheets.
\begin{Rqs}
  \label{remtriples}
  (1) Let $\sS = (e,h,f)$ be as above and pick $g \in
  G$. Then, the same results can be obtained for $g.e$ and
  $g.\sS$. In particular, one can construct a map
  \[
  \varepsilon : g.e + g.\h \to g.e + \g^{g.f}
  \]
  which induces a polynomial map $\varepsilon_{\mid g.e +
    g.\tf}$.
  \\
  (2) The results obtained from \ref{Gamma} to
  \ref{slicesmoothness} depend only on $S_G$ and $\sS$ but
  do not refer to $\lf$, $\tf$ or $\nf$.  Precisely, these
  results remain true when $e$ is replaced by $g.e$ and
  $\sS$ by any \Striplet containing $g.e$.  In particular,
  since $S_G$ contains a unique nilpotent $G$-orbit $G.e$,
  they remain true for any \Striplet $(e',h',f')$ such that
  $e' \in S_G$.
\end{Rqs}

\subsection{The regular $G$-sheet}
\label{regularsheet}
The set $\g^{\mathit{reg}}$ of regular elements in $\g$ is a
sheet, called the regular $G$-sheet, that we will denote by
$S_{G}^{\mathit{reg}}$.  We will use the notation and
results of the previous subsection with
$S_G=S_{G}^{\mathit{reg}}$.  One has $\tf=\h$ and
$G.(e+\h)=S_{G}^{\mathit{reg}}$ for any principal \Striplet
$(e,h,f)$ such that $e$ is regular and $h\in\h$.  Moreover,
$e+\g^f\subset S_{G}^{\mathit{reg}}$ and therefore
$S_{G}^{\mathit{reg}}=G.(e+\g^f)$ (cf. \cite{Ko}).

\begin{lm}\label{eplushconj}
  Adopt the previous notation.\\
  {\rm (i)} The semisimple part of an element $e+x\in e+\h$
  is conjugate to $x$.
  \\
  {\rm (ii)} Two regular elements are conjugate if and only
  if their semisimple parts are in the same $G$-orbit.
  \\
  {\rm (iii)} Two elements $e+x, e+y\in e+\h$ lie in the
  same $G$-orbit if and only if $W.x=W.y$.
\end{lm}

\begin{proof}
  The assertions (i) and (ii) follow from \cite[Lemma~11,
  Theorem~3]{Ko}, whence (iii) is a direct consequence of
  (i) and (ii).
\end{proof}

We can now state an important result of Kostant
\cite[Theorem~8]{Ko} in the following form:

\begin{lm}\label{gammatriv}
  The group $A$ is trivial, thus $\psi :
  S_{G}^{\mathit{reg}} \to e+\g^f=\varepsilon(e+\h)$ is a
  geometric quotient of $S_{G}^{\mathit{reg}}$.
\end{lm}

\subsection{The case $\g=\gl_{N}$}
\label{glN}
\subsubsection{The setting}
\label{settings}
In this section we assume that $\g=\gl(V)$, where $V$ is a
$\K$-vector space of dimension $N$.  By \cite[\S2]{Kr}, we
know that there exist two natural bijections from $G$-sheets
to partitions of $N$. Let $S$ be a $G$-sheet,
\begin{itemize}
\item the first map sends $S$ to the
  partition $\bolda=(\lambda_{1}\geqslant\lambda_{2}\geqslant
  \dots\geqslant\lambda_{\delta_{\Od}})$ of the unique
  nilpotent orbit $\Od$ contained in $S$. 
  (cf.~Proposition~\ref{unilp});
\item the second one sends $S$ to the partition
  $\tilde{\bolda} =(\tilde{\lambda}_1 \geqslant \dots
  \geqslant \tilde{\lambda}_{\delta_{\lf}})$ given by the
  block sizes of the Levi factor $\lf$ occurring in the
  datum $(\lf,0)$ of the dense $J_G$-class contained in~$S$.
\end{itemize}
It is well known that $\tilde{\bolda}$ is the transpose of $\bolda$.

Let $S_{G}$ be a $G$-sheet and
$\bolda=(\lambda_{1}\geqslant\lambda_{2}\geqslant
\dots\geqslant\lambda_{\delta_{\Od}})$ be the partition of
$N$ associated to the nilpotent orbit $\Od$ contained in
$S_G$.  Fix an element $e\in\Od$ and a basis
$$
\mathbf{v} = \bigl\{v_{j}^{(i)}\mid
i\in[\![1,\delta_{\Od}]\!], j\in[\![1,\lambda_{i}]\!]\bigr\}
$$
providing a Jordan normal form of $e$. Precisely, write $e =
\sum_i e_i$, where $e_i \in \g$ is defined by:
\begin{equation}
  \label{vbasis}
  e_i.v_{j}^{(k)}=
  \begin{cases}
    v_{j-1}^{(i)} \ & \text{if $k=i$ and $j=2,\dots,
      \lambda_{i}$;}
    \\
    0 \ & \text{otherwise.}
  \end{cases}
\end{equation}
Set $\qq_{i}:=\gl\bigl(v_{j}^{(i)}\mid j \in
[\![1,\lambda_{i}]\!]\bigr)$, which is a reductive Lie
algebra isomorphic to $\gl_{\lambda_{i}}$, and define
\[
\qq:=\bigoplus_i \qq_{i}.
\]
Let $\pr_{i} : \qq \to \qq_{i}$ be canonical projection.
For $x\in\qq$ we set $x_{i}:=\pr_{i}(x)$; conversely, for any
family $(y_{i})_{i}$ of elements $y_{i}\in\qq_{i}$ we can
define $y=\sum_{i}y_{i}\in \qq$.
 
We apply this construction to get an \Striplet $(e,h,f)
\subset \qq$ as follows.  Fixing the basis
$\bigl(v_1^{(i)},\dots,v_{\lambda_i}^{(i)}\bigr)$, one can
identify $\qq_i$ with the algebra of $\lambda_i\times
\lambda_i$-matrices. Using this identification, embed $e_i$
in the standard \Striplet $(e_i,h_i,f_i)$ of $\qq_i$
induced by the irreducible representation of $\sld$ of
dimension $\lambda_{i}$, i.e.:
$$
e_{i}=\left(\begin{array}{c c c c c} 0&1& 0&\cdots&0\\ 0 &0
    &1& \cdots&0\\ \vdots & \vdots & \vdots & \ddots &
    \vdots\\
    0&0&0&\cdots& 1\\ 0& 0& 0& \cdots & 0\end{array}\right),
\quad h_{i}=\left(\begin{array}{c c c c c} \lambda_{i}-1&0 &
    0&\cdots&0 \\ 0 &\lambda_{i}-3 &0& \cdots&0\\ 0 & 0
    &\lambda_{i}-5 &\cdots & 0
    \\
    \vdots & \vdots & \vdots & \ddots & \vdots\\
    0&0&0&\cdots& -\lambda_{i}+1\end{array}\right)
$$ 
(a well known similar formula gives $f_i$).  Then, $h=\sum_i
h_i$ and $f= \sum_i f_i$.
 
Clearly, the subspace
$$
\lf:=\bigoplus_{j} \gl(v_{j}^{(i)}\mid
i\in[\![1,\tilde{\lambda}_{j}]\!] )
$$ 
is a Levi factor of $\g$.  Denote by $\h := \bigoplus_i \h_i$
the Cartan subalgebra of diagonal matrices with respect to
the chosen basis $\mathbf{v}$. If $\tf$ is the center of
$\lf$ we then have $\tf\subset\h\subset \lf\cap\qq$.
 
Let $E^{i,i}_{j,j}$ be the element of $\h$ defined by
$E^{i,i}_{j,j}.v_{k}^{(l)}=v_{j}^{(i)}$ if $(i,j)=(l,k)$,
and $E^{i,i}_{j,j}.v_{k}^{(l)}=0$ otherwise. Each $t \in \h$
can then be written $t=\sum_{i,j} t_{i,j}E^{i,i}_{j,j}$ and
one has the following easy characterization of $\tf$:
\begin{equation}\label{carat}
  \tf=\{t\in\h\mid
  t_{i,j}=t_{i',j} \mbox{ for all }i\leqslant i';\; j \in
  [\![1,\lambda_{i'}]\!]\}.
\end{equation}
We will need later the following isomorphism:
\begin{equation}\label{alpha}
  \alpha:\left\{\begin{array}{rcl}
      \K^{\lambda_{1}}& \longisomto &\tf\\ 
      (x_{j})_{j\in[\![1,\lambda_{1}]\!]}&\mapsto
      &(t_{i,j})_{i,j}\end{array}\right.
\end{equation} 
where $t_{i,j}=x_{j}$ for all $i\in
[\![1,\lambda_{\delta_{\Od}}]\!]$, $1\leqslant j\leqslant
\lambda_{i}$.

Order, lexicographically, the elements of $\mathbf{v}$ by:
$v_j^{(i)} < v_\ell^{(k)}$ if $j < \ell$ or $j=\ell$ and $i
<k$.  Denote by $\bb$ the Borel subalgebra of $\g$
consisting of upper triangular matrices with respect to this
ordering of $\mathbf{v}$.  Then, the subspace $\bb+\lf$ is a
parabolic subalgebra having $\lf$ as Levi factor. Observe
that $h \in \h \subset \lf$ and that $e$ is regular in the
nilradical of $\bb+\lf$.  The constructions of
\S\ref{Katsylo} can be made here with $\mathfrak{j}= \bb +
\lf$ and the results of that subsection yield
$S_{G}=G.(e+\tf)$ (Proposition~\ref{axiomK}) and a map
$\varepsilon: e+\h\rightarrow e+\g^f$ (Lemma~\ref{epsilon}).

\begin{lm}\label{GeIH}
  {\rm (i)} The group $G^{e}$ is connected.
  \\
  {\rm (ii)} The map $\psi$ induces a bijection between
  $G$-orbits in $S_{G}$ and points in $X$.
\end{lm}

\begin{proof}
  Part (i) is a classical result, see for example
  \cite[6.1.6]{CM}.  Since the group
  $A=G^{e}/(G^{e})^{\circ}$ is then trivial, part~(ii)
  follows from Lemma~\ref{Gamma}.
\end{proof}

By Remark~\ref{ggprime} we may assume that
$\varepsilon=\varepsilon^{\qq}=\sum_{i}\varepsilon_{i}$
where
$$
\varepsilon_{i}:=\varepsilon^{\qq_{i}}:
e_{i}+\h_{i}\rightarrow e_{i}+\qq_{i}^{f_{i}}.
$$  
As $e_{i}\in\qq_{i}$ is regular, the study of $\varepsilon$
is therefore reduced to the regular case.

\subsubsection{The regular case and its consequences}
\label{regA}
We need to study in more details the maps $\varepsilon_{i} :
e_{i}+\h_{i} \to e_{i}+\qq_{i}^{f_{i}}$ introduced at the
end of the previous subsection, where, as already said,
$e_i$ is regular in $\qq_i\cong\gl_{\lambda_i}$.
\\
To simplify the notation we (temporarily) replace
$\gl_{\lambda_i}$ by $\gl_N$ and $e_i$ by
$e^{\mathit{reg}}$, the regular element of
$\g=\gl_{N}$. Hence,
$$
e^{\mathit{reg}}.v_{j}=\left\{\begin{array}{l l} v_{j-1} &
    \mbox{ if } j=2,\dots,N;\\ 0 & \mbox{ if }
    j=1. \end{array}\right.
$$
Recall that $\h \subset \gl_N$ is the set of diagonal
matrices in the basis ${\bf v}^{\mathit{reg}}:=(v_{j})_j$.
We can then define the canonical principal triple
$(e^{\mathit{reg}},h^{\mathit{reg}},f^{\mathit{reg}})$ with
respect to this basis (see the definition of the triple
$(e_{i}, h_{i}, f_{i})$ in~\ref{settings}).  In this case,
$\varepsilon^{\mathit{reg}}:e^{\mathit{reg}}+\h\rightarrow
e^{\mathit{reg}}+\g^{f^{\mathit{reg}}}$ can be considered as
the restriction of the geometric quotient map of
$\g^{\mathit{reg}}$ (cf.~Lemma~\ref{gammatriv}).

Let $0\leqslant k\lnq N$, the \emph{$k$-th subdiagonal}
(resp. $k$-th \emph{supdiagonal}) is the subspace of
matrices $[a_{i,j}]_{i,j}$ such that $a_{i,j}=0$ unless
$i=j+k$ (resp. $i=j-k$). We denote it by $\ff^{(k)}$.
\begin{lm}\label{sym}
  The map $\varepsilon^{\mathit{reg}}$ is given by
  \[
  \varepsilon^{\mathit{reg}}(e^{\mathit{reg}}+t)=e_i+\sum_{j\leqslant
    0} P_j(t) \quad \text{for all $t\in \h$},
  \]
  where each $P_{j}: \h\rightarrow \ff^{(-j)}$ is a
  homogeneous polynomial map of degree $-j+1$, symmetric in
  the eigenvalues of the elements of $\h$.
\end{lm}

\begin{proof}
  Recall that $\g(2j,h^{\mathit{reg}})$ is the $2j$-th
  eigenspace of $\ad_\g h^{\mathit{reg}}$.  It is easily
  seen that $\g(2j,h^{\mathit{reg}})=\ff^{(-j)}$ when
  $j\leqslant0$.  Using Lemma~\ref{epsilon}(ii), the only
  fact remaining to be proved is that the polynomial map
  $P_{j}$ is symmetric.  Observe that the Weyl group
  $W=W(\g,\h)$ acts as the permutation group of
  $[\![1,N]\!]$ on the eigenvalues of $\h$ and recall that,
  by Lemma~\ref{gammatriv}, $\varepsilon^{\mathit{reg}}$ is
  a quotient map with respect to $W$. Consequently, for all
  $t\in\h$ and $w\in W$ one has
  $\varepsilon^{\mathit{reg}}(e^{\mathit{reg}}+w.t)=
  \varepsilon^{\mathit{reg}}(e^{\mathit{reg}}+t)$. Thus
  $P_j$ is symmetric.
\end{proof}

If $t$ is a semisimple element of $\g$ we denote by
$\spec(t)$ the set of eigenvalues of $t$ and by $m(t,c)$ the
multiplicity of $c \in \K$ as an eigenvalue of $t$, with the
convention that $m(t,c)=0$ if $c\notin \spec(t)$.
 
The next lemma is a direct consequence of
Lemma~\ref{eplushconj}.

\begin{lm} \label{eigenjordan} Let $t\in\h$ and $c \in
  \spec(t)$. In a Jordan normal form of
  $e^{\mathit{reg}}+t$, there exists exactly one Jordan
  block associated to $c$, and its size is $m(t,c)$.
\end{lm}

Recall that we want to apply Lemma~\ref{eigenjordan} to the
regular elements $e_{i}$ in $\qq_{i} \cong \gl_{\lambda_i}$;
we therefore generalize the previous notation as
follows. For $t=\sum_{i}t_{i}\in \h\subset\bigoplus_i
\qq_{i}$ and $c\in\K$, let $m_{i}(t,c)$ be the multiplicity
of $c$ as an eigenvalue of $t_{i}$.  Then, $\sum_{i}
m_{i}(t,c)=m(t,c)$ and we have the following easy
consequence of Lemmas~\ref{eplushconj}
and~\ref{eigenjordan}.

\begin{cor}\label{corlm}
  Let $t \in \h$. The semisimple part of $e+t$ is conjugate
  to $t$. Its nilpotent part is associated to the partition
  of $N$ given by the integers $m_{i}(t,c)$, $c \in
  \spec(t)$ and $i\in [\![1,\delta_{\Od}]\!]$.
\end{cor}

\section{Symmetric Lie algebras}
\setcounter{thm}{0}
\label{symmetric}\label{SLA}
We now turn to the {symmetric case}.  We will denote a
symmetric Lie algebra either by $(\g,\theta)$, $(\g,\kk)$ or
$(\g,\kk,\pp)$, where: $\theta$ is an involution of $\g$,
$\kk$ (resp.~$\pp$) is the $+1$(resp.~$-1$)-eigenspace of
$\theta$ in $\g$. Then, $\g=\kk\oplus \pp$, $\kk$ is a Lie
subalgebra and $\pp$ is a $\kk$-module under the adjoint
action. Recall from~\S\ref{somenotation} that $K$ is the
connected subgroup of $G$ such that $\Lie(K)= \ad_\g(\kk)$
and that $K$ is the connected component of
\begin{equation} \label{ktheta} \Kt :=\{g\in G\mid g
  \circ\theta=\theta \circ g\}=N_{G}(\kk).
\end{equation}
Sheets and Jordan classes can naturally be defined in this
setting, see \cite[39.5~\& 39.6]{TY}. One has,
cf.~\cite[Proposition~5]{KR},
\[
\dim K.x= \frac{1}{2} \dim G.x \ \; \text{for all $x \in
  \pp$}
\]
and we set:
\[
\pp^{(m)}:=\{x\in \pp\mid \dim K.x=m\}\subset \g^{(2m)}.
\]

\begin{defi}\label{defJK}
  The \emph{$K$-sheets} of $(\g,\theta)$ are the irreducible
  components of the $\pp^{(m)}$, $m \in \N$.
  \\
  Let $x=s+n$ (where $s,n \in \pp$) be the Jordan
  decomposition of an element $x\in\pp$. The \emph{Jordan
    $K$-class} of $x$, or $J_{K}$-class of $x$, is the set
  \[
  J_{K}(x):=K.(\cpps+n)\subset\pp.
  \]
\end{defi}

It is easily seen that $\pp$ is the finite disjoint union of
its $J_K$-classes and that a $K$-sheet is the union of the
$J_K$-classes it contains \cite[39.5.2]{TY}.

There exists a symmetric analogue to the notion of
\Striplet.  An \Striplet $(e,h,f)$ is called \emph{normal}
if $e,f\in\pp$ and $h\in \kk$.  Similarily to the Lie
algebra case, there is a bijection between $K$-orbits of
nilpotent elements and $K$-orbits of normal {\Striplet}s,
see~\cite[Proposition 4]{KR} or \cite[38.8.5]{TY}.

Any semisimple symmetric Lie algebra can be decomposed as
$(\g,\theta)=\prod_i(\g_{i},\theta_{\mid\g_{i}})$ where
$(\g_{i},\theta_{\mid\g_{i}})$ is a symmetric Lie subalgebra
of one of the following two types:
\begin{itemize}
\item[(a)] $\g_{i}$ simple;
\item[(b)] $\g_{i}=\g_{i}^{1}\oplus\g_{i}^2$, with
  $\g_{i}^{j}$ simple, $\theta_{\mid\g_{i}^j}$ isomorphism
  from $\g_{i}^j$ onto $\g_{i}^{3-j}$, $j=1,2$.
\end{itemize}
Each $(\g_{i},\theta_{\mid\g_{i}})$ is called an irreducible
factor of $(\g,\theta)$; this decomposition is unique (up to
permutation of the factors).

\subsection{Type~0} \label{type0} When $(\g,\theta)$ is the
sum of two simple factors as in the above case~(b), then
$\g$ is said to be of ``type~0''.  We slightly enlarge this
definition by saying that a pair $(\g,\theta)$ is a
\emph{symmetric pair of type}~0 if
$$
\g=\g'\times \g', \quad \theta(x,y)=(y,x), \quad
\kk=\{(x,x)\mid x\in\g'\}, \quad \pp=\{(x,-x)\mid x\in\g'\},
$$ 
where $\g'$ is only assumed to be reductive.  Recall the
following easy observations. Let $\pr_{1}$ be the projection
on the first coordinate. Via $\pr_{1}$, the Lie algebra
$\kk$ is isomorphic to $\g'$, thus $K$ is isomorphic to the
adjoint group $G'$ of $\g'$.  Moreover, the $K$-module $\pp$
is isomorphic to the $G'$-module $\g'$.  

Using Lemma~\ref{pppmmm} it is not hard to prove the following.

\begin{lm}\label{gmgm}
  {\rm (i)} The $G$-sheets of $\g=\g'\times \g'$ are the
  $S'\times S''$ where $S'$ and $S''$ are $G'$-sheets of
  $\g'$.
  \\
  {\rm (ii)} The sets $\{(x,-x)\mid x\in S'\}$, where $S'$
  is a $G'$-sheet of $\g'$, are the $K$-sheets of $\pp$.
\end{lm}


We would like to link the Lie algebra case to the symmetric case in type~0. This partly rely on the following definition.
If $Y$ is a subset of $\pp$, we set
$$\phi(Y):=\pr_{1}(Y)\times \pr_{1}(-Y)\subset \g.$$

\begin{prop}
  {\rm (i)} If $Y$ is a $K$-orbit (resp.~a~$J_{K}$-class or
  a~$K$-sheet) of $\pp$, then $\phi(Y)$ is a $G$-orbit
  (resp.~a~$J_{G}$-class or a~$G$-sheet) of $\g$.
  \\
  {\rm (ii)} If $Z$ is a $G$-orbit (resp.~a~$J_{G}$-class)
  of $\g$ intersecting $\pp$, then $Z\cap\pp$ is a $K$-orbit
  (resp.~a~$J_{K}$-class) of~$\pp$.
  \\
  {\rm (iii)} Each pair of distinct sheets of $\g'$ have an empty
  intersection if, and only if, the intersection of each
  $G$-sheet of $\g$ with $\pp$ is either empty or a single
  $K$-sheet .
\end{prop}

\begin{proof}
(i) and (ii) are straightforward.\\
(iii) Let $Z$ be a $G$-sheet of $\g$ and write $Z$ as the
product of two $G'$-sheets of $\g'$, say $Z=Z_{1}\times
Z_{2}$.  If $(x,-x)\in Z$, it follows that $x \in Z_{1} \cap
Z_2$ and, in particular, $Z_{1}\cap Z_{2}\neq\emptyset$. If
$Z_{1}=Z_{2}$, then Lemma~\ref{gmgm} shows that $Z\cap\pp$
is a $K$-sheet.  Otherwise, one has $Z\cap\pp\subsetneq
(Z_{1}\times Z_{1})\cap\pp$ and $Z\cap\pp$ is not a
$K$-sheet of~$\pp$.
\end{proof}

Since a $G'$-sheet of $\g'$ contains exactly one nilpotent
orbit of $\g'$, two $G'$-sheets of $\g'$ have a non-empty
intersection if and only if they contain the same nilpotent
orbit (cf.~\cite[39.3.2]{TY}).  A necessary and sufficient
condition for $\g'$ to have intersecting sheets is therefore
to have more sheets than nilpotent orbits.  Using \cite{Boh}
one can show that there are only two cases where sheets are
in bijection with nilpotent orbits: when $\g'$ is of type~A
or $\text{D}_{4}$. Therefore we have:

\begin{cor}\label{cor000}
  Any $G$-sheet of $\g$ intersects $\pp$ along one $K$-sheet
  if and only if the simple factors of $\g'$ are of
  type~{\rm A} or {\rm D}$_{4}$.
\end{cor}

The next (easy) result is true in type~0, but false in
general.

\begin{prop}\label{orbX0}
  Let $S_{G}$ be a $G$-sheet of $\g$ intersecting $\pp$. Let
  $\sS=(e,h,f)$ be a normal \Striplet containing a nilpotent
  element $e \in S_{G}\cap\pp$. Then, if
  $e+X(S_{G},\sS)=(e+\g^f)\cap S_{G}$, one has
  \[
  S_{G}\cap\pp=K.(e+X(S_{G},\sS)\cap\pp).
  \]
\end{prop}

\begin{proof}
  Write $S_{G}=S_{1}\times S_{2}$ with $S_{1},S_{2}$ sheets
  of $\g'$(cf.~Lemma \ref{gmgm}) and set $e=(e',-e')$,
  $f=(f',-f')$, $e',f' \in \g'$.
  Recall that $\pr_{1}$ yields an isomorphism between $\pp$
  and $\g'$ and that $\pr_{1}(S_{G}\cap\pp)=S_{1}\cap
  S_{2}$.  If $X_i \subset \g'$ is defined by
  $(e'+X_{i})=(e'+\g'^{f'})\cap S_{i}$, one has
  $\pr_{1}(e+X\cap\pp)=e'+X_{1}\cap X_{2}$.  Moreover,
  $\pr_{1}(K.(e+X\cap\pp))=G'.(e'+X_{1}\cap X_{2})=S_{1}\cap
  S_{2} = \pr_{1}(S_{G}\cap\pp)$. Since ${\pr_{1}}_{\mid
    \pp}$ is an isomorphism, we get the desired result.
\end{proof}

\subsection{Root systems and semisimple elements}
\label{basis}
Let $(\g,\kk,\pp)$ be the semisimple symmetric Lie algebra
associated to an involution $\theta$.  Fix a Cartan
subspace $\af$ of $\pp$; recall that the \emph{rank} of the
symmetric pair $(\g,\kk)= (\g,\theta)$ is $\rk (\g,\theta):=
\dim \af$. Let $\df$ be a Cartan subalgebra of
$\cc_{\kk}(\af)$.  Then, $\h:=\af\oplus\df$ is a
$\theta$-stable Cartan subalgebra of $\g$
(\cite[37.5.2]{TY}).
If $V:=\h^*$ and $\sigma$ denotes the transpose of $\theta$,
one can consider the $\sigma$-stable root system
$R=R(\g,\h)\subset V$ and we set (see~\cite[36.1]{TY}):
\begin{gather*}
  V':=\{x\in \h^*\mid \sigma(x)=x\}=\{x\mid x_{\mid\af}=0\},
  \\ V'':=\{x\in \h^*\mid \sigma(x)=-x\}=\{x\mid
  x_{\mid\df}=0\},
  \\
  R^{0}:=R\cap V'=\{\alpha\in R\mid \sigma(\alpha)=\alpha\},
  \quad R^{1}:=\{\alpha\in R\mid \sigma(\alpha)\neq\alpha\}.
\end{gather*}
Recall that $R^0$ is a root system. One has $V=V'\oplus
V''$; more precisely, $x \in V$ decomposes as $x= x' + x''$,
where $x' :=\frac{1}{2} (x+\sigma(x)) \in V'$, $x'':=
\frac{1}{2}(x-\sigma(x))=x_{\mid\af}\in V''$. When $x \in R$
is a root, $x''$ is called its restricted root.  Set:
\[
S=\{\alpha'' \mid \alpha\in R^{1}\}.
\]
Then, $S \subset \af^*$ is a (not necessarily reduced) root
system, see \cite[36.2.1]{TY}, which is called the
\emph{restricted root system} of $(\g,\theta)$.  We denote
by $W$, resp.~$W_S$, the Weyl group of the root system $R$,
resp.~$S$, and we set
\[
W_\sigma := \{ w \in W \mid w \circ \sigma = \sigma \circ w
\}.
\]
If $B \subset R$ is a fundamental system (i.e.~a basis of
$R$), denote by $R_+$ (resp.~$R_-$) the set of positive
(resp.~negative) roots associated to $B$.  In order to
define the Satake diagram of the symmetric pair $(\g,\kk)$
one needs to work with some special fundamental systems for
$R$. Setting
$$
R^{1}_{\pm}:=R^{1}\cap R_{\pm}
$$
one can give the following definition:
 
\begin{defi} \label{fundamental} (\cite[36.1.4]{TY},
  \cite[2.8]{Ar}) A $\sigma$-fundamental system $B \subset
  R$ is a fundamental system satisfying the following
  conditions:
  \begin{enumerate}
  \item[(i)] $\sigma(R^{1}_{+})=R^{1}_{-}$;
  \item[(ii)] If $\alpha\in R^{1}_{+}$, $\beta\in R$ and
    $\alpha-\beta\in V'$, then $\beta\in R^{1}_{+}$;
  \item[(iii)] $(R^{1}_{+}+R^{1}_{+})\cap R\subset
    R^{1}_{+}$;
  \end{enumerate}
\end{defi}

Let $V_\Q$ be the $\Q$-vector space spanned by $R$; then
$V_\Q=V_\Q'\oplus V_\Q''$ where $V'_\Q:= V_\Q \cap V'$,
resp.~$V''_\Q := V_\Q \cap V''$, are $\Q$-forms of $V'$,
resp.~$V''$ (cf.~\cite[proof of 36.1.4]{TY}). Denote by
$\af_{\Q}$ the $\Q$-form of $\af$ given by the dual of
$V_\Q''$.  The choice of a $\Q$-basis $C=(e_1,\dots, e_l)$
of $V_\Q$ gives rise to a lexicographic ordering $\prec$ on
$V_\Q$ and, therefore, to a set of positive roots
$R_{+,C}:=\{\alpha\in R\mid \alpha\succ0\}$. Recall
\cite[18.7]{TY} that for each choice of such a basis $C$,
there exists a unique fundamental system $B_C$ such that
$R_{+,C}$ is the set of positive roots with respect to $B$.
The existence of a $\sigma$-fundamental system is ensured by
the next lemma, which provides all the $\sigma$-fundamental
systems, see Proposition~\ref{wawsa}(iv).

\begin{lm}\label{lexi}
  Let $(e_1,\dots, e_p)$, resp.~$(e_{p+1},\dots, e_l)$, be a
  basis of $V_\Q''$, resp.~$V_\Q'$, and set $C=(e_1,\dots,
  e_l)$. Then $B_C$ is a $\sigma$-fundamental system such
  that $B_{C}^0 := B_{C}\cap V'$ is a fundamental system of
  $R^0$.
\end{lm}

\begin{proof}
  By \cite[36.1.4]{TY} $B_{C}$ is a $\sigma$-fundamental
  system. The second statement follows from the fact that
  $B_C\cap V'$ is the set of simple roots associated to the
  lexicographic ordering associated to the basis
  $(e_{p+1},\dots, e_l)$.
\end{proof}
 
\begin{prop}\label{wawsa}
  {\rm (i)} The map $w\mapsto w_{\mid V''}$ induces a
  surjective homomorphism $W_\sigma\rightarrow W_S$ whose
  kernel is $W^0$, the Weyl group of $R^0$.
  \\
  {\rm (ii)} For $x\in V_\Q''$, one has $W_S.x=W.x \cap
  V''_\Q$. Dually, $W_S.a = W.a \cap \af_\Q$ for all $a\in
  \af_\Q$.
  \\
  {\rm (iii)} Let $B$ be a $\sigma$-fundamental
  system. Then, the restricted fundamental system
  $B'':=\{\alpha''\mid \alpha\in B\}$ is a fundamental system
  of the restricted root system $S$.
  \\
  {\rm (iv)} $W_{\sigma}$ acts transitively on the set of
  $\sigma$-fundamental systems.
\end{prop}

\begin{proof}
  Claims (i) and (ii) are proved in \cite[36.2.5,
  36.2.6]{TY}, while (iii) and (iv) can be found in
  \cite[2.8 and 2.9]{Ar}.
\end{proof}

\begin{Rqs}
  \label{Klift}
  (1) The restriction to $\af$ yields an isomorphism
  $N_K(\af)/Z_K(\af) \isomto W_S$, cf.~\cite[38.7.2]{TY}.
  \\
  (2) Let $w \in W_\sigma$, then there exists $k \in K$ such
  that $k_{\mid \h} = w$. This can be shown as
  follows. Recall that $\h= \af \oplus \df$, where $\df$ is
  a Cartan subalgebra of $\mathfrak{u}:=
  \cc_{\kk}(\af)$. Note that $w.\af =\af$ and $w.\df = \df$.
  Pick $k_1\in K$ such that ${k_1}_{\mid\af}=w_{\mid \af}
  \in W_S$. Let $U \subset C_K(\af)$ be the connected
  subgroup of $K$ with Lie algebra $\mathfrak{u}$. The Weyl
  group of the root system $R^0= R(\mathfrak{u},\df)$ is
  $W^0\cong N_U(\df)/Z_U(\df)$, see \cite[38.2.1]{TY}.  By
  composing $k_1$ with an element of $U$ we may assume that
  $k_1.\h = \h$ and ${k_1}_{\mid\af}=w_{\mid \af}$. Set $w_0
  := (w \circ {k_1^{-1}})_{\mid \h} \in W$; one has
  ${w_0}_{\mid \af} = \Id_{\af}$, therefore $w_0 \in W^0$
  and we can find $k_0 \in N_U(\df)$ such that ${k_0}_{\mid
    \df}={w_0}_{\mid\df}= w_{\mid\df} \circ{k_1^{-1}}_{\mid
    \df}$. Setting $k:=k_0k_1 \in K$ we obtain $k_{\mid \af}
  = {k_1}_{\mid \af} = w_{\mid \af}$ and $k_{\mid \df}
  ={k_0}_{\mid \df} \circ {k_1}_{\mid \df}= w_{\mid\df}$,
  thus $k_{\mid\h}= w$.
\end{Rqs}

Fix a $\sigma$-fundamental system $B$; from the Dynkin
diagram $D$ associated to $B$ one can construct the Satake
diagram $\bar{D}$ of $(\g,\theta)$ as follows.  The nodes
$\alpha$ of $D$ such that $\alpha''=0$ are colored in black,
the other nodes being white; two white nodes $\alpha \ne
\beta$ of $D$ such that $\alpha''=\beta''$ are related by a
two-sided arrow.  This defines the new diagram $\bar{D}$.
Recall that the Satake diagram of $(\g,\theta)$ does not
depend on the choice of the $\sigma$-fundamental system $B$,
and that two semisimple symmetric Lie algebras are
isomorphic if and only if they have the same Satake diagram
(cf.~\cite[Theorem 2.14]{Ar}).  A classification of
symmetric Lie algebras together with their Satake diagrams
and restricted root systems is given in \cite[Ch.~X]{He1}.

We now recall the (well-known) links between $G$-conjugacy
and $W$-conjugacy, and their analogues for a symmetric Lie
algebra.

\begin{lm} \label{Korbss1} {\rm (i)} Two elements of $\h$
  (resp.~$\af$) are $G$(resp.~$K$)-conjugate if and only if
  they are $W$(resp.~$W_{S}$ or, equivalently,
  $W_{\sigma}$)-conjugate.
  \\
  {\rm (ii)} Let $x,y\in\h$ (resp.~$x,y\in\af$), then the
  Levi factors $\g^x$ and $\g^y$ are
  $G$(resp.~$K$)-conjugate if, and only if, they are
  $W$(resp.~$W_{S}$ or, equivalently,
  $W_{\sigma}$)-conjugate.
\end{lm}

\begin{proof}
  (i) is standard. 

  \noindent (ii) We write the proof for $x,y \in \af$. Thanks to
  \cite[29.2.3 \&~37.4.10]{TY} applied to $(\g^y,\kk^y)$, the Levi factors $\g^x, \g^y$ are ~$K$-conjugate if, and only if,
  there exists $g\in K$ such that $g.\g^x=\g^y$
  and $g.\h=\h$. Then $g$ induces an element of $W$, and
  therefore of $W_{\sigma}$ since $g\circ\sigma=\sigma\circ g$. Observe finally that Proposition~\ref{wawsa}(i) implies the equivalence of $W_{\sigma}$ and  $W_{S}$-conjugacy.
  Conversely, \cite[38.7.2]{TY} applied to $\cppnb{x}$ and $\cppnb{y}$ shows
  that the conjugation under $W_{S}$ implies the
  $K$-conjugation.

\end{proof}

In general, if $x \in \pp$, the intersection of $G.x$ with
$\pp$ contains more than one orbit
(cf.~\cite[38.6.1(i)]{TY}). But, when $x$ is semisimple one
can prove the following result, for which we provide a proof
since we did not find a reference in the literature.

\begin{prop}\label{Korbss}
  Let $s\in\pp$ be semisimple. Then, $G.s\cap\pp=K.s$.
\end{prop}

\begin{proof}
  Recall that any semisimple element of $\pp$ is
  $K$-conjugate to an element of $\af$,
  cf.~\cite[37.4.10]{TY}. Therefore, by
  Lemma~\ref{Korbss1}(i), it suffices to show that the
  property~(ii) of Proposition~\ref{wawsa} holds for all
  $a\in\af$, i.e.~$W_S.a = W.a\cap \af$.  Denote by $\LL$
  one of the fields $\Q$ or $\K$. For $(w,w') \in W\times
  W_{S}$, define linear subspaces of
  $\af_{\LL}:=\af_{\Q}\otimes_\Q \LL$ by:
  \[
  E^{w,w'}_{\LL} :=\ker_{\af_{\LL}}(w-w') =
  \{a\in\af_{\LL}\mid w.a=w'.a\}, \quad
  E^w_{\LL}:=w^{-1}(\af_{\LL})\cap\af_{\LL}.
  \]
  From Proposition~\ref{wawsa}(ii) one gets that
  $E^w_{\Q}=\bigcup_{w'\in W_{S}} E^{w,w'}_{\Q}$; thus,
  there exists $w'\in W_{S}$ such that
  $E^w_{\Q}=E^{w,w'}_{\Q}$.
  The flatness of ${-}\otimes_\Q\K$ yields:
  \[
  E^{w,w'}_{\K}=E^{w,w'}_{\Q} \otimes_\Q\K, \quad E^w_{\K}=
  E^w_{\Q} \otimes_\Q \K.
  \]
  Therefore, for any $w\in W$, there exists $w'\in W_{S}$
  such that $w'_{\mid E^w_{\K}}=w_{\mid E^w_{\K}}$. It
  follows that Proposition~\ref{wawsa}(ii) is satisfied for
  all $a\in \af= \af_{\K}$.
\end{proof}

\begin{cons}
  Proposition~\ref{Korbss} yields a bijection between
  $K$-orbits of semisimple elements of $\pp$ and $G$-orbits
  of semisimple elements intersecting $\pp$.
\end{cons}

Recall \cite{Ko,KR} that the set of semisimple
$G$(resp.~$K$)-orbits is parameterized by the categorical
quotient $\g \qmod G$ (resp.~$\pp\qmod K)$, and that $\K[\g
\qmod G] \cong \K[\h/W]= S(\h^*)^W$, $\K[\pp\qmod K] \cong
\K[\af/ W_S] = S(\af^*)^{W_S}$.  The previous consequence
can then be interpreted as follows.
 
Let $\gamma$ be the map which associates to the
$W_{S}$-orbit of $a \in \af$, the orbit $W.a \subset \h$;
hence, $\gamma : \af/W_S \to \h/W$. Define
$Z:=\gamma(\af/W_{S})\subset \h/W$ to be the image of
$\gamma$ and let $\phi : \af/W_S \to Z$ be the induced
surjective map. Write $\gamma = \iota \circ \phi$, where
$\iota : Z \to \h/W$ is the natural inclusion.  

Then we can get the following from Proposition~\ref{Korbss}:

\begin{cor} \label{normalization} The morphism $\phi :
  \af/W_{S}\rightarrow Z$ is a bijective birational map, and
  $\af/W_{S}$ is the normalization of $Z$.
\end{cor}

One must observe that the injective comorphism $\phi^*$ is not
surjective, i.e.~$Z$ is not normal, in general.  This
question has been studied in \cite{He2,He3,Ri2,Pa4}.  The
notation being as in \cite[Ch.~X]{He1}, the results obtained
in the previous references show that $\phi$ is an
isomorphism when $\g$ is of classical type, and in the
exceptional cases of type EI, EII, EV, EVI, EVIII, FI,
FII,~G. In cases EIII, EIV, EVII, EIX, it is known that
$\phi^*$ (or, equivalently, $\gamma^*$) is not surjective,
cf.~\cite{He2,Ri2}.

\begin{Rq}
  By standard arguments one can see that the results
  obtained in \ref{Klift}, \ref{Korbss1} and \ref{Korbss}
  remain true when $(\g,\theta)$ is a \emph{reductive}
  symmetric Lie algebra.
\end{Rq}

\subsection{Property (L)} \label{property(L)} Let
$(\g,\theta)=(\g,\kk,\pp)$, $\af, \h$, $R, R^0, R^1, S$ be
as in~\ref{basis}, and fix a $\sigma$-fundamental system $B$
of $R$ (cf.~Definition~\ref{fundamental}).  The next
definition introduces an important property in order to
study the $K$-conjugacy classes of Levi factors of the form
$\g^s$, $s\in\pp$ semisimple. Observe that $(\g^s,\kk^s)$ is
a symmetric Lie algebra, that we will call a
\emph{subsymmetric pair}. 

\begin{defi}\label{propL}
  The pair $(\g,\kk)$ satisfies the property~(L) if, for all
  semisimple elements $s,u\in\pp$:
  \begin{equation*}
    \tag{L}
    \{\exists \, g\in G, \ g.\g^s=\g^{u}\} \, \iff \,
    \{\exists \, k\in K, \ k.\g^s=\g^{u}\} . 
  \end{equation*}
\end{defi}




\begin{Rq}
  More generally, when $(\g,\theta)$ is a reductive
  symmetric Lie algebra, the condition (L) holds if and only
  if it holds for $([\g,\g],\theta)$.
\end{Rq}

The aim of this section is to prove that the property~(L)
holds for any reductive symmetric Lie algebra
(cf. Theorem~\ref{L}).  We are going to show that it is
sufficient to check~(L) for some Levi factors $\g^s$ of a
particular type, cf.~Corollary~\ref{carac2}.

\begin{defi} \label{standardlevi}
One says that a standard Levi factor $\lf$ \emph{arises from $\pp$} if there is $s\in \af_{\Q}$ lying in the positive Weyl chamber for $B$ and such that $\lf=\g^s$.
\end{defi}

Recall from Section~\ref{general} that there is a natural
one to one correspondence between standard Levi factors and
subsets of $B$. In this correspondence, to a Levi factor
$\lf$ one associates the subset
$$
I_{\lf}:=\{\alpha\in B\mid \alpha(s)=0\}
$$ 
where $s$ is any element in $(\g^{\lf})^{\bullet}$.
Conversely, from any subset $I\subset B$ one gets a Levi
subalgebra by setting:
$$
\lf_{I}:=\h \oplus \left(\oplus_{\alpha\in\langle
    I\rangle}\g^{\alpha}\right)
$$
where $\langle I\rangle = \Z I \cap R$. Remark that
$\g^{\lf_{I}}=\{h \in \h : \alpha(h) = 0 \ \text{for all
  $\alpha \in I$}\}$.

Let $D$ be the Dynkin diagram defined by $B$ and denote by
$\bar{D}$ the associated Satake diagram.  Let $B^{0}\subset
B$ be the set of black nodes of $\bar{D}$; recall that $B^0$
is a fundamental system of $R^0$ (cf.~Lemmas~\ref{lexi}
and~\ref{wawsa}).  Set
\[
B^{2} := \left\{ (\alpha_{1},\alpha_{2}) \in B \times B :
  \alpha_1 \ne \alpha_2, \; \alpha_1'' = \alpha_2''\right\},
\quad B^{3}:= \left\{\alpha_{1}-\alpha_{2} \mid
  (\alpha_{1},\alpha_{2}) \in B^{2}\right\} \subset\h^*_{\Q}
.
\]
Thus, $B^{2}$ is the set of pairs of white nodes
$(\alpha_{1} \ne \alpha_{2})$ of $\bar{D}$ connected by a
two-sided arrow (note that $(\alpha_{1},\alpha_{2})\in
B^{2}$ $\iff$ $(\alpha_{2},\alpha_{1}) \in B^{2}$).  Denote
by $\overline{B^{2}}\subset B$ the set of all nodes pointed
by such an arrow, i.e.~$\overline{B^{2}} = \{\alpha \in B :
\exists \, \beta \in B, \; (\alpha,\beta) \in B^2\}$.  A
subset $I \subset B$ is said to be \emph{stable under
  arrows} if $(\alpha_{1},\alpha_{2}) \in B^2$ with
$\alpha_1 \in I$ implies $\alpha_2 \in I$.

\begin{Rq}\label{carac}
  The subspace $\af_\Q \subset \h_\Q$ is the intersection of
  the kernels of elements of $B^{0}\cup B^{3}$.
  A standard Levi factor $\lf$ arises from $\pp$ if, and
  only if, $I_{\lf}$ is stable under arrows and contains
  $B^{0}$.
\end{Rq}

We now want to describe the subalgebra $\g^s$ when $s\in
\af$ semisimple.  Set
\[
E_s := \{\varphi \in \h^*_\Q=V_\Q : \varphi(s) = 0\}, \quad
R_{s}:=E_{s}\cap R .
\]
Then, $R_{s}$ is a root subsystem of $R$
(cf.~\cite[18.2.5]{TY}) and, with obvious notation, the
$\Q$-vector space $F_{s}$ spanned by $R_s$ decomposes as
$F_{s}'\oplus F_{s}''$.  The restriction to
$\h_{s,\Q}:=\h_{\Q}\cap[\g^s,\g^s]$ identifies $F_{s}$ with
$\h_{s,\Q}^*$ and $R_{s}$ with the root system of
$(\g^s,\kk^s)$. One can therefore apply to $R_s$ the results
of section~\ref{basis}.
 
Let $S_{s}$ be the restricted root system of $R_{s}$. As
$s\in \af$, one has:
\begin{equation}\label{Fs}
  S_{s}=\{x''\mid x\in R^1,
  x(s)=0\}=\{x''\mid x\in R^1, x''(s)=0\}=S\cap
  F_{s}''.
\end{equation}
Let $B_{s}$ be a $\sigma$-fundamental system of $R_{s}$. One
can write $B_{s}=B_s^0\sqcup B_{s}^{1}$ with $B_{s}^0\subset
R^0$, $B_{s}^1\subset R^1$ and we denote by $B_{s}''$ the
restricted fundamental system of $S_{s}$ associated to
$B_{s}$.
 
We can now prove the following result:

\begin{prop}\label{levistandard}
  Each Levi factor $\g^s$, $s\in \pp$, is $K$-conjugate to a
  standard Levi factor that arises from $\pp$.
\end{prop}

\begin{proof}
  Since the element $s\in\pp$ is semisimple, it is
  $K$-conjugate to an element of $\af$ and we may as well
  suppose that $s\in \af$.  We will use the previous
  notation relative to $R_{s}, S_{s}$ and a fixed
  $\sigma$-fundamental system $B_s \subset R_s$.

  We first show that there exists $w\in W_{\sigma}$ such
  that $B_{s}\subset w.B$.  Since $V'_{\Q}\subset E_{s}$ one
  has $R^0\subset R_s$, and $B_{s}^0$ being a fundamental
  system of the root system $R^0$, it can be conjugated to
  $B^{0}$ by an element of $W^0$.  As $B_{s}''$ is a
  fundamental system of $S_{s}=S\cap F_{s}''$
  (see~\eqref{Fs}), \cite[18.7.9(ii)]{TY} implies that
  $B_{s}''$ is a $W_{S}$-conjugate of a subset of $B''$.
  Combining these two facts and Lemma~\ref{wawsa}(i), one
  gets the existence of $w\in W_{\sigma}$ such that
  $B_{s}^0=w.B^{0}$ and $B_{s}''\subset w.B''$.
  \\
  We claim that $B_{s}\subset w.B$, i.e.~$B_s^1 \subset
  w.B$.  Let $\alpha \in B_s^1$.  Since $w.B$ is a
  $\sigma$-fundamental system of $R$, there exist integers
  $(n_{\gamma})_{\gamma\in w.B}$, of the same sign, such
  that $\alpha=\sum_{\gamma\in w.B}n_{\gamma}\gamma$ and
  $\alpha''=\sum_{\gamma\in w.B^1}n_{\gamma}\gamma''$.  As
  $\alpha''\in w.B''$, the $n_{\gamma}$'s must be positive
  and there exists a unique $\beta\in w.B^1$ such that
  $\alpha''=\beta''$, $n_{\beta}=1$, $n_{\gamma}=0$ for
  $\gamma\in w.B^1\smallsetminus\{\beta\}$.  One then gets
  $\beta=\alpha-\sum_{\gamma\in w.B^0=B^0_{s}} n_{\gamma}
  \gamma$, hence $\beta\in R_{s}$.  But $B_{s}$ is a
  fundamental system of $R_{s}$, thus the previous
  decomposition of $\beta$ as a sum of positive and negative
  elements of $B_{s}$ forces $n_{\gamma}=0$ for $\gamma\in
  B^0_{s}$. Therefore $\alpha=\beta\in w.B$, as desired.
  
  Pick $\dot{w} \in K$ such that $\dot{w}.s=w.s$, see
  Remark~\ref{Klift}(2); replacing $\g^s$ by
  $\g^{\dot{w}.s}$ we may assume that $w=\Id$ and
  $B_{s}\subset B$.  Define $t\in \h_{\Q}$ by the
  conditions: $\alpha(t)=0$ for $\alpha\in B_{s}$ and
  $\beta(t)=1$ for $\beta\in B\smallsetminus B_{s}$.  Then,
  $t \in \bigcap_{\varphi \in B^0 \cup B^3} \ker \varphi =
  \af_{\Q}$ (cf.~Remark~\ref{carac}). Finally, since $B_{s}$
  is a fundamental system of $R_{s}$, it is easily seen that
  $\g^t=\g^s$.
\end{proof}

From the previous proposition one deduces the announced
result:

\begin{cor}\label{carac2}
  The property~{\rm (L)} is equivalent to: ``Two standard
  Levi factors arising from $\pp$ are $G$-conjugate if, and
  only if, they are $K$-conjugate''.
\end{cor}

\begin{Rq}\label{rqthmL} 
 Assume that there is no arrow in the Satake
  diagram of $(\g,\kk)$. 
  Then $\overline{B^{2}}=\emptyset$.  
  Let $s\in \pp$ be a semisimple element.
  By Proposition~\ref{levistandard} we may assume that $\g^s=\g^t$ 
  with $t \in \af_\Q$ is standard. Then, obviously, $B^0 \subset
  I_{\g^s}$ and one deduces from  the characterization
  of $\af$ given in Remark~\ref{carac} that $\cc_{\g}(\g^s) \subset
  \pp$. 
Hence, the center of any Levi arising from $\pp$ is wholly included in $\pp$ in this case.
\end{Rq}


\begin{thm}\label{L} 
Every reductive symmetric Lie algebra satisfies
  the property {\rm (L)}.
\end{thm}

\begin{proof}

Assume that $\lf_{1}$ and $\lf_{2}$ are two standard
  $G$-conjugate Levi factors arising from $\pp$ such that
  \begin{equation}B^{0}\cup\overline{B^{2}}\subset I_{\lf_{1}}.\label{noarrow}\end{equation}
  The characterization
  of $\af$ given in Remark~\ref{carac} yields $\g^{\lf_{1}}\subset \af \subset \pp$.  Let
  $s\in(\g^{\lf_{2}})^{\bullet}\cap\pp$, hence $\g^{s}=\lf_{2}$ and, by
  hypothesis, there exists $g\in G$ such that $g.s\in
  (\g^{\lf_{1}})^{\bullet}\subset\pp$.
  Proposition~\ref{Korbss} then implies the existence of
  $k\in K$ such that $g.s=k.s$, thus: $\lf_{1}
  =\g^{k.s}=k.\lf_{2}$.

 When there is no arrow in the Satake diagram of $(\g,\kk)$, \eqref{noarrow} is satisfied (cf. Remark \ref{rqthmL}).
It then follows from the previous discussion that property~{\rm (L)} is satisfied in this case.
  
 In the other cases,
 let $\g^{s_i}$, $s_i \in \af_{\Q}$,
  $i=1,2$, be two standard Levi factors arising from
  $\pp$. Observe first that Proposition~\ref{Korbss1}(ii)
  yields: \begin{itemize}\item$\g^{s_1}, \g^{s_2}$ are $G$-conjugate $\iff$
  $\g^{s_1}, \g^{s_2}$ are $W$-conjugate, \item $\g^{s_1},
  \g^{s_2}$ are $K$-conjugate $\iff$ $\g^{s_1}, \g^{s_2}$
  are $W_{\sigma}$-conjugate.\end{itemize}
  Let $B$ be a $\sigma$-fundamental system; denote by $\Phi$
  the set of all subsets of $B$ which contain all black
  nodes and which are stable under arrows.  Observe that $E
  \in \Phi$ is equivalent to $E= I_{\lf}$ for some standard
  Levi factor $\lf$ arising from $\pp$. Therefore, by the
  previous remark, we need to show that two elements of
  $\Phi$ are $W$-conjugate if and only if they are
  $W_\sigma$-conjugate.
 
  For $E \in \Phi$ we define a subset $\phi(E)$ of $B''$,
  the fundamental system of the restricted root system $S$,
  by setting $\phi(E):=\{\alpha'' : \alpha \in E\} \sminus
  \{0\}$.  It is easy to see that $\phi$ defines a bijection
  from $\Phi$ onto $\Phi''$, the set of all subsets of
  $B''$, and that two elements of $\Phi$ are
  $W_{\sigma}$-conjugate if and only if their images by
  $\phi$ are $W_{S}$-conjugate.  By abuse of notation, we
  denote by $\Phi/W$ and $\Phi/W_\sigma$ resp.~$\Phi''/W_S$,
  the set of orbits under $W$ and $W_\sigma$, resp.~$W_S$,
  of elements of $\Phi$, resp.~$\Phi''$. Since $W_\sigma
  \subset W$, there exists a natural surjection $\pi$ from
  $\Phi/W_\sigma$ onto $\Phi/W$, hence $\#\Phi/W \le
  \#\Phi''/W_S = \#\Phi/W_\sigma$, and we need to show that
  $\pi$ is bijective. We have remarked above that
  $\phi^{-1}$ yields a bijection between $\Phi''/W_S$ and
  $\Phi/W_\sigma$.  Let $\delta : \Phi''/W_S \to \Phi/W$ be
  the surjection induced by $\pi \circ \phi^{-1}$. It
  remains to show that $\delta$ is injective, or,
  equivalently, that $\# \Phi/W \ge \# \Phi''/W_S$.

  When $(\g,\theta)$ is of type~0 there is an
  obvious bijection between $W$-conjugacy classes of
  elements $\Phi$ and $W_{S}$-conjugacy classes in $\Phi''$. 
  In the other types, the description of $\phi$, $\Phi$ and $\Phi''$ can be
  deduced from~\cite[p.~532]{He1}.  The $W$-conjugacy classes of subsets of $B$ are
  given in \cite[p.~5]{BC} (cf.~\cite[Theorem 5.4]{Dy} for the original
  classification). 
  Using these results, it is then easy to
  make a case by case comparison of $\Phi/W$ and
  $\Phi''/W_S$ and prove that they are in one-to-one correspondence.  
For example when $(\g,\kk)$ is irreducible of type EIII, one finds that
  $\Phi/W=\{\text{E}_{6},\text{A}_{5}, \text{D}_{4},
  \text{A}_{3}\}$ and
  $\Phi''/W_S=\{\text{B}_{2},\text{B}_{1},\text{A}_{1},
  \emptyset\}$.  In case EII, one easily sees that $\#
  \Phi/W=\#\Phi''/W_S=12$. One can deal with cases DI, DIII and AIII in the same way.

  Since $\g$ is a direct product of irreducible symmetric
  Lie algebras and the only irreducible Lie algebra whose Satake diagram has arrows 
are of type 0 or of type {\rm AIII, DI, DIII, EII, EIII}, property~(L) follows in the general case.
\end{proof}

\subsection{Jordan $K$-classes}
\label{JKclass}
Let $(\g,\kk)$ be a reductive symmetric Lie algebra. We
adopt the notation of \S\ref{Gclass} and
Definition~\ref{defJK}. Observe the following easy result:

\begin{lm}{\label{Jclass}}
  The intersection of a $J_{G}$-class with $\pp$ is either
  empty or the union of $J_{K}$-classes it contains.
\end{lm}

\begin{proof}
  Let $J$ be a Jordan $G$-class intersecting $\pp$ and
  $x=s+n\in J\cap\pp$.  Then $J_{K}(x)=K.(\cpps+n)\subset
  G.(\cggs+n)=J_{G}(x)$.
\end{proof}

In Lemma~\ref{Jclass2} we fix a $J_{G}$-class $J$ such that
$J\cap\pp\neq \emptyset$, and an element $x=s+n\in J \cap
\pp$. Let $\lf:=\g^s$ and $L :=G^s \subset G$ be the
associated Levi factors.  Observe that:
\begin{equation}\label{LGs}
  L  = Z_{G}(\cggs).  
\end{equation}
Then, $(\g^s,\kk^s)$ is a symmetric pair and $K_{L}:=(K\cap
L)^{\circ} \subset K^s$ acts naturally on $\pp^s$.  Denote
by $\Od_{1}$ the orbit $L.n\in\lf$, so that $(\lf,\Od_{1})$
is a datum of $J$.  Let $\Od_{i} \subset \lf$ ($i\gnq 1$) be
the $L$-orbits (if they exist) different from $\Od_1$ such
that $(\lf,\Od_{i})$ is a datum of $J$.  Define nilpotent
$K_{L}$-orbits in $\pp^s$ by
\[
\Od_{i}\cap\pp^s=\bigcup_{j} \Od_{i}^j, \quad \Od_{i}^j =
K_{L}.n^{j}_{i}.
\]

\begin{lm}\label{Jclass2}
{\rm (i)} One has
$J\cap\pp=\bigcup_{i,j}K.(\cpps+n^{j}_{i})$.
\\
{\rm (ii)} Any $J_{K}$-class contained in $J\cap\pp$ has
dimension $\dim \cppsnb +\dim K.x$.
\end{lm}

\begin{proof}
  Let $y=s'+n' \in J \cap \pp$. Since $x$ and $y$ belong to
  the same $J_{G}$-class, $\g^{s'}$ is $G$-conjugate to
  $\g^s$ \cite[39.1.3]{TY}. By Property~(L), see
  Theorem~\ref{L}, the subalgebra $\g^{s'}$ is then
  $K$-conjugate to $\g^{s}$.  We can therefore assume that
  $s'\in \cpps$. It follows that $n'$ belongs to one of the
  orbits $K_{L}.n^{j}_{i}$, hence $J_K(y)=
  K.(\cpps+n^{j}_{i}) \subset J \cap \pp$.
  \\
  By \cite[39.5.8]{TY} one knows that $\dim J_{K}(y)=\dim
  K.y+ \dim \cppsnb = \dim K.x+ \dim \cppsnb = \dim
  J_K(x)$. This proves (i) and (ii).
\end{proof}

Note that the union in Lemma~\ref{Jclass2}(i) is not
necessarily a disjoint union.

\begin{lm}\label{GJ}
{\rm (i)}  Let $g\in G$ and a semisimple element $s \in \pp$ be such
that $g.s\in\pp$; then $g.\cppsnb\subset\pp$.
\\
{\rm (ii)} For $x,y\in\pp$ such that $G.x=G.y$, one has
$G.J_{K}(x)=G.J_{K}(y).$
\end{lm}

\begin{proof}
  (i) By Lemma~\ref{Korbss} there exists $k\in K$ such that
  $k.(g.s)=s$, hence $kg\in L =Z_{G}(\cggs)$
  (see~\eqref{LGs}) and $kg.\cppsnb=\cppsnb$.  This gives
  $g.\cppsnb=k^{-1}.\cppsnb\subset\pp$.
  \\
  (ii) By Lemma~\ref{Korbss}, again, we may assume that
  $x=s+n$ and $y=s+n'$. Then, $J_{K}(x)=K.(\cpps+n)$ and
  $J_{K}(y)=K.(\cpps+n')$.  Write $y=g.x$, $g \in G$; from
  \eqref{LGs} it follows that $g.(s'+n)=s'+n'$ for all
  $s'\in\cpps$.
\end{proof}

We can now describe the intersection of a $J_G$-class with
$\pp$.

\begin{thm}\label{compirr}
  Let $J$ be a Jordan $G$-class. The variety $J \cap \pp$ is
  smooth. The $J_K$-classes contained in $J \cap \pp$ are
  its (pairwise disjoint and smooth) irreducible components.
\end{thm}

\begin{proof} 
  We may obviously assume that $J \cap \pp \nempty$; pick $x
  \in J \cap \pp$.  Recall \cite{Bro} that $J$ is smooth and
  that the tangent space $T_x J$ is equal to $[x,\g]\oplus
  \cggsnb$, see \cite[39.2.8, 39.2.9]{TY}.  By
  \cite[39.5.5]{TY} there exists a dominant morphism $\mu: K
  \times \cpp{x} \to J_K(x)$, $(k,u) \mapsto k.u$. Therefore
  $d_{(\Id,x)}\mu(\kk \times \cc_{\pp}(\pp^x)) =
  [x,\kk]\oplus \cppsnb$ (cf.~\cite[39.5.7]{TY}) is a
  subspace of the tangent space $T_xJ_K(x)$, and we then
  obtain:
$$
T_x(J\cap\pp)\subset T_xJ \cap \pp =
([x,\g]\oplus\cggsnb)\cap\pp=[x,\kk]\oplus\cppsnb \subset
T_xJ_K(x)\subset T_x(J\cap\pp).
$$ 
Thus $T_x(J\cap\pp)=T_xJ_K(x)$ has dimension $\dim
J_K(x)=\dim \cppsnb +\dim K.x$.  By Lemma~\ref{Jclass2}(ii),
this dimension does not depend on the element $x$ chosen in
$J \cap \pp$.  Therefore $J_K(x)$, $J\cap\pp$ are smooth and
each element of $J \cap \pp$ belongs to a unique irreducible
component (see, for example, \cite[17.1.3]{TY}).  Then,
Lemma~\ref{Jclass} yields the desired result.
\end{proof}

The smoothness of $J \cap \pp$ can be deduced from a
general result that we now recall, see, for example,
\cite[Proposition 1.3]{Iv} or \cite[6.5, Corollary]{PV}.

\begin{thm} \label{Iv} Let $\Gamma$ be a linear reductive group
  acting on a smooth variety $X$. Then the subvariety of
  fixed points $X^\Gamma:=\{x\in X\mid \Gamma.x=x\}$ is
  smooth, and $T_xX^\Gamma= (T_x X)^\Gamma$ for all $x \in
  X^\Gamma$.
\end{thm}

This theorem can be applied to a $J_G$-class $J$ as
follows. Let
\[
\Gamma :=\{\Id,\tilde{\theta}\} \subset \GL(\g)
\]
be the group, of order two, generated by $\tilde{\theta}:= -
\theta$ (thus $\tilde{\theta}$ is an anti-automorphism of
$\g$).  Now, we can note \cite[39.1.7]{TY} that $J =
J_G(x)=G.{\cc_{\g}(\g^x)^{\bullet}}$. From the definition of
a Jordan class, or this description, it follows that $J$
is stable under the $\K^\times$-action $y \mapsto \lambda
y$, $\lambda \in \K^\times$ so, when $J \cap \pp \nempty$,
we have $\tilde{\theta}(J)= \theta(J)= J$.  Therefore, the
group $\Gamma$ acts on the smooth variety $J$ and we get
from Theorem~\ref{Iv} that $J^\Gamma= J \cap \pp$ is
smooth. This provides another proof of
Theorem~\ref{compirr} (see the four last lines in the proof
of that theorem).

\subsection{$K$-sheets}
\label{Ksheet}
We continue with the same notation. Fix a $G$-sheet
$S=S_{G}\subset\g^{(2m)}$, $m\in\N$. Since each $K$-sheet is an irreducible component of $\pp^{(m')}\subset \g^{(2m')}$ for some $m'\in \N$, we aim to describe the
irreducible components of $S_G \cap \pp$. 
In this way, we will get informations on $K$-sheets 
One important remark is the following
\begin{lm} If $S_G\cap\pp\neq \emptyset$ then the
unique nilpotent orbit $\Od$ contained in $S_G$ intersects $\pp$.\label{nilporbp}
\end{lm}
\begin{proof}
Let $x\in S_G\cap\pp$. It follows from \cite[38.6.9]{TY} that there exists a nilpotent element $n\in\pp$ such that $n\in \overline{K.(\K x)}^{\bullet}$.
Since $K.(\K^{\times} x)\subset S_G$, we get that $n\in S_G$ and $n\in \Od\cap\pp$.
\end{proof}
The description of the
irreducible components of $S_G \cap \pp$ will be given in
terms of the $K$-orbits contained in $\Od$, see
Theorem~\ref{equidim}.

We first want to prove that when $S$ is smooth, and
$(\g,\theta)$ has no irreducible factor of type~0, the
intersection $S \cap \pp$ (which can be empty) is also
smooth. To obtain this result we will apply
Theorem~\ref{Iv}, as in the case of a Jordan $G$-class. We
adopt the notation of the end of the previous subsection, in
particular we set $\Gamma:=\{\Id,\ttheta=-\theta\}$.  Observe
that $S$ is stable under the $\K^\times$-action, thus
$\ttheta(S) = \theta(S)$; but, contrary to the case of a
Jordan class, the stability of $S$ under $\Gamma$ requires
some hypothesis, even in the case where $S \cap \pp
\nempty$.

We begin with the following technical
result which is a reformulation of \cite[Lemma 4.5]{Boh}. 
Its proof is based on a case by case study and goes along the same lines as \cite[\S3.9]{Boh}.
Recall \cite[7.1]{CM} that a nilpotent orbit
$\mathcal{O}$ is called \emph{rigid} if it can not be
obtained by induction of a proper parabolic subalgebra of
$\g$; equivalently, when $\g$ is semisimple, $\mathcal{O}$
is rigid if $\mathcal{O}$ is a $G$-sheet,
cf.~\cite[\S4]{Boh}. 

\begin{lm} \label{autodata} Let $\lf$ be a Levi factor of a
  simple Lie algebra $\g$ and $\Od$ be a rigid nilpotent
  orbit of $\lf$. Then, $\tau(\Od)=\Od$ for all $\tau\in
  \Aut(\lf)$.
\end{lm}

The next lemma ensures that when $\g$ is simple, $S \cup \theta(S)$ inherites its smoothness from $S$.

\begin{lm}
  \label{autosheet}
  Let $\g$ be a simple Lie algebra. If $\g$ is not of type
  \emph{D}, then $\theta(S)=S$.
  \\
  If $\g$ is of type \emph{D}, one has either $\theta(S)=S$
  or $S \cap\theta(S)=\emptyset$.
\end{lm}

\begin{proof}
  Let $J_{1}$ be the dense Jordan class contained in $S$ and
  let $(\lf,\Od)$ be a datum of $J_{1}$.  Then, the dense
  Jordan class $J_{2}$ in the sheet $\theta(S)$ has datum
  $(\theta(\lf),\theta(\Od))$.
  \\
  If $\g$ is of type different from D or E$_{7}$, it follows
  from the classification of Levi factors in \cite[Theorem
  5.4]{Dy} that $\theta(\lf)$ is $G$-conjugate to $\lf$
  (cf. also \cite[Proposition 6.3]{BC}). In these cases we
  can therefore assume that $\theta(\lf)=\lf$, and
  Lemma~\ref{autodata} yields $\theta(\Od)=\Od$.  Thus,
  $J_{1}=J_{2}$ and $\theta(S)=S$.
  \\
  If $\g$ is of type E$_{7}$, there exists no outer
  automorphism of $\g$ so $\theta(S)\subseteq G.S=S$.
  \\
  Suppose that $\g$ is of type D. If $\lf$ and $\theta(\lf)$
  are $G$-conjugate, the previous argument applies and one
  gets $\theta(S)=S$. Otherwise, \cite[Corollary~3.15]{IH}
  implies that $S\cap\theta(S) =
  \overline{J_{1}}^{\bullet}\cap \overline{J_{2}}^{\bullet}
  =\emptyset$.
\end{proof}

We can now prove the desired result:

\begin{prop}
  \label{sheetsmoothness} {\rm (i)} Let $(\g,\theta)$ be a
  reductive symmetric Lie algebra which has no irreducible
  factor of type \emph{0}. If $S$ is a smooth $G$-sheet then the
  intersection $S\cap\pp$ is smooth.
  \\
  {\rm (ii)} Let $(\g,\theta)$ be a symmetric Lie algebra
  and $S'$ be a $K$-sheet contained in a smooth $G$-sheet
  $S$.  Then $S'$ is smooth.\\
  {\rm (iii)} Under the assumptions of (ii), $S'$ is a union of Jordan $K$-classes.
\end{prop}

\begin{proof}
  Decompose the symmetric algebra $(\g,\theta)$ as
  $(\zz(\g),\theta_{\mid{\zz(\g)}}) \oplus \bigoplus_i
  (\g_{i},\theta_{\mid {\g_i}})$ where each
  $(\g_i,\theta_{\mid {\g_i}})$ is an irreducible factor
  (see the beginning of this section).

  \noindent (i) We want to apply Theorem~\ref{Iv} with
  $\Gamma=\{\Id,\ttheta= - \theta\}$ and $X:= S \cup
  \theta(S) \subset\g$. Note that $X^\Gamma = (S \cap \pp)
  \cup (\theta(S) \cap \pp)$ and that $\theta(S)$ is smooth.
  \\
  If $\g$ is simple, Lemma~\ref{autosheet} yields that $X=S$
  or $S \sqcup \theta(S)$ (in type D) is smooth; therefore
  $X^\Gamma$, and consequently $S \cap \pp$, is smooth.
  Suppose that $\g$ is not simple. By hypothesis, each
  $\g_i$ is simple and the result then follows from
  Corollary~\ref{sheetdecomposition}.

  \noindent (ii)
The $K$-sheet $S'$ is an irreducible component of a $\pp^{(m)}$ for some $m\in\N$.
Since $S'\subset S\cap\pp\subset\g^{(2m)}\cap\pp=\pp^{(m)}$, $S'$ is an irreducible component of $S\cap\pp$.
It is therefore sufficient to prove that $S\cap\pp$ is smooth.\\
 From (i), we are reduced to the case of type 0, i.e.,
  $\g=\g^1\oplus\g^2$ with
  $\theta: \g^1 \isomto \g^2$.  From the results of
  \S\ref{type0} it follows that there exists a $G^1$-sheet
  $S^{1}\subset \g^1$ such that $S' =\{x-\theta(x)\mid x\in
  S^1\}$. Then $S = S^1\times\theta(S^1)$, which is smooth
  if and only if $S^1$ is smooth. As $S^1$ is isomorphic to
  $S'$, one gets the desired result.

  \noindent (iii) 
Note that
$$ S\cap\pp=\bigcup_{J_G\subset S} J_G\cap\pp$$ where $J_G$ runs in the Jordan $G$-classes included in $S$. 
Then, Lemma \ref{Jclass} implies that, with obvious notations,
$$S\cap\pp=\bigcup_{J_K\subset S\cap\pp} J_K.$$
By (ii), $S\cap\pp$ is smooth and  therefore is the disjoint union of its irreducible components. 
In particular, $S'$ is a union of Jordan $K$-classes since those are irreducible subvarieties of $S\cap\pp$.
\end{proof}

\begin{Rqs}
  \label{slsheets}
  (1) The sheets in a classical Lie algebra are smooth, see
  Theorem~\ref{slicesmoothness}. Therefore if $(\g,\kk)$ is
  a symmetric Lie algebra with $\g$ of classical type, 
  Proposition~\ref{sheetsmoothness}
  implies that its $K$-sheets are smooth and union of Jordan $K$-classes.
  \\
  (2) When $\g = \gl_N$, case which will be studied in
  details in Section~\ref{AAA}, the smoothness of $S_G \cap
  \pp$ can been explained in different (equivalent)
  terms. Indeed, recall first that, if $\g = \gl_N$, a
  nilpotent orbit is contained in a unique $G$-sheet,
  cf.~Remark~\ref{XsXss}. Assume that the sheet $S=S_G$
  intersects $\pp$ and let $\Od=G.e$ be the nilpotent orbit
  contained in $S$. Then, since we may assume that $e \in
  \pp$, it follows from $G.\theta(e) =G.(- e)= G.e \subset
  \theta(S) \cap S$ that $\theta(S) = S$. Therefore, the
  group $\Gamma$ acts on $S$ and $S^\Gamma = S \cap \pp$ is
  smooth.\\
 (3) Micha\"el Le Barbier has recently proved that \emph{The closure of a Jordan $K$-class is a union of Jordan $K$-classes}, see \cite[Theorem~B.1]{Le}. 
  Proposition~\ref{sheetsmoothness}(iii) may be seen as a straightforward consequence of this result.

\end{Rqs}

Assume that the sheet $S_G$ intersects $\pp$, pick
$e\in\Od\cap\pp$ and set
\[
\Od_{e}:=K.e \subset \Od \cap \pp.
\]
Denote by $\sS=(e,h,f)$ a normal \Striplet containing
$e$. We are going to apply the results recalled in
\S\ref{Katsylo} to various triples deduced from $\sS$.
Recall from Remarks~\ref{remtriples} that these results hold
for any such \Striplet.
 
Let $\mathsf{Z}\subset G$ be a subset such that
$\{g.e\}_{g\in \mathsf{Z}}$ is a set of representatives of
the $K$-orbits contained in $\Od\cap\pp$; we assume that
$\Id\in \mathsf{Z}$. Observe that, since the {\Striplet}s
containing $g.e$ are conjugate, we may also assume that
$g.\sS:=(g.e,g.h,g.f)$ is a normal \Striplet for all $g\in
\mathsf{Z}$. Recall that $X(S_{G},g.\sS)$ is defined by
\[
g.e + X(S_{G},g.\sS)= S_G \cap (g.e + \g^{g.f}) = g.(S_G
\cap (e + \g^{f})) = g.(e + X(S_{G},\sS)).
\]
(Hence $X(S_{G},g.\sS) = g.X(S_{G},\sS)$.)  Set
\begin{equation}
  \label{Xpp}
  X_{\pp}(S_{G},g.\sS):= X(S_{G},g.\sS)\cap\pp.
\end{equation}

\begin{Rq}\label{dimGY}
  Recall that $S \subset \g^{(2m)}$. Let $ Y\neq \emptyset$ be a subvariety of $g.e+\XP(S_{G},g.\sS)$; then, each $G$-orbit
  (resp.~$K$-orbit) of an element of $Y$ has dimension $\dim
  G.e=2m$ (resp.~$\dim K.e=m$).  Lemma~\ref{Gamma} implies
  that the fibers of the morphisms $G\times Y\rightarrow
  G.Y$ and $K\times Y\rightarrow K.Y$ are of respective
  dimension $\dim G^{e}$ and $\dim K^{e}$. Then, by
  \cite[15.5.5]{TY}, we get that $\dim G.Y=\dim Y+2m$ and
  $\dim K.Y=\dim Y+m$.
\end{Rq}

We now introduce some conditions which will be sufficient to
give a description of the irreducible components of $S_G
\cap \pp$ in terms of the $X_{\pp}(S_{G},g.\sS)$, see
Theorem~\ref{equidim}.

Recall that $S_G = G.(e + X(S_{G},\sS))$. The first
condition ensures that $e+\XP$ is large enough:
\begin{equation}
  \label{heart}
  G.(g.e+\XP(S_{G},g.\sS))=G.(S_{G}\cap\pp)\mbox{ 
    for all }g\in \mathsf{Z}.
  \tag{$\heartsuit$}
\end{equation}
The condition~\eqref{heart} was established for pairs of
type 0 in Proposition~\ref{orbX0}, and we will see that it
also holds for all symmetric pairs when $\g=\gl_{N}$
(cf.~Theorem~\ref{unif}).  Set:
$$
A(g.e):=G^{g.e}/(G^{g.e})^\circ.
$$ 
By Theorem~\ref{psiquotient} the Slodowy slice $g.e +
X(S_{G},g.\sS)$ provides the geometric quotient
\[
\psi_{S_G,g.\sS} : S_G \longrightarrow (g.e +
X(S_{G},g.\sS))/A(g.e)
\]
and we will be interested in some cases where the following
property is satisfied:
\begin{equation}
  \label{sstar}
  G^{e}\mbox{ is connected. }
  \tag{*}
\end{equation}  
Recall that~\eqref{sstar} is true when $\g= \gl_N$
(see~Lemma~\ref{GeIH}).  Clearly, \eqref{sstar} implies that
$g.e+X(S_{G},g.\sS)$ is the geometric quotient of $S_{G}$.
In this case, the restriction of $\psi_{S_{G},g.\sS}$ to the
subset
$\bigl(g.e+\bigoplus_{i\leqslant0}\g(2i,g.h)\bigr)\cap
S_{G}$ is given by the map $\varepsilon_{S_{G},g.\sS}$
constructed in Lemma~\ref{epsilon}, and if
hypothesis~\eqref{heart} is also satisfied, one has:
$\psi_{S_{G},g.\sS}(S_{G}\cap\pp)=g.e+\XP(S_{G},g.\sS)$.

Let $J_{1}$ be a $J_{K}$-class contained in
$S_{G}\cap\pp$. As $J_{1}$ is $K$-stable, the dimension of
$J_{1}\cap (g.e+\pp^{g.f})$ does not depend on the
representative element $g.\sS$ of the orbit $K.g.\sS$. Since
$K$-orbits of normal {\Striplet}s are in one to one
correspondence with $K$-orbits of their nilpositive parts
(i.e. the first element of such an \Striplet),
we may introduce the following definition.

\begin{defi}\label{flatten}
  Let $g \in \mathsf{Z}$. A $J_{K}$-class $J_1$ contained in
  $S_{G}$ is said to be \emph{well-behaved with respect to
    $\Od_{g.e}=K.g.e$}, if:
  \begin{equation}\label{flatteneq}
    \dim J_{1}\cap  (g.e+\pp^{g.f})=\dim J_{1}- m.
  \end{equation} 
\end{defi}

\begin{Rq}
  \label{flattenremark}
  It follows from Lemma~\ref{Jclass2}(ii) that a
  $J_{K}$-class $J_{1}=K.(\cpps+n)$ is well-behaved
  w.r.t.~$\Od_{g.e}$ if and only if $Y
  =J_{1}\cap(g.e+\pp^{g.f})$ satisfies $\dim Y=\dim \cppsnb
  \, (= \dim J_1 -m)$.  By Remark~\ref{dimGY} this is also
  equivalent to $\dim K.Y=\dim J_{1}$, which is in turn
  equivalent to $J_{1}\subset\overline{K.Y}$. In this case
  one has $J_{1}\subset
  \overline{K.(g.e+\XP(S_{G},g.\sS))}$, property which will
  be of importance for the description of $S_{G}\cap\pp$.
\end{Rq}

The following lemma shows that, assuming \eqref{heart},
well-behaved $J_{K}$-classes exist.

\begin{lm}\label{Jclass4}
  Let $J$ be a $J_{G}$-class contained in $S_{G}$ such that
  $J\cap\pp\neq\emptyset$. Fix $g\in \mathsf{Z}$ and set
  $\psi := \psi_{S_G,g.\sS}$.  Assume that the property
  \eqref{heart} is satisfied.
  \\
  {\rm (i)} Let $J_{1} \subset J\cap\pp$ be a $J_{K}$-class.
  There exists a subvariety $Y\subset g.e+\XP(S_{G},g.\sS)$
  such that: $Y$ is irreducible and $\psi(Y)$ is dense in
  ${\psi(J_{1})}$.  Moreover, if $Y \subset
  g.e+\XP(S_{G},g.\sS)$ is maximal for these two properties,
  then $\psi(Y)=\psi(J_{1})$ and $J_{2} :=\overline{K.Y}\cap
  J$ is a $J_{K}$-class (contained in $J$) which is
  well-behaved~w.r.t.~$\Od_{g.e}$.
  \\
  {\rm (ii)} The class $J_{1}$ is well-behaved w.r.t.
  $\Od_{g.e}$ if and only if one can find $Y$, as in {\rm
    (i)}, such that $J_{1}=\overline{K.Y}\cap J$.
  \\
  {\rm (iii)} If \eqref{sstar} holds, there exists a unique
  maximal $Y$ as in {\rm (i)}, namely
  $Y=\psi_{S_{G},g.\sS}(J_{1})$, thus
  $J_{2}=\overline{K.\psi_{S_{G},g.\sS}(J_{1})} \cap J$.
\end{lm}

\begin{proof}
  In order to simplify the notation, we suppose that $g=\Id$
  and we set $X:=X(S_{G},\sS)$, $X_\pp := X \cap \pp$, $\psi :=
  \psi_{S_{G},\sS}$, $S:= S_G$, etc.

  \noindent (i) Consider the following commutative diagram
  \begin{equation*}
    \xymatrix @C=0.5pc @R=1.5pc 
    {
      e+\XP    \ar[dr]^{\gamma_\pp} \ar[r]^i
      & e+X \ar[d]^\gamma \\
      & (e+X)/A = \psi(S)
    }
  \end{equation*} 
  where $i$ is the natural closed embedding and $\gamma$ is
  the quotient morphism, see~\eqref{mapgamma}.  Observe
  that, the group $A$ being finite, the morphisms $\gamma$
  and $\gamma_\pp$ are finite, hence closed. Moreover,
  \eqref{heart} implies that $\im(\gamma_\pp) = \psi(S \cap
  \pp)$.  Let $Y'$ be any irreducible component of
  $\gamma_\pp^{-1}\bigl(\overline{\psi(J_{1})}\bigr)$
  dominating $\overline{\psi(J_{1})} \subset (e+X)/A$ and
  set:
  \[
  Y:=\gamma_\pp^{-1}(\psi(J_{1}))\cap Y'.
  \]
  Then $Y \subset J$ is a dense irreducible subset of $Y'$
  such that $\psi(Y)= \gamma_\pp(Y) = \psi(J_1)$.
  Since the fibers of $\psi$ are of dimension $m$ and
  $\gamma_\pp$ is finite, one has $\dim Y=\dim J_{1}-
  m$. Set
  \[
  J_2 := \overline{K.Y}\cap J.
  \]
  As $K.Y \subset J_2 \subset J \cap \pp$, we see that $J_{2}$ is a
  closed irreducible subset of $J \cap \pp$ of dimension
  $\dim K.Y = \dim Y + m = \dim J \cap \pp$
  (cf.~Remark~\ref{dimGY}).  One obtains from
  Theorem~\ref{compirr} that $J_2$ is a $J_K$-class, which
  is well-behaved w.r.t.~$\Od_e$ (recall that $J_2 \subset
  \overline{K.Y}$).
  \\
  Suppose now that $Y_1 \subset e + X_\pp$ is maximal for
  the properties: $Y_1$ irreducible and ${\psi(Y_1)}$ dense
  in ${\psi(J_1)}$. Observe that the closure $Y_1'$ of
  ${Y_1}$ inside $e+\XP$ is irreducible, and
  $\gamma_\pp(Y'_1) = \overline{\gamma_\pp(Y_1)}=
  \overline{\psi(J_1)}$. The argument of the previous
  paragraph, together with the maximality of $Y_1$, implies
  that $Y_1= \gamma_\pp^{-1}(\psi(J_{1})) \cap Y_1'$. As
  above, we then get that $\overline{K.Y_1} \cap J$ is a
  well-behaved $J_K$-class contained in $J \cap \pp$.

  \noindent (ii) Set $Y_1:= J_1 \cap (e + X_\pp)$ and suppose
  that $J_1$ is well-behaved w.r.t.~$\Od_e$, thus $\dim J_1=
  \dim Y_1 + m$. Let $Y_2 \subset Y_1$ be an irreducible
  component of maximal dimension; since $\gamma_\pp$ is
  finite, one has $\dim \gamma_\pp(Y_2) = \dim Y_1 = \dim
  \psi(J_1)$, hence $\psi(Y_2) $ is dense in $\psi(J_1)$. We
  then deduce from~(i) that $J_2 := \overline{K.Y_2} \cap J$
  is a $J_K$-class; since $Y_2 \subset J_2 \cap J_1$, it
  follows that $J_1=J_2$ is well-behaved w.r.t.~$\Od_e$.
  The converse is clear.

  \noindent (iii) Here, $\gamma_\pp : e + \XP \longisomto
  \psi( S \cap \pp)$ is the identity; thus $Y'=
  \overline{\psi(J_1)}$ and $Y= \psi(J_1)$.
\end{proof}

\begin{Rqs}\label{SKO}
  (1) In part~(i) of the previous lemma, the $J_K$-class
  $J_2 \, (\subset J \subset S_G)$ is contained in the
  following variety:
  \begin{equation}
    \label{SKge}
    S_{K}(S_{G},g.\sS):=\overline{K.(g.e+\XP(S_{G},g.\sS))}^{\bullet}. 
  \end{equation}
  Since $K$-orbits of normal {\Striplet}s are in bijection
  with nilpotent $K$-orbits, $S_{K}(S_{G},g.\sS)$ depends
  only on the sheet $S_{G}$ and the orbit $\Od_{g.e}=K.g.e$.
  Therefore we can write
  \[
  S_{K}(S_{G},g.\sS)= S_{K}(S_{G},\Od_{g.e}).
  \]
  Furthermore when $\g$ is of type A, thanks to
  Remark~\ref{XsXss}, we may also write
  $S_{K}(S_{G},g.\sS)=S_{K}(g.\sS)= S_{K}(\Od_{g.e})$.
  \\
  (2) Under assumption~\eqref{sstar},
  Lemma~\ref{Jclass4}(iii) yields a well defined map
  \[
  J_{1} \mapsto J_{2} := J \cap
  \overline{K.\psi_{S_G,g.\sS}(J_1)}
  \]
  from the set of $J_{K}$-classes contained in
  $S_{G}\cap\pp$ to the set of $J_{K}$-classes contained in
  $S_{K}(S_{G},g.\Od)$.
  \\
  In case A, we will show in Lemma~\ref{clubAI} and
  Lemma~\ref{desired} that each $J_{K}$-class contained in
  $S_{G}\cap\pp$ is in the image of such an map, for
  an appropriate choice of $g\in \mathsf{Z}$.
\end{Rqs}

We now introduce a condition ensuring that the varieties
$S_{K}(S_{G},\Od_{e})$ are irreducible:
\begin{equation}\label{diamond}
  \XP(S_{G},g.\sS)\mbox{ is irreducible for all }g\in
  \mathsf{Z}.
  \tag{$\diamondsuit$}
\end{equation} 

\begin{cor}\label{sheetsheetsheet}\label{sheeet}
  Assume that conditions \eqref{heart} and \eqref{diamond}
  hold. Then, $S_{K}(S_{G},\Od_{g.e})$ is an irreducible
  component of $S_{G}\cap\pp$ of maximal dimension.
\end{cor}

\begin{proof}
  Let $J_{1}$ be a $J_{K}$-class of maximal dimension
  contained in $S_G \cap \pp$ and $J \subset S_G$ be the
  $J_{G}$-class containing $J_1$.  Since \eqref{heart} is
  satisfied, one can find $Y$ as in Lemma~\ref{Jclass4}(i)
  such that $J_{2}:=\overline{K.Y}\cap J$ is a $J_{K}$-class
  contained in $J$.  Then, $J_{2}\subset
  S_{K}(S_{G},\Od_{g.e})\subset S_{G}\cap\pp$ and
  Theorem~\ref{compirr} implies that $\dim J_{2}=\dim
  J_{1}=\dim S_{G}\cap\pp$. Therefore
  $S_{K}(S_{G},\Od_{g.e})=\overline{J_{2}}^{\bullet}$ is an
  irreducible component of $S_{G}\cap\pp$ of maximal
  dimension.
\end{proof}

In view of the previous corollary, it is then natural to
ask: Are all the irreducible components of $S_{G}\cap\pp$ of
the form $S_{K}(S_{G},\Od_{g.e})$?  We introduce the next
additional condition to answer that question:
\begin{equation}\label{club}
  \mbox{For each }J_{K}\mbox{-class }J_{1}\mbox{ in
  }S_{G}\cap\pp \mbox{, there exists 
    $g\in \mathsf{Z}$ such that }
  J_{1}\mbox{ is well-behaved w.r.t. }\Od_{g.e}. 
  \tag{$\clubsuit$}
\end{equation}

\begin{thm}\label{equidim}
  Assume that conditions \eqref{heart}, \eqref{diamond} and
  \eqref{club} are satisfied.\\
(i) The irreducible components of $S_{G}\cap\pp$
  are the $S_{K}(S_G,\Od_{g.e})$ with $g\in
  \mathsf{Z}$. \\
(ii) $S_{G}\cap\pp$ is
  equidimensional.\\
(iii) There exists a unique $J_G$-class $J$ such that
  $S_G \cap \pp = \overline{J \cap \pp}^\bullet$ and, for
  each $g \in \mathsf{Z}$, $S_{K}(S_G,\Od_{g.e})=
  \overline{J_g}^\bullet$ for a unique $J_K$-class $J_g\subset J$. \\
(iv) The map $S_{K}(S_G,\Od_{g.e}) \to J_g$ gives a
  bijection between irreducible components of $S_{G}\cap\pp$
  and the set of $J_K$-classes contained in $J \cap \pp$.
\end{thm}

\begin{proof}
  Write $S_{G}\cap\pp=\bigcup_{J\subset S_{G}} J\cap\pp$,
  where the union is taken over the $J_{G}$-classes $J$
  intersecting~$\pp$.  For any such $J$, $J \cap \pp$ is the
  union of the $J_{K}$-classes it contains
  (cf.~Lemma~\ref{Jclass}), thus~\eqref{club} and
  Lemma~\ref{Jclass4}(ii) imply that
  $S_{G}\cap\pp=\bigcup_{g\in \mathsf{Z}}
  S_{K}(S_{G},\Od_{g.e})$.  Then, apply
  Corollary~\ref{sheeet} to get (i) and (ii).
  \\
  Now, let $J_1$ be a $J_K$-class of maximal dimension
  contained in $S_G \cap \pp$ and denote by $J \subset S_G$
  the $J_G$-class containing $J_1$. Let $g \in \mathsf{Z}$;
  as in the proof of Corollary~\ref{sheeet} one can find a
  $J_K$-class $J_g \subset J \cap \pp$ such that
  $S_{K}(S_G,\Od_{g.e})= \overline{J_g}^\bullet$. It then
  follows from (i) that $S_G \cap \pp =
  \overline{J \cap \pp}^\bullet$. Furthermore, as
  $J_K$-classes are locally closed, $J_g$ is the unique
  dense $J_K$-class in $S_{K}(S_{G},\Od_{g.e})$. This
  implies the unicity of the class $J$ and (iii) follows.  Finally, one deduces  (iv) from (i), (iii) and
  Theorem~\ref{compirr}.
\end{proof}

\section{Type~A}
\label{AAA}

We show in this section that the conditions \eqref{heart},
\eqref{diamond} and \eqref{club}, introduced in
Section~\ref{Ksheet} in order to describe the $K$-sheets of
a reductive (or semisimple, see Corollary~\ref{simpleprod})
symmetric Lie algebra $(\g,\theta)$, are satisfied in
type~A, i.e.~when $\g= \gl_N$ (or $\fsl_N$).

Thereafter, unless otherwise specified, e.g.~in
\ref{involA0}, we set $\g = \gl_N$, $N \in \N^*$, and if
$\theta$ is an involution on $\g$ we adopt the notation of
Section~\ref{SLA} relative to the symmetric pair
$(\g,\theta)= (\g,\kk)$.  The natural action of
$\tilde{G}=\GL_{N}$ on $\g$ factorizes through the adjoint
action to give the surjective morphism:
$$
\rho : \tilde{G}\longrightarrow G \cong
\tilde{G}/\K^{\times}\Id=\mathrm{PGL}_{N}=\mathrm{PSL}_{N}
$$ 
Recall that $\Kt:=\{g\in G\mid g \circ \theta=\theta \circ
g\}$ and $K:=(\Kt)^{\circ}$.  If $H$ is an algebraic subgroup
of $G$ we set:
\begin{equation}
  \label{rho}
  \tilde  H:=\rho^{-1}(H).
\end{equation}
Thus, $H.x=\tilde{H}.x$ for all $x\in\g$.  After recalling
the three different possible types of involutions, we will
establish the three aforementioned conditions:

\eqref{heart} in Theorem~\ref{unif} (types AI, AII) and
Proposition~\ref{unifAIII} (type AIII);

\eqref{diamond} in Remark~\ref{diamondAI} (types AI, AII)
and Remark~\ref{diamondAIII} (type AIII);

\eqref{club} in Corollary~\ref{clubAI} (types AI, AII) and
Proposition~\ref{desired} (type AIII).

\subsection{Involutions in type A}
\label{invol}
We recall below a construction of the involutions on
$\gl_{N} =\gl(V)$.  We will also have to consider the
involution by permutation of factors on
$\gl_{N}\times\gl_{N}$, cf.~\ref{type0}; this case will be
called ``type~A0''.

Recall that the nilpotent orbits in $\g=\gl_{N}$ are in
bijection with the partitions of $N$ and that, to each
partition $\boldsymbol{\mu} =
(\mu_{1}\geqslant\dots\geqslant\mu_{k})$, one associates a
Young diagram having $\mu_{i}$ boxes on the $i$-th row.
 
We fix a $G$-sheet $S_{G} \subset \g$ and an element $e$ in
the nilpotent orbit $\Od\subset S_{G}$. The partition
associated to $e$ is denoted by
\[
\bolda=(\lambda_{1}\geqslant \lambda_{2}\geqslant
\dots\geqslant \lambda_{\delta_{\Od}}).
\]
We adopt the notation introduced in~\ref{settings}; in
particular, the basis $\mathbf{v}$ (in which $e =\sum_i e_i$
has a Jordan normal form, see~\eqref{vbasis}) and the
subalgebras $\qq_{i} \cong \gl_{\lambda_{i}}$, $\qq =
\oplus_i \qq_{i}, \lf, \tf$ are fixed.
 
We want to construct symmetric pairs $(\g,\theta) \equiv
(\g,\kk,\pp) \equiv (\g,\kk)$ such that $e\in \pp$.  These
constructions are inspired by \cite{Oh1,Oh2}.  The
notation being as in \cite{He1,GW}, one obtains three types
of non-isomorphic symmetric pairs: AI, AII and AIII. Recall
that the involution $\theta$ is outer in types AI, AII and
inner in type AIII.
 
The most complicated case is type~AIII, where it is possible
to embed $e$ in several non-isomorphic ways in different
$\pp$'s. These possibilities will be parameterized by
functions $\Phi:[\![1,\delta_{\Od}]\!] \rightarrow \{a,b\}$,
where $a,b$ are different symbols.

\subsubsection{Case A0}
\label{involA0}
Let $\theta$ be the involution on $\g=\gl_{N}\times\gl_{N}$
sending $(x,y)$ to $(y,x)$.  Recall that $\kk=\{(x,x)\mid
x\in \gl_{N}\}\cong \gl_{N}$, $\pp=\{(x,-x)\mid x\in
\gl_{N}\}$. The $\kk$-module $\pp$ is isomorphic to the $\ad
\gl_{N}$-module $\gl_{N}$; thus, $G.y\cap\pp=K.y$ for
$y=(x,-x)\in\pp$. Suppose that $y=(x,-x)$ is nilpotent,
i.e.~$x\in\gl_{N}$ is nilpotent.  The elements $x$ and $-x$
share the same Young diagram
$\boldsymbol\mu=(\mu_{1}\geqslant\dots\geqslant\mu_{k})$ and
the orbit $K.y$ is uniquely determined by $\boldsymbol\mu$.

\subsubsection{Case AI}
\label{involAI}
Let $\chi$ be the nondegenerate symmetric bilinear form on
$V$ defined, in the basis $\mathbf{v}$ (cf.~\ref{settings}), by:
$$
\chi(v_{j}^{(i)},v_{k}^{(l)}):= \left\{\begin{array}{l l} 1 &
    \mbox{ if } l=i \mbox{ and } j+k=\lambda_{i}+1; \\ 0 &
    \mbox{ otherwise.}\end{array}\right.
$$ 
Set
$$
\kk:=\{k\in\g \mid \forall u,v\in V, \, \chi(k.u, v)=-\chi(u,
k.v)\}\cong \so_{N},$$ $$ \pp:=\{p\in\g \mid \forall u,v\in
V, \, \chi(p.u, v)=\chi(u, p.v)\}.
$$
The symmetric Lie algebra $(\g,\kk,\pp)$ is of type~AI with
associated involution $\theta$ on $\g$ having $\kk$
(resp.~$\pp$) as $+1$ (resp. $-1$) eigenspace.  In
particular $\zz(\g)=\K\Id\subset\pp$.
 
In this case, each $(\qq_{i},\kk\cap\qq_{i})$ is a simple
symmetric pair of type~AI isomorphic to $(\gl_{\lambda_{i}},
\so_{\lambda_{i}})$.  Denote by $s_{k}$ the $(k\times
k)$-matrix with entries equal to $1$ on its antidiagonal and
$0$ elsewhere, as in \cite[3.2]{GW}.  The involution
$\theta$ associated to $(\g,\kk,\pp)$ acts on each element
$x\in\qq_{i}$ by
$\theta(x)=-s_{\lambda_{i}}{}^txs_{\lambda_{i}}$ (which is
the opposite of the symmetric matrix of $x$ with respect to
the antidiagonal).
 
The group $\tilde{\Kt}=\rho^{-1}(\Kt)$, cf.~\eqref{rho}, is
a nonconnected group isomorphic to the orthogonal group
$\mathrm{O}_{N}$ and $\Kt\cong \mathrm{O}_N/\{\pm\Id\}$.
Fix $\tilde{\omega}\in \tilde{\Kt}\smallsetminus
(\tilde{\Kt})^{\circ}$, then: $\tilde{\Kt} =
(\tilde{\Kt})^{\circ} \sqcup \tilde{\omega}
(\tilde{\Kt})^{\circ}$, $\Kt=K \cup \omega K$, where
$\omega:=\rho(\tilde{\omega})$.  When $N$ is odd, $\omega\in
K=\Kt\cong \SO_{N}$ and $\Kt$ is connected.  If $N$ is even,
one has $\Kt=K \sqcup \omega K$ and $K\cong \SO_{N}/\{\pm
\Id\}$.

Let $(\g,\kk',\pp')$ be another symmetric Lie algebra of
type~AI, then $\Od\cap\pp'\nonv$ and, moreover, for any
element $e'\in\Od\cap\pp'$ there exists an isomorphism of
symmetric Lie algebras $\tau:(\g,\kk',\pp')\rightarrow
(\g,\kk,\pp)$ such that $\tau(e')=e$ (see \cite[Theorem~3.4]{GW}).

\subsubsection{Case AII}
\label{involAII}
Assume that $\theta'$ is an involution of type AII on $\g$ (i.e. $\kk\cong \spn_N$)
such that $\theta'(e)=-e$; the following condition is then
necessarily satisfied:
$$
\lambda_{2i+1}=\lambda_{2i+2} \ \; \text{for all $i$}.
$$
We therefore assume, in this subsection, that the previous
condition holds. In particular, $N$ is even and we write
$N=2N'$.
 
Define a symplectic form $\chi$ on $V$ by setting
$$
\chi(v_{j}^{(i)},v_{k}^{(l)}):= \left\{\begin{array}{l l} 1 &
    \mbox{ if } i+1=l \equiv 0 \!\! \pmod{2} \mbox{ and }
    j+k=\lambda_{i}+1;
    \\
    -1 & \mbox{ if } l+1=i\equiv0 \!\! \pmod{2} \mbox{ and }
    j+k=\lambda_{i}+1;
    \\
    0 & \mbox{ otherwise.}
  \end{array}\right.
$$ 
The subspaces $\kk$ and $\pp$ are then defined, through
$\chi$, as in the~AI case and, $\theta$ being the associated
involution, one has:
\begin{equation}
\kk\cong \spn_{2N'}, \quad
K=\Kt\stackrel{\rho}{\cong}\tilde{\Kt}/\{\pm1\}\cong \Sp_{2N'}/\{\pm1\}, \quad
\zz(\g)\subset\pp.\label{referee}
\end{equation}
Set $\qq'_{2i+1}:=\gl\bigl(v_{j}^{(2i+1)}$,
$v_{j}^{(2i+2)}\mid j=1,\dots,\lambda_{2i+1}\bigr)$; then,
$(\qq_{2i+1}',\kk\cap\qq_{2i+1}')$ is a simple symmetric
pair of type~AII isomorphic to $(\gl_{2\lambda_{2i+1}},
\spn_{2\lambda_{2i+1}})$.  We can identify $\qq_{2i+1}$ with
$\qq_{2i+2}$ via the isomorphism $u_{i}:
\qq_{2i+1}\stackrel{\sim}{\rightarrow}\qq_{2i+2}$ defined as
follows:
$$
u_{i}(x).v_{j}^{(2i+2)}=x.v_{j}^{(2i+1)} \ \; \text{for all
  $j\in[\![1,\lambda_{2i+1}]\!]$ and $x\in\qq_{2i+1}$.}
$$ 
The involution $\theta$ associated to $(\g,\kk,\pp)$ acts on
each element $x\in\qq_{2i+2}$, resp.~$x \in \qq_{2i+1}$, by
$\theta(x) =
-u_{i}^{-1}(s_{\lambda_{2i+2}}{}^{t}xs_{\lambda_{2i+2}})$,
resp.~$\theta(x) =
-u_{i}(s_{\lambda_{2i+1}}{}^{t}xs_{\lambda_{2i+1}})$.

As in case AI, if $(\g,\kk',\pp')$ is another symmetric pair
of type AII then $\Od\cap\pp'\nonv$ and, for any element
$e'\in\Od\cap\pp'$, there exists an isomorphism of symmetric
pairs $\tau:(\g,\kk',\pp')\rightarrow (\g,\kk,\pp)$ with
$\tau(e')=e$ (see~\cite{GW} and \eqref{referee}).

\subsubsection{Case AIII}
\label{involAIII}
Following \cite{Oh1,Oh2} we will use the notion of
$ab$-diagram to classify nilpotent orbits in classical
reductive symmetric pairs of type AIII,
i.e.~$(\g,\kk)=(\gl_N,\gl_p \times \gl_q)$.

\begin{defi}
  \label{abdiag}
  An $ab$-picture is a Young diagram in which each box is
  labeled by an $a$ or a $b$, in such a way that these two
  symbols alternate along rows.
  Two $ab$-picture $\Lambda, \Lambda'$ are equivalent if they
  differ by permutations of lines of the same length. We then note $\Lambda\cong\Lambda'$.\\ 
An $ab$-diagram is an equivalence class of $ab$-pictures. The shape of an $ab$-diagram is the shape of any of its $ab$-picture.
\end{defi}

Recall that $\Od \subset \g$ is a nilpotent orbit with
associated partition $\bolda=(\lambda_1\geqslant\dots\geqslant\lambda_{\delta_{\Od}})$. To any function
\[
\Phi: [\![1,\delta_{\Od}]\!] \longrightarrow\{a,b\}
\]
one can associate an $ab$-picture $\Lambda(\Phi)$ of shape
$\bolda$ as follows: label the first box of the $i$-th row
(of size $\lambda_{i}$) of $\Lambda(\Phi)$ by $\Phi(i)$, and
continue the labeling to get an $ab$-picture as defined
above.  We denote by $\Delta(\Phi)$ the $ab$-diagram associated to $\Lambda(\Phi)$.
Observe that we may have $\Phi\neq\Psi$ and
$\Delta(\Phi)=\Delta(\Psi)$.
 
Fix a such a function $\Phi$ and decompose $V$ into a direct
sum $V= V_{a}^{\Phi} \oplus V_{b}^{\Phi}$ by defining
(cf.~\cite{Oh2})
\begin{align*}
  V_{a}^{\Phi} & :=\bigl\langle v_{j}^{(i)}\mid
  (\text{$\Phi(i)=a$ and $\lambda_{i}-j\equiv 0 \bmod{2}$)
    or ($\Phi(i)=b$ and $\lambda_{i}-j\equiv 1
    \bmod{2}$})\bigr\rangle
  \\
  V_{b}^{\Phi} & :=\bigl\langle v_{j}^{(i)}\mid
  (\text{$\Phi(i)=b$ and $\lambda_{i}-j\equiv 0 \bmod{2}$)
    or ($\Phi(i)=a$ and $\lambda_{i}-j\equiv 1
    \bmod{2}$})\bigr\rangle.
\end{align*}
Set $N_{a}:=\dim V_{a}^{\Phi}$ and $N_{b}:=\dim V_{b}^{\Phi}$,
hence $N=N_{a}+N_{b}$.  Now, if
\[
\kk^{\Phi}:=\gl(V_{a}^{\Phi})\oplus\gl(V_{b}^{\Phi})\subset\g,
\quad \pp^{\Phi}:=\Hom(V_{a}^{\Phi},V_{b}^{\Phi})\oplus
\Hom(V_{b}^{\Phi},V_{a}^{\Phi})\subset\g
\]
we obtain a symmetric Lie algebra
$$
(\g,\kk,\pp):=(\g,\kk^{\Phi},\pp^{\Phi}),
$$ 
such that $([\g,\g],\kk \cap [\g,\g])$ is irreducible of
type~AIII and $\zz(\g)\subset\kk$. One has:
$K=\rho(\GL(V_{a}^{\Phi})\times\GL(V_{b}^{\Phi}))$ and,
$\theta$ being the associated involution, $K=\Kt$ if and
only if $N_{a}\neq N_{b}$.  It is easily seen that
$(\qq_{i},\kk^{\Phi}\cap\qq_{i})$ is a reductive symmetric
pair (of type~AIII) isomorphic to
$\bigl(\gl_{\lambda_{i}},\gl_{\lfloor\frac{\lambda_{i}}2\rfloor}
\oplus \gl_{\lceil\frac{\lambda_{i}}2\rceil}\bigr)$.

The $ab$-diagram associated to a nilpotent element
$e'\in\pp^{\Phi}$ is defined in the following way (see, for
example, \cite[(1.4)]{Oh2}).  Let $\boldsymbol
{\mu}=(\mu_{1}\geqslant\dots \geqslant\mu_{k})$ be the
partition associated to $e'$.  Fix a normal \Striplet
$(e',h',f')$ and a basis of $V$
$$
\bigl\{\zeta_{j}^{(i)}\mid i\in [\![1,k]\!], \ j\in
[\![1,\mu_{i}]\!]\bigr\}
$$
such that: $\zeta_{j}^{(i)}$ belongs either to
$V^{\Phi}_{a}$ or $V^{\Phi}_{b}$, $\zeta_{1}^{(i)}$ is a
basis of $\ker f'$ and
$e'(\zeta_{j}^{(i)})=\zeta_{j+1}^{(i)}$.  Then, label the
$j$-th box in the $i$-th row of the Young diagram associated
to $\boldsymbol{\mu}$ by $a$, resp.~$b$, if
$\zeta_{j}^{(i)}\in V^{\Phi}_{a}$, resp.~$\zeta_{j}^{(i)}\in
V^{\Phi}_b$. This defines an $ab$-picture whose $ab$-diagram, denoted by $\Gamma^{\Phi}(e')$, does not depends on the choice of the $\zeta_{j}^{(i)}$.
The map $K.x \mapsto \Gamma^{\Phi}(x)$ gives a
parameterization of the nilpotent $K$-orbits in $\pp^\Phi$,
see~\cite[Proposition~1(2)]{Oh2}.
 
Remark that the element $e_{i}$, defined in~\ref{settings},
belongs to $\pp^{\Phi}\cap\qq_{i}$ in the symmetric Lie
algebra $(\qq_{i},\kk^{\Phi}\cap\qq_{i})$; its $ab$-diagram
has only one row, with first box labeled with $\Phi(i)$.  An
$ab$-diagram whose equivalent class is of the form $\Gamma^{\Phi}(x)$ is said to be
\emph{admissible} for $\Phi$.  For example,
$\Gamma^{\Phi}(e)=\Delta(\Phi)$ is admissible.  It is easy
to see that a necessary and sufficient condition for an
$ab$-diagram to be admissible is to have exactly $N_{a}$
labels equal to $a$ and $N_{b}$ labels equal to $b$.
 
The number $N_{a}-N_{b}$ is called the \emph{parameter} of
the symmetric pair $(\g,\kk^{\Phi})$.  Its absolute value
$|N_{a}-N_{b}|$ can be read from the Satake diagram of the
symmetric pair $(\g,\kk^{\Phi})$.
The parameter is different from $0$ when all the white nodes
are connected by arrows; then, its absolute value is the
number of black nodes plus one, cf.~\cite[p.~532]{He1}.  Two
symmetric pairs $(\g,\kk)$ of type AIII are isomorphic if
and only if their parameters have the same absolute value.

Assume that $(\g,\kk',\pp')$ is a symmetric Lie algebra of
type AIII such that $\Od\cap\pp'\nonv$. Then, for every
element $e'\in\Od\cap\pp'$ with $ab$-diagram $\Gamma'$, there
exists a function
$\Psi:[\![1,\delta_{\Od}]\!]\rightarrow\{a,b\}$ such that
$\Gamma'=\Delta(\Psi)$.  Furthermore, it is not
difficult to show that, in this case, there exists an
isomorphism of symmetric Lie algebras
$\tau:(\g,\kk',\pp')\rightarrow (\g,\kk^{\Psi},\pp^{\Psi})$
such that $\tau(e')=e$.

\subsubsection{Notation and remarks}
\label{involreduc}
Let $(\g,\kk,\pp)$ be a symmetric Lie algebra with
$\g=\gl_{N}=\gl(V)$ and $S_{G}$ be a $G$-sheet intersecting
$\pp$. We follow the notation introduced in
sections~\ref{settings} and~\ref{Ksheet}.
 
Recall from Lemma \ref{nilporbp} that the nilpotent orbit $\Od \subset S_{G}$
intersects $\pp$ and fix $e\in \Od\cap\pp$. Then, the
symmetric pair $(\g,\kk)$ can be described as
in~\ref{involAI}, \ref{involAII} or~\ref{involAIII}.  The
notation for $\mathbf{v}$, $\qq= \bigoplus_i \qq_i$, $\lf$,
$\tf \subset \h \subset \lf \cap \qq$, being as
in~\ref{settings}, set:
\[
\kk_{i}:=\qq_{i}\cap\kk, \quad 
\pp_{i}:=\qq_{i}\cap\pp, \quad
\theta_{i}:=\theta_{\mid\qq_{i}}.
\]
The normal \Striplet $\sS=(e,f,h)$ is then given by $e=
\sum_i e_i$, $h= \sum_i h_i$, $f= \sum_i f_i$. The map
\[
\varepsilon = \varepsilon^\g: e + \h \to e + \g^f
\]
is defined as in Remark~\ref{ggprime}; it is the restriction
of the polynomial map $\epsilon$ from Lemma~\ref{epsilon}.

Recall also that the subset $\mathsf{Z} \subset G$ is chosen
such that: $\Id \in \mathsf{Z}$, $\{g.e\}_{g\in \mathsf{Z}}$
is a set of representatives of the $K$-orbits contained in
$G.e\cap\pp$ and $g.\sS:=(g.e,g.h,g.f)$ is a normal
\Striplet.  The ``Slodowy slices'' are defined by:
\[
g.e + X(S_{G},g.\sS):= S_G \cap (g.e + \g^{g.f}), \quad
X_\pp(S_G,g.\sS) := X(S_{G},g.\sS) \cap \pp.
\]
As observed in~Remark~\ref{XsXss}, we may simplify the
notation by setting:
\[
X:= X(\sS)=X(S_{G},\sS), \quad
X_\pp:=X_{\pp}(\sS)=X(S_{G},\sS) \cap \pp.
\]
It follows from the results of Section~\ref{Katsylo} that:
$X$ is smooth, $e + X = \varepsilon(e+ \tf) $ is
irreducible, $S_G= G.(e+X)$ and $\psi : S_G \to e+X$ is a
geometric quotient of the sheet $S_G$,
cf.~Theorem~\ref{psiquotient} (recall that the group $G^e$
is connected).

Since $e, g.e \in \pp$, the remarks at the end of the
previous subsections show that there exists an
isomorphism~$\tau$ (depending on $g$) of symmetric Lie
algebras sending $(\g,\kk,\pp)$ to a symmetric pair of the
same type (AI, AII or AIII) such that $\tau(e)=g.e$.  It is not hard to see that we
can further assume that $\tau(\sS)=g.\sS$. 
The main consequence of this observation is that, applying $\tau$,
any property obtained for $e+X_{\pp}(\sS)$ also holds for
$g.e+X_{\pp}(g.\sS)$. In particular, we will mainly work
with $e+X_{\pp}(\sS)$.

\subsection{Properties of slices} \label{slicesproperties}
We continue with the notation of~\ref{involreduc}. Hence
$S_G \subset \g$ is a $G$-sheet, $e \in S_G \cap \pp$ is a
fixed nilpotent element and $\sS=(e,f,h)$, $\mathbf{v}$,
$\qq$, etc., are as defined in~\ref{settings}.

\subsubsection{The slice property~(1)}
\label{sliceproperty}
In this subsection we give a proof of Theorem~\ref{unif} for
types AI and AII. It asserts that the
condition~\eqref{heart} holds; we also refer to it as the
\emph{slice property}.

\begin{thm}\label{unif}
  Assume that $(\g,\theta)$ is of type~{\rm A}. Then, one
  has:
  \begin{equation}
    G.(e+X_{\pp})=G.(S_{G}\cap\pp).\tag{$\heartsuit$} 
  \end{equation} 
  Moreover, in types {\rm AI} and {\rm AII} a stronger
  version holds, namely: $e+X\subset\pp$.
\end{thm}

\begin{proof}
  Since $S_G= G.(e+X)$ and $e+ X_\pp = (e +X) \cap \pp
  \subset S_G \cap \pp$, the inclusion $G.(e+X_{\pp})
  \subset G.(S_{G}\cap\pp)$ is obvious.  Clearly, $e + X
  \subset \pp$ yields $G.(e+ X_\pp) = G.(e+ X) = S_G \supset
  G.(S_G \cap \pp)$. We prove below that the inclusion
  $e+X\subset\pp$ is true when $(\g,\kk)$ is of type AI or
  AII. The proof of the theorem in type AIII is postponed to
  subsection~\ref{sliceAIII}, see
  Proposition~\ref{unifAIII}.

  Type~AI: As said in subsection~\ref{involAI}, each
  $(\qq_{i},\kk_{i})$ is a symmetric pair of type~AI. Since
  this pair has maximal rank and $e_{i} \in \pp_i$ is a
  regular element, one has
  $\qq_{i}^{f_{i}}=\pp\cap\qq_{i}^{f_{i}}$.  Therefore the
  image of each map $\varepsilon_{i} : e_i + \h_i \to e_i +
  \qq^{f_i}$, as defined in~\ref{settings}, is contained in
  $\qq_{i}^{f_{i}}\subseteq\pp_{i}$.
  From $\varepsilon = \sum_i \varepsilon_i$ one gets that
  $e+X=\varepsilon(e+\tf)\subset S_{G}\cap\pp$.

  Type~AII: Recall that $\lambda_{2i+1}=\lambda_{2i+2}$ if
  $2i+2\leqslant \delta_{\Od}$. Let $x=\sum_{i}x_{i}\in\qq$;
  then $x\in\pp\cap\qq$ if and only if, for all $i$,
  $x_{2i+1}=-\theta_{2i+2}(x_{2i+2})=s_{\lambda_{2i+1}}{}^{t}(u_{i}^{-1}(x_{2i+2}))s_{\lambda_{2i+1}}$  (cf.~\S\ref{involAII}). 
  This last condition means that $x_{2i+1}$ is the
  transpose of $u_{i}^{-1}(x_{2i+2})$ with respect to the
  antidiagonal.  Fix $t\in\tf$, hence
  $e+t\in S_G$; from the description of $\tf$ given
  in~\eqref{carat}, one deduces that
  $u_{i}(e_{2i+1}+t_{2i+1})=e_{2i+2}+t_{2i+2}$.  Set
  $x=\varepsilon(e+t)$. It follows from
  $u_{i}\circ\varepsilon_{2i+1}=\varepsilon_{2i+2}\circ
  u_{i}$ that $u_{i}(x_{2i+1})=x_{2i+2}$.  Since
  $e_{2i+1}+\qq^{f_{2i+1}}_{2i+1}$ is fixed under the
  conjugation by $s_{\lambda_{2i+1}}$, one obtains
  $-\theta_{2i+2}(x_{2i+2})= s_{\lambda_{2i+1}} {}^tx_{2i+1}
  s_{\lambda_{2i+1}} =x_{2i+1}$. Hence
  $\varepsilon(e+t)\in\pp$ and, therefore,
  $\varepsilon(e+\tf) =e +X \subset\pp$.
\end{proof}

\begin{cor}\label{corounif}
  Every $G$-orbit contained in $S_{G}$ and intersecting
  $\pp$, also intersects $(\qq\cap\pp)^{\bullet}$.
\end{cor}
\begin{proof}
  It suffices to observe that $e+X \subset \qq^{\bullet}$
  and $(\qq\cap\pp)^{\bullet}\subset\qq^{\bullet}$
\end{proof}

\begin{Rq}\label{diamondAI}
  (1) One can deduce Theorem~\ref{unif} from
  Corollary~\ref{corounif}.  Indeed, let $x \in S_{G}$ and
  suppose that $y\in G.x\cap (\qq\cap\pp)^{\bullet}$.  Since
  $e$ is regular in $\qq$, it follows from \cite{KR} that
  $y$ is $(Q\cap K)^{\circ}$-conjugate to an element of
  $e+X_{\pp}$.
  \\
  (2) Assume that $(\g,\kk)$ is of type AI or~AII. Then,
  since $e+ X_\pp = e+X$ is irreducible and smooth in type~A
  (see \S\ref{involreduc}), Theorem~\ref{unif} yields that
  the condition~\eqref{diamond},
  cf.~\S\ref{Ksheet}, holds.
\end{Rq}

\subsubsection{The slice property (2)}
\label{sliceAIII}
We assume in this section that $(\g,\kk,\pp)=
(\g,\kk^\Phi,\pp^\Phi)$ is of type~AIII.  Let $\af \subset
\pp$ be a Cartan subspace and $\h' \subset \g$ be a Cartan
subalgebra containing $\af$.  Denote by $B$ a
$\sigma$-fundamental system of the root system $R(\g,\h'):=
R([\g,\g],\h'\cap [\g,\g])$, see~\S\ref{basis} with $\h$
replaced by $\h' \cap [\g,\g]$. Let $\bar{D}$ be the Satake
diagram of type~AIII associated to $B$
(cf.~\cite[p.~532]{He1}).  Since $\af\subset[\g,\g]$, see
\ref{involAIII}, one can define a $\Q$-form of $\af$ by
\[
\af_{\Q}:=\{a\in\af \mid \; \text{$\alpha(a)\in\Q$ for all
  $\alpha\in R(\g,\h')$}\}.
\]
The nodes of $\bar{D}$ can be labeled by the elements
$\alpha_{1},\dots,\alpha_{N-1}$ of $B$. Set
$\alpha_{i}':=\alpha_{N-i}$, $1 \le i \le N-1$, hence
${\alpha_{i}}_{\mid\af} = {\alpha_{i}'}_{\mid\af}$; there
exists an arrow between $\alpha_{i}$ and $\alpha'_{i}$ when
these nodes are colored in white and $i\neq N/2$.

Let $s \in \g$ be semisimple and let $c \in \spec(s)=\{ \mbox{eigenvalues of $s$ on $V$}\}$.  
Denote by $V_{s,c}$ the
eigenspace associated to $c$; thus, $m(s,c) := \dim V_{s,c}$
is the multiplicity of $c$. More generally,
see~\S\ref{regA}, we set $V_{s,d}:=\ker(s - d \, \Id_V)$ and
$m(s,d):= \dim V_{s,d}$ for every $d \in \K$.  One can
identify $\gl(V_{s,c})$ with a Lie subalgebra of $\gl(V)$ by
extending an element $x\in \gl(V_{s,c})$ by $0$ on
$\bigoplus_{c'\neq c}V_{s,c'}$.  Under this identification,
$\sln(V_{s,c})$ is a simple factor of $\g^s$ if and only if
$m(s,c)\geqslant 2$.  Setting
\[
\wfr_{s,c}':=\sln(V_{s,c}), \quad \wfr_{s,c}:=\gl(V_{s,c}),
\]
one has:
\begin{equation}\label{Levisimplegl}
  \g^s= \bigoplus_{c \in \spec(s)}\wfr_{s,c}  =  \cggsnb \,
  \oplus  \bigoplus_{m({s,c})\geqslant 2}\wfr_{s,c}'.
\end{equation}
Denote by $M_{s,c}$ the connected algebraic subgroup of $G$
group with Lie algebra $\wfr_{s,c}'$.  Then, $M_{s,c}$ acts
on $\wfr_{s,c}$ via the adjoint action and the group $G^s$
is generated by $C_{G}(\g^s)$ and the $M_{s,c}$, $c \in
\spec(s)$ (see \S\ref{Levi} and
Proposition~\ref{Levisimple}).
 
The group $\{\pm 1\}$ acts by multiplication on $\spec'(s):=
\{c \in \spec(s) \mid - c \in \spec(s)\}$; let $\Sppm(s) :=
\spec'(s)/\{\pm 1\}$ be the orbit space.  The class of $c
\in \spec'(s)$ in $\Sppm(s)$ is denoted by $\pm c$. When $0
\in \spec(s)$ we simply write 0 instead of $\pm 0$.  We then set
\[
\g_{s,\pm c} :=\wfr_{s,c} \oplus \wfr_{s,-c}, \quad
\g_{s,0}:=\wfr_{s,0}.
\]
If $0 \ne c \in \spec'(s)$, the connected subgroup of $G$
generated by $M_{s,c}$ and $M_{s,-c}$ is denoted by
$G_{s,\pm c}$ and we set $G_{s,0}:=M_{s,0}$. One has
$\Lie(G_{s,\pm c}) = [\g_{s,\pm c},\g_{s, \pm c}]$.

Recall that we have written $V = V_a^\Phi \oplus V_b^\Phi$;
we set $V_a := V_a^\Phi$, $V_b:= V_b^\Phi$. The parameter of
$(\g,\kk)$ is $N_a - N_b$ where $N_a:= \dim V_a$, $N_b:= \dim
V_b$, see~\ref{involAIII}.

\begin{lm}\label{fondAIII}
  Let $s \in \pp$ be a semisimple element. Then:
  \\
  {\rm (1)} $m(s,c)=m(s,-c)$ for all $c\in\K$;
  \\
  {\rm (2)} the symmetric Lie algebra $(\g^s,\kk^s)$
  decomposes as $\bigoplus_{\pm c\in\Sppm(s)} (\g_{s,\pm
    c},\kk_{s,\pm c})$, where $\kk_{s,\pm
    c}:=\kk\cap\g_{s,\pm c}$;
  \\
  {\rm (3)} if $c\neq0$, $(\g_{s,\pm c},\kk_{s,\pm c})$ is a
  reductive symmetric pair whose semisimple part is
  irreducible of type~{\rm A0};
  \\
  {\rm (4)} $V_{s,0} = (V_{s,0}\cap V_a)\oplus (V_{s,0}\cap
  V_b)$ and the symmetric Lie algebra $(\g_{s,0},\kk_{s,0})$
  is a reductive symmetric pair whose semisimple part is
  irreducible of type {\rm AIII}, with the same parameter as
  $(\g,\kk)$. In particular, the parameter of $(\g,\kk)$ is
  $0$ when $0 \notin \spec(s)$.
\end{lm}

\begin{proof}
  (1) Since the involution $\theta$ is inner, the claim
  follows from the following elementary observation.
  Suppose that $A \in \GL_N$, $x \in \gl_N$, and set $x' :=
  AxA^{-1}$. Then, $m(x,c)=\dim \ker(x -c \, \Id) = \dim
  \ker(x' - c \, \Id)$; in particular, $m(x,c)=m(x',c)$,
  thus $ m(x,c)= m(x,-c)$ when $x'= -x$.
  \\
  (2) The assertion is an easy consequence of
  \eqref{Levisimplegl} and $\theta(\wfr_{s,c})=
  \wfr_{s,-c}$.
  \\
  (3) \& (4).  We may assume that $N_{a} \geqslant N_{b}$
  and, by Proposition~\ref{levistandard}, $s\in \af_\Q$.
  Then, the claims can be read on the Satake diagram of type
  AIII, except for the equality of the parameters when $c=0$
  (one only sees in this way that the absolute values are
  equal).  A complete proof can be given as follows.
  \\
  Let $(v_{a,i})_{i\in [\![1,N_{a}]\!]}$ and
  $(v_{b,i})_{i\in[\![1,N_{b}]\!]}$ be bases of $V_{a}$ and
  $V_{b}$.  For each $i\in[\![1,N_{b}]\!]$, define $u_{i}\in
  \pp$ by
  \[
  u_i(v_{d,j})=
  \begin{cases}
    v_{\bar{d},i} & \text{if $i=j$}
    \\
    0 & \text{otherwise,}
  \end{cases}
  \]
  where $\bar d$ is the element of $\{a,b\}\smallsetminus
  \{d\}$.  The subspace generated by the $u_{i}$,
  $i\in[\![1,N_{b}]\!]$, is a Cartan subspace of $\pp$. If
  $s=\sum_{i}c_{i} u_{i}$, the eigenvalues of $s$ are given
  by square roots of the $c_{i}$'s and one has
  $V_{s,0}=\bigl\langle \{v_{a,i},v_{b,i}\mid
  c_{i}=0\}\cup\{v_{a,i}\mid i\gnq N_{b}\}\bigr\rangle$.  It
  is then not difficult to get the desired assertions.
\end{proof}

Recall from \S\ref{regA} that if $t=\sum_{i}t_{i} \in
\qq=\bigoplus_{i}\qq_{i}$ is semisimple, $m_{i}(t,c)$
denotes the multiplicity of the eigenvalue $c$ for
$t_{i}\in\qq_{i}$; recall also that $\h \subset \qq$.

\begin{lm}\label{miha}
  Let $t\in\h$ be such that $G.(e+t)\cap\pp\nonv$. Then:
  \begin{equation}\label{mitc}
    m_{i}(t,c)=m_{i}(t,-c) \ \; \text{ for all $c\in\K$}.
  \end{equation}
\end{lm}

\begin{proof}
  Let $s_{1} + n_{1}$ be the Jordan decomposition of $e+t$
  and pick $g\in G$ such that $g.(e+t)\in\pp$. Therefore,
  $s:=g.s_{1}\in\pp$ and $n:=g.n_{1} \in \pp\cap \g^{s}$.  By
  Corollary~\ref{corlm} we know that $t$, $s_{1}$ and $s$
  are in the same $G$-orbit. Then, Lemma~\ref{fondAIII}(1)
  gives $m(t,c)=m(s,c)=m(s,-c)=m(t,-c)$.  On the other hand,
  $n \in \pp\cap \g^{s}$ is a nilpotent element of the
  subsymmetric pair $(\g^{s},\g^s\cap\kk)=\prod_{\pm
    c\in\Sppm(s)}(\g_{s,\pm c},\kk_{s,\pm c})$,
  cf.~Lemma~\ref{fondAIII}(3,4).
  The orbit $K^s.n$ belongs to $\pp\cap \g^{s}$. Hence it can be decomposed along the previous direct product:
  \[
  K^s.n=\prod_{\pm c\in\Sppm(s)} \Od_{\pm c}
  \]
 where $\Od_{\pm c}$ is the projection of the orbit $K^s.n$ onto  $\pp_{s,\pm c}$. 
The result in the case $c=0$ is vacuous.  Recall that when
  $c\neq0$ one has $\g_{s,\pm
    c}=\wfr_{s,c}\oplus\wfr_{s,-c}$, and we can further
  decompose each orbit $G_{s,\pm c}\cdot\Od_{\pm c}$ as
  $\Od_{c}\times\Od_{-c} \subset \wfr_{s,c} \times
  \wfr_{s,-c}$.  Then, $G_{s,\pm c}\cdot\Od_{\pm c}$, is
  characterized by the Young diagrams of the nilpotent
  orbits $\Od_{c}$, $\Od_{-c}$.  Since $(\g_{s,\pm
    c},\kk_{s,\pm c})$ is of type~A0, these two Young
  diagrams are equal (cf.~\S\ref{involA0}).  The results of
  \S\ref{corlm} then yield that the partition of
  $\Od_{\delta c}$, $\delta\in\{-1,1\}$, is given by the
  sequence $(m_{i}(t,\delta c))_{i}$.  As these two
  sequences are decreasing on $i$, cf.~\eqref{carat}, one
  obtains $m_{i}(t,c)=m_{i}(t,-c)$ for all $i$.
\end{proof}

The following proposition completes the proof of
Theorem~\ref{unif} and Corollary~\ref{corounif} in
case~AIII.

\begin{prop}\label{unifAIII}\label{cormiha}
  Let $t\in \tf$.
  \\
  {\rm (i)} If $t$ satisfies~\eqref{mitc}, then
  $\varepsilon(e+t)\in e+\XP$.
  \\
  {\rm (ii)} One has $G.(e+t)\cap\pp\neq\emptyset$ if and
  only if $t$ satisfies~\eqref{mitc}.
  \\
  {\rm (iii)} The condition~\eqref{heart} holds,
  i.e.~$G.(S_{G}\cap\pp)=G.(e+\XP)$.
\end{prop}

\begin{proof}
  (i) Recall that $t= \sum_i t_i$, $e = \sum_i e_i$ with
  $t_i \in \qq_i$ and $e_i\in \pp \cap \qq_i$ regular in
  $\qq_i \cong\gl_{\lambda_{i}}$.  The map $\varepsilon$ can
  be written as $\sum_i \varepsilon_i$, where
  $\varepsilon_i$ is given by Lemma~\ref{sym} applied in the
  algebra $\qq_i$.  Thus $\varepsilon_i(e_i+t_i) = e_i +
  \sum_{j \le 0} P_j(t_i)$.  From~\eqref{mitc} and since
  the polynomials $P_j$ are symmetric in eigenvalues of $t_i$ (Lemma~\ref{sym}), one obtains
  $P_{j}(t_i)=0$ if $j$ is even.  One can deduce from the
  construction made in~\ref{involAIII} that the subspaces
  $\pp_{i} := \pp \cap \qq_i$ are the sum of the
  $j$-subdiagonals and $j$-supdiagonals of $\qq_{i}$ for $j$
  odd. It follows that $\varepsilon_{i}(e_{i}+t_{i}) \in
  e_{i}+\pp_{i}^{f_{i}}$, hence $\varepsilon (e+t) \in
  e+X\cap\pp$.
  \\
  (ii) By Lemma~\ref{epsilon} one has
  $G.(e+t)=G.\varepsilon(e+t)$, thus part~(i) shows that the
  condition is sufficient. Lemma~\ref{miha} gives the
  converse.
  \\
  (iii) The inclusion $G.(e+\XP) \subset G.(S_{G}\cap\pp)$
  is always true. By Proposition~\ref{axiomK}, every $x\in
  S_{G}\cap\pp$ is $G$-conjugate to an element $e+t\in
  e+\tf$; parts~(i) and~(ii) give $\varepsilon(e+t)\in
  G.x\cap (e+\XP)$ and the result follows.
\end{proof}

We now find a convenient subspace $\cc \subset \tf$ such
that $\varepsilon(e + \cc) = e+ X_\pp$. For
$i\in[\![1,\lambda_{\delta_{\Od}}]\!]$ and $j\in [\![
0,\left\lfloor (\lambda_{i}-\lambda_{i+1})/ 2\right\rfloor-1
]\!]$, define elements $c(i,j) = (c({i,j})_k)_k \in
\K^{\lambda_{1}}$ by:
\begin{equation} \label{defcprime}
  c(i,j)_{k}:=\left\{\begin{array}{l l}1
      &\mbox{ if } k=\lambda_{i+1}+2j+1;\\
      -1 &\mbox{ if } k=\lambda_{i+1}+2j+2;\\
      0 &\mbox{ otherwise.}
    \end{array}\right.
\end{equation}
Let $\cc'$ be the subspace of $\K^{\lambda_{1}}$ generated
by the elements $c(i,j)$. Recall from~\eqref{alpha} the
isomorphism $\alpha : \K^{\lambda_{1}} \isomto \tf$ and set:
\begin{equation}\label{defc}
  \cc:=\alpha(\cc')\subset\tf.
\end{equation}
The main property of the subspace $\cc$ is the following.
By construction every element of $\cc$
satisfies~\eqref{mitc}; conversely, Lemma~\ref{eplushconj}
applied in each $\qq_{i}$ implies that any element $e+t$
(with $e = \sum_i e_i$, $t = \sum_i t_i$)
satisfying~\eqref{mitc} is conjugate to an element of $\cc$.

\begin{prop} \label{corcc} In the previous notation one
  has: $\varepsilon(e+\cc)=e+\XP$ and
  $G.(e+\cc)=G.(S_{G}\cap\pp)$. Moreover,
  \begin{equation}
    \label{dimd}
    \dim \cc =
    \sum_{i=1}^{\delta_\Od} \textstyle{\left\lfloor 
        \frac{\lambda_i-\lambda_{i+1}}2\right\rfloor},
  \end{equation}
  which only depends on $\bolda$.
\end{prop}

\begin{proof}
  The formula~\eqref{dimd} follows without difficulty from
  the definition of $\cc'$.
  Since the elements of $e + \cc$ satisfy~\eqref{mitc},
  Proposition~\ref{unifAIII}(i) gives $\varepsilon(e + \cc)
  \subset e + \XP$. Conversely, let $e + x \in e + \XP$. As
  $e+X= \varepsilon(e+\tf)$, the element $e+x
  =\varepsilon(e+t)$, $t \in \tf$, is the unique point of
  $e+X$ intersecting the orbit $G.(e+x)=
  G.\varepsilon(e+t)=G.(e+t)$
  (see~Lemma~\ref{epsilon}(i)). By
  Proposition~\ref{unifAIII}(ii), $e+t$
  satisfies~\eqref{mitc} and, as noticed above, $e+t$ is
  conjugate to an element $e+c \in e+ \cc \subset e+\tf$. It
  follows that $\{e+x\}=G.(e+x)\cap (e+X)=
  G.\varepsilon(e+c) \cap (e+X)=
  \{\varepsilon(e+c)\}$. Hence, $e+x = \varepsilon(e+c) \in
  \varepsilon(e + \cc)$.
  Finally, $G.(S_G \cap \pp)= G.(e+ \XP)= G.\varepsilon(e +
  \cc)=G.(e + \cc)$.
\end{proof}

\begin{Rq}
  \label{diamondAIII} \label{corccc} Proposition~\ref{corcc}
  implies that condition~\eqref{diamond} holds in case AIII,
  i.e., $e+\XP$ is irreducible.  
\end{Rq}

Corollary~\ref{corounif} says that in each $G$-orbit
contained in $S_{G}$ and intersecting $\pp$ one can find an
element $x=s+n\in (\qq\cap\pp)^{\bullet}$.  The next
corollary summarizes various results which can be deduced
from Lemma~\ref{fondAIII}. Recall that $\qq =\bigoplus_i
\qq_i$ and that $(\qq_{i},\kk\cap\qq_{i})$ is a symmetric
Lie algebra of type~AIII. Applying Lemma~\ref{fondAIII} in
each symmetric pair $(\qq_{i},\kk\cap\qq_{i})$ yields:

\begin{cor}\label{fondAIIIbis}
  Let $x= s +n \in (\qq\cap\pp)^{\bullet}$ and write
  $s=\sum_{i} s_{i}$, $n=\sum_{i} n_{i}$ with $s_{i},
  n_{i}\in\pp\cap\qq_{i}$, as in~{\rm \ref{settings}}.
  \\
  {\rm (1)} The Levi factor $\qq_{i}^{s_{i}}$ of $\qq_{i}$
  has the following decomposition:
$$
\qq_{i}^{s_{i}}=\bigoplus_{c\in \K} \wfr_{i,s_i,c}
$$ 
where $\wfr_{i,s_i,c}:=\gl \bigl(\ker(s_i -c \, \Id)\bigr) \subset \qq_{i}$.
\\
{\rm (2)} The symmetric pair
$(\qq_{i}^{s_{i}},\qq_{i}^{s_{i}}\cap\kk)$ decomposes as
\[
(\qq_{i}^{s_{i}},\qq_{i}^{s_{i}}\cap\kk) = \bigoplus_{\pm
  c\in\Sppm(s_i)} (\qq_{i,s_i,\pm c},\kk_{i,s_i,\pm c})
\]
where $(\qq_{i,s_i,0},\kk_{i,s_i,0}):=(\wfr_{i,s_i,0},
\wfr_{i,s_i,0} \cap\kk)$ is of type~\emph{AIII} and, when
$c\neq 0$, $(\qq_{i,s_i,\pm c},\kk_{i,s_i,\pm
  c}):=\bigl((\wfr_{i,s_i,c}\oplus\wfr_{i,s_i,-c}),
(\wfr_{i,s_i,c} \oplus\wfr_{i,s_i,-c}) \cap\kk\bigr)$ is of
type~\emph{A0}.
\\
{\rm(3)} The factor $(\qq_{i,s_i,0},\kk_{i,s_i,0})$ has the
same parameter as $(\qq_{i},\qq_{i}\cap\kk)$.  In
particular, the ranks of $\qq_{i}$ and $\qq_{i,s_i,0}$ have
the same parity.
\\
{\rm(4)} The nilpotent element $n_{i}$ is regular in
$\qq_{i}^{s_{i}}$; thus, the orbit $(Q\cap K)^{\circ}.n_{i}$
is uniquely determined by its one row $ab$-diagram (see~{\rm
  \ref{involAIII}}).
\end{cor}

\subsection{$J_{K}$-classes in type A}
\label{JK-classes-type-A}
Knowing that \eqref{heart} holds, we want to prove below
that condition~\eqref{club}, introduced in~\S\ref{Ksheet},
is satisfied. As above, $S_{G}\subset\g^{(2m)}$ is a
$G$-sheet and $e \in S_G$ is a nilpotent element.
We fix a Jordan ${G}$-class $J \subset S_G$ such that $J
\cap \pp \nonv$.  Recall from Theorem~\ref{compirr} that
$J\cap\pp$ is a (disjoint) union of $J_K$-classes.


\subsubsection{Cases AI and AII} \label{JclassAI+II}
In this subsection we assume that $(\g,\theta)=(\g,\kk,\pp)$
is a symmetric Lie algebra of type AI or AII, as described
in~\ref{involAI} and~\ref{involAII}.

We will need the following result, which is a formulation of
\cite[Proposition~4]{Oh1} in a slightly more general
setting. (Its proof is exactly the same.)

\begin{prop}[Ohta]\label{othageneral} 
  Let $\kappa$ be a linear involution of the associative
  algebra $\g=\gl_{N}$ and $x \mapsto x^*$ be a linear
  anti-involution of the associative algebra $\g$ which
  commutes with $\kappa$. Define:
$$
G':=\g^{\kappa} \cap \GL_{N}, \quad G'':=\bigl\{g\in G' :
g^*=g^{-1}\bigr\}.
$$ 
Set $\sigma(x):=-x^*$ and let $\eta,\eta'$ be elements of
$\{\pm1\}$. Then, via the adjoint action, $G'$ acts on
$\g^{\eta' \kappa}$ and $G''$ acts on $\g^{\eta
  \sigma}\cap\g^{\eta' \kappa}$. The elements $x, y \in
\g^{\eta \sigma}\cap\g^{\eta' \kappa}$ are conjugate under
$G''$ if and only if they are conjugate under $G'$.
\end{prop}


We may apply this proposition in the two following
situations. Fixing $\eta=-1$, $\eta'=1$, we take:
$(\kappa=\Id, x^*={}^tx)$ in type~AI, $(\kappa=\Id,
x^*=-J{}^tx J)$ in type~AII, where $J=
\bigl[ \begin{smallmatrix} 0 & \Id \\ -\Id &
  0 \end{smallmatrix} \bigr] \in \gl_{2N'}$. Observe that
$\g'=\g=\gl_{N}$, $G''=\Ort_{N}$, resp.~$G''=\Sp_{N}$, and recall
that the action of $G'=\GL_{N}=\tilde{G}$ factorizes through
$G\cong \tilde{G}/\{\K^{\times} \Id\}$.  Then, $\sigma$ is
an involution of the Lie algebra $\g$ of type~AI, resp.~AII
(cf.~\cite[Theorem~3.4]{GW}).  Using an isomorphism $\tau$
as explained in~\ref{invol}, we may assume that
$\kk=\tau(\g^{\sigma})$ and
$\pp=\tau(\g^{-\sigma})$. Moreover, in each case
$\rho(\tau(G''))=\Kt$ (cf.~\ref{involAI}
and~\ref{involAII}).
 
We therefore have obtained the (well known) result:

\begin{prop} \label{JclassAIbis} Let $(\g,\theta)$ be of
  type~{\rm AI} or~{\rm AII}. For $x,y\in\pp$ one has the
  equivalence:
$$
\Kt.x=\Kt.y \, \Longleftrightarrow \, G.x=G.y.
$$
\end{prop}

\begin{cor}\label{JclassAI}
  If $(\g,\theta)$ is of type~{\rm AI} or~{\rm AII}, the
  $J_{K}$-classes contained in $J\cap\pp$ are conjugate
  under~$\Kt$.
\end{cor}

\begin{proof}
  Let $J_{1}:=K.(\cpps+n)$ be the Jordan $K$-class containing
  $x=s+n\in J\cap\pp$ and denote by $J_{2}:=K.(\cpps+n')$
  another Jordan $K$-class contained in $J\cap\pp$
  (cf.~\ref{Jclass2}(i)).  Since $J=G.(\cggs+n')$, there
  exists $g\in G$ such that $g.x\in \cggs+n'$ and
  Remark~\ref{rqthmL} implies that $g.x\in J_{2}$.  Now, by
  Proposition~\ref{JclassAIbis}, we may assume that
  $g\in\Kt$.  Then, $g.J_{1}$ is an irreducible subvariety
  of $J\cap\pp$ of dimension $\dim J\cap\pp$
  (see~Lemma~\ref{Jclass2}(ii)) which intersects $J_{2}$.
  It follows from Theorem~\ref{compirr} that
  $g.J_{1}=J_{2}$.
\end{proof}

\begin{Rq} \label{RqKtheta}As $\Kt=K\cup\omega K$ in type~AI
  (cf.~\ref{involAI}), there are at most two Jordan
  $K$-classes in $J\cap\pp$. In type~AII one has $\Kt=K$ and
  $J\cap\pp$ is a Jordan $K$-class.
\end{Rq}

\begin{cor}\label{clubAI}
  The condition~\eqref{club} of section~{\rm \ref{Ksheet}}
  is satisfied.
\end{cor}

\begin{proof}
  Let $J_1 \subset J\cap \pp$ be a $J_K$-class. By
  Lemma~\ref{Jclass4} there exists a $J_{K}$-class $J_{2}
  \subset J \cap \pp$ such that $J_{2}$ is well-behaved
  w.r.t.~$K.e$, and Corollary~\ref{JclassAI} gives $k\in
  \Kt$ such that $J_1= k.J_2$. Since $k$ defines an
  automorphism of the symmetric Lie algebra $(\g,\kk,\pp)$,
  the class $J_1=k.J_{2}$ is well-behaved
  w.r.t.~$K.(k.e)=k.(K.e)$.
\end{proof}

\subsubsection{Case AIII~(1)}
\label{JclassAIII}
We fix $(\g,\theta)=(\g,\kk,\pp)=(\g,\kk^{\Phi},\pp^{\Phi})$
of type~AIII as in section~\ref{involAIII} and we use the
notation introduced in~\ref{sliceAIII}. For simplicity we
assume that the numbers $N_a,N_b$ are such that $N_b \le
N_a$.

Let $\af \subset \pp$ be a Cartan subspace. Since the
involutions of type~AIII are conjugate, and the Cartan
subspaces are $K$-conjugate, one can find a Cartan
subalgebra $\h'$ containing $\af$ and satisfying the
following conditions (see, for example,
\cite[Polarizations-Type AIII, p.~20]{GW}).  There exists a
basis $(\varpi_1,\dots,\varpi_N)$ of $\h'^*$ such that:
$\varpi_{j}(t)$, $1 \le j \le N$, are the eigenvalues of
$t\in\h'$ and $B:= \{\alpha_j = \varpi_j - \varpi_{j+1} \mid
1 \le j \le N-1\}$ is a $\sigma$-fundamental system of the
root system $R:=R(\g,\h')$.  Recall that the Weyl group
$W(\g,\h') = N_G(\h')/Z_G(\h')$ can be naturally identified
with the group
$\Sf\bigr(\{\varpi_1,\dots,\varpi_N\}\bigr)\cong\Sf_N =
\Sf([\![1,N]\!])$, where we denote by $\Sf(E)$ the
permutation group of a set $E$.  Moreover, the action of
$\theta$ on $\h'$ is defined by:
\begin{equation}\label{thetaeps}
  \varpi_i(\theta(t)):= \left\{\begin{array}{l
        l} \varpi_{N+1-i}(t)& \mbox{ if } 
      \min(i,N+1-i)\leqslant N_b;\\ 
      \varpi_{i}(t)&\mbox{ otherwise.}\end{array}\right.
\end{equation}
Fix the semisimple part $s$ of an element belonging to $J
\cap \pp$. By Lemma~\ref{Jclass2}, $J\cap\pp$ is the union
of $J_{K}$-classes of the form $K.(\cpps+n)$ where $n \in
\pp^s$ is nilpotent.  Thanks to
Proposition~\ref{levistandard} we may assume that $s\in
\af_\Q$ is in the positive Weyl chamber defined by $B$.
Recall from~\eqref{Levisimplegl} that we write
\[
\g^s=\bigoplus_{c\in\spec(s)}\wfr_{s,c}, \quad \wfr_{s,c} :=
\gl(V_{s,c}),
\]
where $\gl(V_{s,c})$ is naturally embedded into $\g= \gl(V)$.
Note that $\cc_{\g}(\g^s)=\bigoplus_{c\in\spec(s)} \K\,
\Id_{V_{s,c}}$.
Let $g \in N_{G}(\g^s)$; then $s':= g.s \in \cc_{\g}(\g^s)$,
hence $s'_{\mid V_{s,c}}=c'\, \Id_{V_{s,c}}$ for some $c'\in
\spec(s)$, that is to say $V_{s,c} \subset V_{s',c'}$. It is
then easily seen that the map $\eta : c \mapsto c'$ defines
a permutation of $\spec(s)$ such that $V_{s,c} = V_{s',c'}$.
If $\rtt(g): = \eta^{-1}\in \mathfrak{S}\bigl(\spec(s)\bigr)$
one has $V_{s',c}=g.V_{s,c}=V_{s,\rtt(g)(c)}$ and it follows
that:
\[
g.\wfr_{s,c}=\wfr_{s,\rtt(g)(c)} \ \; \text{for all $c \in
  \spec(s)$}.
\]
From this observation one deduces that $\rtt$ is a group homomorphism
\[
\rtt : N_{G}(\cggsnb)=N_{G}(\g^s) \, \longrightarrow \,
\mathfrak{S}\bigl(\spec(s)\bigr), \quad g \mapsto \rtt(g).
\]
Clearly, if $\gamma= \rtt(g)$ one has:
\begin{equation}\label{dimmsc}
  m(s,\gamma(c))=m(s,c) \ \; \text{for all $c \in
    \spec(s)$}.
\end{equation}
This condition characterizes the elements of the image of
$\rtt$:

\begin{lm}\label{permMM}
  An element $ \gamma \in\mathfrak S\bigl(\spec(s)\bigr)$ is
  in the image of the morphism $\rtt$ if and only if it
  satisfies~\eqref{dimmsc}.
\end{lm}

\begin{proof}
  Let $c_1,\dots,c_\ell$ be the distinct eigenvalues of
  $s$. By construction, $\gamma$ can be identified with the
  element $\gamma \in \Sf_\ell$ such that $\gamma(c_i)=
  c_{\gamma(i)}$, $1 \le i \le \ell$. Write $[\![1,N]\!]$ as
  a disjoint union $\bigsqcup_{j=1}^\ell J_j$, where
  $J_j := \{ k : \varpi_k(s)= c_j\}$.  By~\eqref{dimmsc} one
  has $\# J_j = m(s,c_j)= \# J_{\gamma(j)} =
  m(s,\gamma(c_j))$. One can therefore find $w \in \Sf_N
  \cong W(\g,\h')$ such that $w(J_j)= J_{\gamma(j)}$ for $j
  =1, \dots,\ell$. Let $g \in N_G(\h')$ be a representative
  of $w$. One then gets $g \in N_G(\cggsnb)=N_{G}(\g^s)$ and
  $\rtt(g)= \gamma$.
\end{proof}

Recall from~\S\ref{sliceAIII} that we denote by $\Sppm(s)$
the set of classes $\{\pm c : c \in \spec'(s)\}$.  For every
$k\in N_{K}(\g^s)$ we have
$$
\wfr_{s,\rtt(k)(-c)}=k.\wfr_{s,-c}=k.\theta(\wfr_{s,c})
=\theta(k.\wfr_{s,c}) =\wfr_{s,-\rtt(k)(c)}.
$$ 
Thus $\rtt(k)(-c)=-\rtt(k)(c)$ and, since $\g_{s,\pm c} =
\wfr_{s,c} \oplus \wfr_{s,-c}$, one gets $k.\g_{s,\pm
  c}=\g_{s,\pm\rtt(k)(c)}$.  Therefore, any element of
$\rtt(N_{K}(\g^s))$ induces a permutation of $\Sppm(s)$.  By
Lemma~\ref{fondAIII}, if $0 \in \spec(s)$, the factor
$(\g_{s,0},\kk_{s,0})$ is the unique factor of type~AIII in
the decomposition of the symmetric Lie algebra
$(\g^s,\kk^s)$ and, as $k\in N_{K}(\g^s)$ defines an
automorphism of this symmetric pair, one necessarily has
$\rtt(k)(0)=0$. It follows that $\rtt$ induces a
homomorphism:
$$
\rtt':N_{K}(\cppsnb)=N_{K}(\g^s) \, \longrightarrow \,
\mathfrak{S}\bigl(\Sppm(s) \smallsetminus \{0\}\bigr), \quad
k \mapsto \rtt'(k),
$$
with the convention that $\Sppm(s) \smallsetminus \{0\} =
\Sppm(s)$ when $0 \notin \spec(s)$.

\begin{lm}\label{transposition}\label{transbis}
  {\rm (1)} Let $c_{0},c_{1}\in \spec(s) \smallsetminus
  \{0\}$ be such that $m(s,c_{0})=m(s,c_{1})$. There exists
  $k\in N_{K}(\cppsnb)$ such that: $\rtt'(k)(\pm
  c_{0})=c_{\pm 1}$, for $i=0,1$, and
  $\rtt'(k)(\pm c)=\pm c$ for all $\pm c\in\Sppm(s)
  \smallsetminus\{\pm c_{0},\pm c_{1}\}$.
  \\
  {\rm (2)} A permutation $\gamma$ of
  $\Sppm(s)\smallsetminus\{0\}$ belongs to
  $\rtt'(N_{K}(\g^{s}))$ if and only if
$$
m(s,\pm c)=m(s,\gamma(\pm c)) \ \; \text{for all $\pm c\in
  \Sppm(s) \smallsetminus \{0\}$}.
$$
In particular, for such a permutation $\gamma$ there exists
$k\in N_{K}(\g^s)$ such that
$$
k.\g_{s,\pm c}=\g_{s,\gamma(\pm c)}
$$
where $\gamma$ is, if necessary, extended to $\Sppm(s)$ by
$\gamma(0)=0$.
\end{lm}

\begin{proof}
  (1) Recall that $s\in\af_{\Q}$ is in the positive Weyl
  chamber defined by $B$.  Therefore, for $i=0,1$,
  $I_{i}:=\{j\mid c_{i}=\varpi_{j}(s)\} \subset [\![1,N]\!]$
  is an interval; set $I_{i}=[\![d^{1}_{i},d^{2}_{i}]\!]$ .
  In the case $\pm c_{0}=\pm c_{1}$ the element $k=\Id$
  obviously works.  Otherwise, we may replace $c_{i}$ by
  $-c_{i}$ to ensure that $d^{2}_{i}\leqslant N_b\leqslant
  N/2$ and we define a permutation $\gamma \in \Sf_N$ by:
$$
\gamma(j):= \left\{\begin{array}{l l}
    j-d^{1}_{i}+d^{1}_{1-i} & \mbox{ if } j\in I_{i};\\
    j& \mbox{ if } j\leqslant (N+1)/2 \mbox{ and } j\notin
    I_{1}\cup I_{2};\\
    N+1-\gamma(N+1-j)& \mbox{ if } j\gnq (N+1)/2.
  \end{array}\right.
$$ 
One has: $\varpi_{j}(s) = \pm c_{1- i}$ if $\varpi_j(s)= \pm
c_i$, $i=0,1$ and $\pm \varpi_{\gamma(j)}(s)= \pm
\varpi_{j}(s)$ otherwise.  Denote by $w \in W=W(\g,\h')
\cong \Sf_N$ the element corresponding to the permutation
$\gamma$, hence $w.\varpi_{j}=\varpi_{\gamma(j)}$.
From~\eqref{thetaeps} one deduces that:
$$
\theta(w.\varpi_{j})=\left\{\begin{array}{ll}
    \varpi_{N+1-\gamma(j)}=\varpi_{\gamma(N+1-j)}=
    w.\theta(\varpi_{j})
    \mbox{ if } \min(j,N+1-j)\leqslant N_b; \\
    \varpi_{\gamma(j)}=w.\theta(\varpi_{j}) \mbox{
      otherwise.}\end{array}\right.
$$ 
This implies $\theta\circ w(\alpha)=w\circ \theta(\alpha)$
for all $\alpha\in R(\g,\h')$; thus $\theta$ commutes with
$w$, i.e.~$w \in W_\sigma$ in the notation of
\S\ref{basis}. By~Remark~\ref{Klift}(2) there exists $k\in
K$ acting like $w$ on $\h'$.  Therefore $k\in
N_{K}(\cggsnb)$, $\rtt(k) =\gamma$ and $k$ has the desired
properties.
\\
(2) It suffices to write an element of
$\Sf\bigl(\Sppm(s)\smallsetminus \{0\}\bigr)$ as a product
of transpositions and to apply part~(1).
\end{proof}

If $x=t+n\in\g^s$ we write $x= \sum_{c}x_{s,c} = \sum_{c}
(t_{s,c} + n_{s,c})$ where $t_{s,c} + n_{s,c}$ is the Jordan
decomposition of $x_{s,c} \in \wfr_{s,c}$ (thus $n_{s,c}$ is
the nilpotent part of $x_{s,c}$).

We first state consequences of Lemma~\ref{fondAIII} for a
nilpotent element $x=n \in \pp^s$.  As $\theta$ sends
$n_{s,c}$ onto $-n_{s,-c}$, the Young diagram of $n_{s,c}\in
\wfr_{s,c}$ is the same as the Young diagram of
$n_{s,-c}\in\wfr_{s,-c}$.  Morevover, the
$(K^s)^{\circ}$-orbit of $n$ in $\pp^s$ is characterized by
the Young diagrams of the $n_{s,c}$ for $c\neq 0$ and the
$ab$-diagram of $n_{s,0}$.

\begin{lm}\label{xdr}
  Let $x=t+n$ and $x'=t'+n'$ be $G$-conjugate elements of
  $\pp$ with $t,t'\in\cpps$.  Then $n_{s,0}$ and $n_{s,0}'$
  have the same Young diagram.  Furthermore, if $x$ and $x'$
  are $K$-conjugate, $n_{s,0}$ and $n'_{s,0}$ have the same
  $ab$-diagram.
\end{lm}

\begin{proof}
  If $m(s,0)\leqslant1$ one has $n_{s,0}=n'_{s,0}=0$; we
  will therefore assume that $0\in \spec(s)$ and
  $\wfr_{s,0}' = \mathfrak{sl}(V_{s,0}) \ne \{0\}$.
  One can define equivalence relations $\mathcal R$ and
  $\mathcal R'$ on $\spec(s)$ as follows.  Say that
  $c\mathcal{R}d$ if the two following conditions are
  satisfied: $\wfr_{s,c}$ is isomorphic to $\wfr_{s,d}$,
  {i.e.}~$m(s,c)=m(s,d)$, and $n_{s,c},n_{s,d}$ have the
  same Young diagram.  The relation $\mathcal R'$ is defined
  similarly with $n'$ instead of $n$.  As observed above,
  the elements $c$ and $-c$ are in the same equivalence
  class. Consequently, the class containing $0$ is the only
  class, for $\mathcal R$ or $\mathcal R'$, having odd
  cardinality.
  \\
  Since $t,t' \in \cggs$, there exists $g\in N_{G}(\g^s)$
  such that $g.x'=x$ and we can set $\gamma:=\rtt(g)$.  One
  then has $n_{s,\gamma(c)}=g.n_{s,c}'$, therefore $\gamma$
  sends each $\mathcal{R}'$-equivalence class to an
  $\mathcal{R}$-equivalence class.  Thus, as the cardinality
  of the equivalence class of $\gamma(0)$ is odd,
  $\gamma(0)\mathcal{R}0$, $g.n_{s,0}'=n_{s,\gamma(0)}$ and
  $n_{s,0}, n_{s,0}'$ have the same Young diagram. This
  proves the first statement.
  \\
  Now assume that $g\in K$, hence $g\in N_{K}(\g^s)$.  We
  have already shown before Lemma \ref{transposition} that,
  in this situation, $\gamma(0)=0$.  Thus
  $g.n_{s,0}=n_{s,0}'$ with $g\in K$, as desired.
\end{proof}

Let $y=t+n\in J\cap\pp$. Then $(\g_{t,0},\kk_{t,0})$ is
either $(0)$ or a reductive factor of type AIII.  By
Lemma~\ref{fondAIII} the parameter of this factor is the
same as the parameter of $(\g,\kk)$, thus it does not depend
on the choice of $y \in J\cap\pp$.  Recall that $n_{t,0}$ is
the component of $n$ lying in $\g_{t,0} = \wfr_{t,0}$ and
define $\Gamma^{\Phi}(y)$
to be the $ab$-diagram of $n_{t,0}$ in
$(\g_{t,0},\kk_{t,0})$.  Remark that one can recover the
$ab$-diagram of $n_{t,0}$ in $(\g,\kk)$ by adding to
$\Gamma^{\Phi}(y)$ some pairs of rows of length~$1$, one row
beginning by $a$ and the other by $b$.

\begin{prop} \label{cft} \label{cftt} {\rm (1)} Let
  $x^{1},x^{2} \in J\cap\pp$. The following conditions are
  equivalent
  \[
  \text{{\rm (i)}
    $\Gamma^{\Phi}(x^{1})=\Gamma^{\Phi}(x^{2})$\, ; \quad
    {\rm (ii)} $J_{K}(x^{1})=J_{K}(x^{2})$.}
  \]
  Set $\Gamma^{\Phi}(J_K(x)) := \Gamma^{\Phi}(x)$ for $x \in
  J\cap\pp$.

  \noindent {\rm (2)} The map $J_1 \mapsto
  \Gamma^{\Phi}(J_1)$ gives an injection from the set of
  $J_{K}$-classes contained in $J\cap \pp$ to the set of
  admissible $ab$-diagrams for the symmetric pair
  $(\g_{s,0},\kk_{s,0})$.
\end{prop}

\begin{proof}
  (1) Write $x^{i}=t^{i}+n^{i}$ for $i=1,2$.  By
  Lemma~\ref{Jclass2} there exists $k^{i}\in K$ such that
  $k^{i}.t^{i}\in\cpps$.  Observe that
  $\Gamma^{\Phi}(k^{i}.x^{i})=\Gamma^{\Phi}(x^{i})$ and
  $J_{K}(k^{i}.x^{i})=J_{K}(x^{i})$, therefore we may assume
  that $x^i \in \g^s$ and $t^{i}\in\cpps$ for $i=1,2$.  We
  may also assume that $m(t^1,0)=m(t^2,0) \geqslant 1$,
  otherwise each $n^{i}_{t^i,0} =0$ is zero and the
  equivalence is clear.
  \\
  As $n^{i}_{t^i,0}$ belongs to the unique simple factor of
  type~AIII of $(\g^{t^{i}},\kk^{t^{i}})$, one has
  $n^{i}_{t^i,0}\in\wfr_{s,0}$, thus $n^i_{t^i,0}=
  n^i_{s,0}$ and we can set
  $n^{i}_{0}:=n^{i}_{t^i,0}=n^{i}_{s,0}$. For $0 \ne
  c\in\spec(s)$, set $n^{i}_{c}:=n^{i}_{s,c}$.  Recall that
  the $J_{K}$-class of $x^{i}$ is
  $J_{K}(x^{i})=K.(\cpps+n^i)$.
  \\
  (ii) $\imply$ (i): By hypothesis there exists an element
  of $K.(\cpps+n^{1})$ which is $K$-conjugate to
  $x^2$. Lemma~\ref{xdr} then shows that $n^{1}_{0}$ has the
  same $ab$-diagram as $n^{2}_{0}$ for the pair $(\g,\kk)$,
  which implies that
  $\Gamma^{\Phi}(x^{1})=\Gamma^{\Phi}(x^{2})$ (cf.~remark
  above).
  \\
  (i) $\imply$ (ii): Suppose that $n^{1}_{0}$ and
  $n^{2}_{0}$ have the same $ab$-diagram in
  $(\g_{s,0},\kk_{s,0})$. We want to show that $n^{1}$ is
  $N_{K}(\g^s)$-conjugate to $n^{2}$.
  Observe that $n^1_{0}$ and $n^2_{0}$ have the same orbit
  under the group $K_{s,0}$, where we set $K_{s,\pm
    c}:=(G_{s,\pm c}\cap K)^{\circ}$.
  As $n^1$ is $G$-conjugate to $n^2$ there exists $g\in
  N_{G}(\g^s)$ such that $g.n^1_{c}=n^2_{\gamma(c)}$, which
  defines $\gamma = \rtt(g) \in
  \mathfrak{S}\bigl(\spec(s)\bigr)$. Since $n^{i}_{c},
  n^{i}_{-c}$ have the same diagrams for all $c$, there
  exists $\gamma' \in \mathfrak{S}\bigl(\spec(s)\bigr)$ such
  that:
$$
\text{$\wfr_{s,c}\cong\wfr_{s,\gamma'(c)}$, \quad $n^1_{c}$
  has the same diagram as $n^2_{\gamma'(c)}$, \quad
  $\gamma'(-c)=-\gamma'(c)$,}
$$ 
for all $c \in \spec(s)$.  The permutation $\gamma'$ fixes
$0$ and induces $\gamma'' \in
\mathfrak{S}\bigl(\Sppm(s)\bigr)$.  Lemma~\ref{transbis}(2)
gives an element $k\in N_{K}(\g^s)$ such that $k.\g_{s,\pm
  c}=\g_{s,\pm \gamma'(c)}$ for $c\in \Sppm(s)$.
Set $n^3:=k.n^1$; then $n^3_{c}$ has the same diagram as
$n^{2}_{c}$ for all $c \ne 0$, and the same $ab$-diagram
when $c=0$. By the results on type~A0, $n^3_c+n^3_{-c}$ and
$n^2_{c}+n^2_{-c}$ are $K_{s,\pm c}$-conjugate for $c\neq
0$.  This proves the existence of $k'\in Z_{K}(\cpps)\subset
N_{K}(\g^s)$ such that $k'.n^3=n^2$ and $k'k.n^1=n^2$.  In
particular, $K.(\cpps+n^1)=K.(\cpps+n^2)$ and the result
follows.
\\
(2) is an obvious consequence of (1).
\end{proof}


\subsubsection{Case AIII~(2)} \label{JclassAIIIb}

We continue with the same notation. Thus: $e \in \g = \gl_N$
is a nilpotent element, the partition of $N$ associated to
$\Od:=G.e$ is denoted by
$\bolda=(\lambda_{1}\geqslant\dots\geqslant\lambda_{\delta_{\Od}})$,
$\Phi : [\![1,\delta_{\Od}]\!]  \longrightarrow\{a,b\}$ is
an arbitrary function and
$(\g,\theta)=(\g,\kk,\pp)=(\g,\kk^{\Phi},\pp^{\Phi})$ is the
symmetric Lie algebra defined in \S\ref{involAIII}, hence $e
\in \pp=\pp^{\Phi}$.  As above, $S=S_G$ is the $G$-sheet
containing $e$ and $J$ is a $J_G$-class of $S$ intersecting
$\pp$. Recall from section~\ref{Ksheet} that the set
$\{g.e\}_{g \in \mathsf{Z}}$ parameterizes the $K$-orbits
$\Od_{g.e} := K.(g.e)$ contained in $\Od \cap \pp$. We aim to
show that the following condition defined in section \ref{Ksheet} holds (see Proposition~\ref{desired}):
 \begin{equation}
  \mbox{For each }J_{K}\mbox{-class }J_{1}\mbox{ in
  }S_{G}\cap\pp \mbox{, there exists 
    $g\in \mathsf{Z}$ such that }
  J_{1}\mbox{ is well-behaved w.r.t. }\Od_{g.e}. 
  \tag{$\clubsuit$}
\end{equation}

Let $\Gamma_{1}:=\Delta(\Phi)$ be the admissible
$ab$-diagram associated to $e\in\pp^{\Phi}$ and let $J_{1}
\subset J\cap\pp$ be a $J_{K}$-class. By Theorem~\ref{unif}
and Lemma~\ref{GeIH} the conditions \eqref{heart} and
\eqref{sstar} are satisfied; therefore,
Lemma~\ref{Jclass4}(iii) can be applied in this
situation. Let $J_{2}$ be given by this lemma (for $g=\Id)$,
thus $J_{2} \subset J$ is a $J_K$-class which is well
behaved w.r.t~$\Od_{e}$.  Set $Y:=J_{2}\cap(e+\XP)\subset
J\cap(\qq\cap\pp)^{\bullet}$; as observed in
Remark~\ref{flattenremark}, we have:
\begin{equation}\label{dimY}
  \dim Y=\dim J\cap\pp - m.
\end{equation}
Let $s$ be the semisimple part of an element of $J \cap \pp$
and recall that $\Gamma^{\Phi}(J_1)$,
resp.~$\Gamma^{\Phi}(J_{2})$, is the admissible
$ab$-diagram, for $(\g_{s,0},\kk_{s,0})$, associated to
$J_{1}$, resp. $J_{2}$, by Proposition~\ref{cftt}(2).  We
are going to compare these diagrams with $\Gamma_{1}$ in
order to obtain an element $g.e$ ($g\in
\mathsf{Z}$) 
such that $J_{1}$ is well behaved w.r.t~$\Od_{g.e}$.

Let $\qq = \bigoplus_i \qq_i$ be as in \ref{settings} and
$x=s+n$ be an element of $J\cap(\qq\cap\pp)^{\bullet}$,
cf.~Corollary~\ref{corounif}.  Recall that we write $n =
\sum_{i=1}^{\delta_{\Od}} n_{i}$ with $n_{i}\in \qq_{i}$.
Let $\Od' \subset \g_{s,0}$ be the nilpotent orbit
$G_{s,0}.n_{s,0}$ and let
$\boldsymbol{\mu}=(\mu_{1}\geqslant\dots\geqslant
\mu_{\delta_{\Od'}})$ be the associated partition of
$m(s,0)$.  Remark that the shape of the Young diagram
underlying $\Gamma^{\Phi}(J_1)$ or, equivalently,
$\Gamma^{\Phi}(J_2)$, is given by $\boldsymbol{\mu}$.
 
On the other hand $n = \sum_{c \in \spec(s)} n_{s,c}$ with
$n_{s,c} \in \mathfrak{w}_{s,c}$ and, by
Corollary~\ref{fondAIIIbis}, one can write $n_{s,0}=\sum_i
n_{i,s,0}$ where each $n_{i,s,0} \in \qq_{i,s,0} \cap
\pp^{\Phi}$ is regular.  This yields in particular that
$\delta_{\Od'}\leqslant\delta_{\Od}$.  We can therefore
define a map
\begin{equation}
  \label{flatphi}
  \flat^{\Phi}(x) : [\![1,\delta_{\Od'}]\!] \longrightarrow
  \{a,b\}
\end{equation}
where $\flat^{\Phi}(x)(i)$ is the first symbol of the one
row ab-diagram of $n_{i,s,0}\in\qq_{i,s,0}\cap\pp^{\Phi}$.
Observe that when $\lambda_{i}$ is odd,
Corollary~\ref{fondAIIIbis}(3-4) yields
\begin{equation} \label{ddd} \mu_{i}\equiv 1 \bmod{2} \ \,
  \text{and $\flat^{\Phi}(x)=\Phi(i)$ for all $x \in
    J\cap(\qq\cap\pp)^{\bullet}$}.
\end{equation}
It is not difficult to see that the $ab$-diagram
$\Delta(\flat^{\Phi}(x))$ associated to the function
$\flat^{\Phi}(x)$, see~\S\ref{involAIII}, coincides with the
$ab$-diagram $\Gamma^{\Phi}(x)$ defined before
Proposition~\ref{cft}. Thus, according to the previous
notation:
$$
\Delta(\flat^{\Phi}(y)) =
\Gamma^{\Phi}(y)=\Gamma^{\Phi}(J_{2}) \ \; \text{for all
  $y\in Y\subset J_{2}\cap(\qq\cap\pp)^{\bullet}$}.
$$

\begin{remark}
  One may have $\flat^{\Phi}(x) \neq \flat^{\Phi}(x')$ with
  $K.x'=K.x$.  Such examples can be easily obtained by
  permuting blocks $\qq_{i}$ and $\qq_{j}$ such that
  $\lambda_{i}=\lambda_{j}$.
\end{remark}

Now, let $\Psi':[\![1,\delta_{\Od'}]\!]\longrightarrow
\{a,b\}$ be a map such that its associated $ab$-diagram,
$\Delta(\Psi')$, is equal to $\Gamma^{\Phi}(J_{1})$.
Under this notation, we want to construct
$\Psi:[\![1,\delta_{\Od}]\!] \rightarrow \{a,b\}$ such that
$\Psi'=\flat^{\Psi}(y)$ and
$$
\Delta(\flat^{\Psi}(y)) =
\Gamma^{\Psi}(y)=\Gamma^{\Phi}(J_{1}) \ \; \text{for all
  $y\in Y$}.
$$
Fix $y \in Y$ and define $\Psi$ as follows:
$$
\left\{\begin{array}{ll}\Psi(i)=\Phi(i) \ \, \text{if
      $\Psi'(i)=\flat^{\Phi}(y)(i)$ {and} $i\leqslant
      \delta_{\Od'}$;}
    \\
    \Psi(i)\neq\Phi(i) \ \, \text{if
      $\Psi'(i)\neq\flat^{\Phi}(y)(i)$ and $i\leqslant
      \delta_{\Od'}$};
    \\
    \Psi(i)=\Phi(i) \ \, \text{for
      $i\in[\![\delta_{\Od'}+1,\delta_{\Od}]\!]$}.
  \end{array}\right.
$$
By~\eqref{ddd}, for each $i \in [\![1,\delta_{\Od}]\!]$ such
that $\lambda_i$ is odd one has $\Psi'(i)=\Psi(i)$.

\begin{lm}\label{dde}
  The $ab$-diagram $\Gamma_{2}:=\Delta(\Psi)$ is admissible
  for the symmetric pair $(\g,\kk^\Phi)$.
\end{lm}

\begin{proof}
  The only thing to prove is that $N'_{a}$ (resp.~$N'_{b}$),
  the number of $a$ (resp.~$b$) in $\Gamma_2$ is equal to
  $N_{a}$ (resp.~$N_{b}$).  This is equivalent to showing
  that $N'_{a}-N'_{b}=N_{a}-N_{b}$. From~\eqref{ddd} and the
  definition of $\Psi$ one deduces:
  \begin{align*}
    N_{a}-N_{b}&-(N'_{a}-N'_{b}) &\\
    = & \; \# \{ i \mid \Phi(i)=a \ \, \text{and
      $\lambda_{i}\equiv 1\bmod{2}$}\} - \#\{i\mid \Phi(i)=b
    \ \, \text{and $\lambda_{i}\equiv 1\bmod{2}$} \}&
    \\
    & -\#\{i\mid \Psi(i)=a \ \, \text{and $\lambda_{i}\equiv
      1\bmod{2}$} \}+\#\{i\mid \Psi(i)=b \ \, \text{and
      $\lambda_{i}\equiv 1\bmod{2}$} \}&
    \\
    = &\; \#\{i \mid \flat^{\Phi}(y)(i)= a \ \, \text{and
      $\lambda_{i}\equiv 1\bmod{2}$} \}-\#\{i \mid
    \flat^{\Phi}(y)(i)=b \ \, \text{and $\lambda_{i}\equiv
      1\bmod{2}$} \}&
    \\
    &-\#\{i \mid \Psi'(i)=a \ \, \text{and
      $\lambda_{i}\equiv 1 \bmod{2}$} \}+\#\{i \mid
    \Psi'(i)=a \ \, \text{and $\lambda_{i}\equiv 1\bmod{2}$}
    \}&.
  \end{align*}
  Since the diagrams $\Delta(\Psi')= \Gamma^{\Phi}(J_{1})$
  and $\Delta(\flat^{\Phi}(y))= \Gamma^{\Phi}(J_{2})$ are
  admissible in the same symmetric pair
  $(\g_{s,0},\kk_{s,0})$, the previous equation implies that
  $N_{a}-N_{b}-(N'_{a}-N'_{b}) =0$.
\end{proof}

From the function $\Psi$ one constructs, as in
\S\ref{involAIII}, the symmetric Lie algebra
$(\g,\kk',\pp')= (\g,\kk^{\Psi}, \pp^{\Psi})$ with
$V=V_{a}^{\Psi} \bigoplus V_{b}^{\Psi}$.
Since $\qq_{i}\cap\kk$ and $\qq_{i}\cap\kk'$ are both
spanned by even sup- and sub-diagonals, we obtain the same
symmetric Lie subalgebras $(\qq_{i},\qq_{i}\cap\kk,
\qq_{i}\cap \pp)=(\qq_{i},\qq_{i}\cap\kk', \qq_{i}\cap
\pp')$. It follows that the function $\flat^{\Psi}(z) :
[\![1,\delta_{\Od'}]\!] \rightarrow \{a,b\}$ is well defined
for all $z \in J\cap(\qq\cap\pp)^{\bullet} =
J\cap(\qq\cap\pp')^{\bullet}$.

Recall that $y \in(\qq\cap\pp)^{\bullet} =
(\qq\cap\pp')^{\bullet}$, thus $\flat^{\Psi}(y)$ is defined;
we claim that $\flat^{\Psi}(y)=\Psi'$. Set $V_a^\Phi(i) :=
\langle v_j^{(i)} : 1 \le j \le \lambda_i \rangle \cap
V_a^\Phi$, $V_b^\Phi(i) := \langle v_j^{(i)} : 1 \le j \le
\lambda_i \rangle \cap V_b^\Phi$, and define $V_a^\Psi(i),
V_b^\Psi(i)$ accordingly.  Observe that: $V_a^\Phi(i) =
V_a^\Psi(i)$, $V_b^\Phi(i) = V_b^\Psi(i)$ when $\Phi(i)=
\Psi(i)$, and $V_a^\Phi(i) = V_b^\Psi(i)$, $V_b^\Phi(i) =
V_a^\Psi(i)$ otherwise. Suppose that $\Phi(i) \ne \Psi(i)$;
by definition of $\flat^\Phi, \flat^\Psi$ one has
$\flat^\Phi(y)(i) \neq \flat^\Psi(y)(i)$, therefore
$\flat^\Psi(y)(i) =\Psi'(i)$ by definition of $\Psi$. The
equality $\flat^\Psi(y)(i) =\Psi'(i)$ is obtained in the
same way when $\Phi(i) = \Psi(i)$.  The equality
$\flat^{\Psi}(y)=\Psi'$ implies in particular
$\Gamma^{\Psi}(y)=\Gamma^{\Phi}(J_{1})$.

We can now show that the condition~\eqref{club} is satified
in type AIII:

\begin{prop}\label{desired}
  For each $J_{K}$-class $J_{1} \subset J\cap\pp$, there
  exists $g\in \mathsf{Z}$ such that $J_{1}$ is well-behaved
  w.r.t.~$\Od_{g.e}$
\end{prop}

\begin{proof}
  By Lemma~\ref{dde} one can find $g'\in \GL_{N}$ such that
  $g'.V_{a}^{\Psi}=V_{a}^{\Phi}$ and
  $g'.V_{b}^{\Psi}=V_{b}^{\Phi}$.  Then, $g=\rho(g')\in G$
  induces an isomorphism of symmetric Lie algebras between
  $(\g,\kk',\pp')$ and $(\g,\kk,\pp)$ (cf.~end of
  \ref{involAIII}).  Since $e\in\qq\cap\pp$, one has
  $e\in\pp'$ and $g.e\in\pp$; therefore, up to conjugation
  by an element of $K^{\Phi}$ (the algebraic group
  associated to $\kk=\kk^{\Phi}$), we may assume that $g\in
  \mathsf{Z}$ (see~\S\ref{involreduc}).  These remarks imply
  that $\Gamma^{\Phi}(g.y) = \Gamma^{\Psi}(y) =
  \Gamma^{\Phi}(J_1)$ is the $ab$-diagram associated to
  $J_{1}$ with respect to $\Phi$,
  cf.~Proposition~\ref{cftt}.
  From $Y\subset \qq \cap \pp = \qq\cap\pp'$ one gets
  $g.y\in g.Y\subset J\cap\pp$ and, since $g.Y$ is
  irreducible, one has $g.Y\subset J_1$. In particular, $g.Y
  \subset g.(e+\XP(g.\sS)) \cap \pp \subset g.e +
  \XP(g.\sS)$ is contained in $J_{1}$ with $\dim g.Y = \dim
  J_{1}-m$.  The result then follows from
  Remarks~\ref{dimGY} and~\ref{flattenremark}.
\end{proof}

\section{Main theorem and remarks}
\label{lastsection}

\subsection{Main theorem}
\label{mainthm}

In this subsection we give the description of the $K$-sheets
when $(\g,\theta)$ is of type~A.  Thus, $\g \cong \gl_{N}$
and $(\g,\theta)=(\g,\kk,\pp)$ is a symmetric Lie algebra.
Suppose that $S_{G} \subset \g$ is a $G$-sheet intersecting
$\pp$.
In~\eqref{SKge}, cf.~Remark~\ref{SKO}, we have defined, for
any nilpotent element $e\in S_{G}\cap \pp$ and any normal
\Striplet $\sS=(e,h,f)$, the following subvariety of $S_G
\cap \pp$:
$$
S_{K}(S_{G},\sS) = S_{K}(\sS) = S_{K}(K.e) :=
\overline{K.(e+\XP(\sS))}^{\bullet}.
$$ 
We aim to describe the $K$-sheets and the varieties $S_G
\cap \pp$ in terms of the $S_{K}(K.e)$.

Recall from Remark~\ref{slsheets}(2) that $S_G \cap \pp$ is
smooth; in particular, its irreducible components are
disjoint. The next lemma reduces the study of $K$-sheets to
the study of irreducible components of $S_{G}\cap\pp$; this
result may be false in some cases of type~0, see the remark
previous to Corollary~\ref{cor000}.

\begin{lm}\label{final1}
  Let $S_{G}$ be a $G$-sheet of $\g$ intersecting $\pp$,
  then each irreducible component of $S_{G}\cap\pp$ is a
  $K$-sheet.
\end{lm}

\begin{proof}
  Let $S_{K}$ be an irreducible component of
  $S_{G}\cap\pp$. As $S_{G}\cap\pp$ is a union of $K$-orbits
  of same dimension, there exists a $K$-sheet $S_{K}'$
  containing $S_{K}$. Recall that, as $\g\cong \gl_N$, two
  distinct $G$-sheets are disjoint (see the discussion
  previous to Corollary~\ref{cor000}). It follows that
  $S_{K}'$ must be contained in $S_{G}$ and, therefore, in
  $S_{G}\cap\pp$.  This proves that $S_{K}'=S_{K}$, hence
  the result.
\end{proof}

\begin{thm}\label{final}
{\rm (i)} The $K$-sheets of $\pp$ are disjoint, they are
exactly the smooth irreducible varieties $S_{K}(\Od_K)$
where $\Od_K \subset \pp$ is a nilpotent $K$-orbit.
\\
{\rm (ii)} Let $S_{G}$ be a $G$-sheet intersecting $\pp$.
Then, $S_G\cap\pp$ is a smooth equidimensional variety and
each of its irreducible component is some $S_{K}(\Od_K)$,
where $\Od_K \subset S_G\cap\pp$ is a nilpotent $K$-orbit.\\
{\rm (iii)} Let $S_K\subset \pp$ be a $K$-sheet and $e$ be a
nilpotent element of $S_K$ embedded in a normal \Striplet
$\sS=(e,h,f)$.  Define $Y$ by $e+Y:=S_K\cap(e+\pp^f)$. Then
$S_K=\overline{K.(e+Y)}^{\bullet}$.
\end{thm}

\begin{proof}
  We need to summarize the conditions introduced
  in~\S\ref{Ksheet} and proved in cases~AI, AII and AIII:
  \eqref{heart}~has been proved in Theorem~\ref{unif} (with
  proof in Proposition~\ref{unifAIII} for type AIII);
  \eqref{diamond} was established in Remark~\ref{diamondAI}
  (types AI, AII) and Remark~\ref{diamondAIII} (type AIII);
  \eqref{club} has been obtained in Corollary~\ref{clubAI}
  (types AI, AII) and Proposition~\ref{desired} (type AIII).
  \\
  Claim (ii) is therefore consequence of
  Remark~\ref{slsheets}(2) (or equivalently
  Proposition~\ref{sheetsmoothness}) and
  Theorem~\ref{equidim}.\\
  Recall that $G$-sheets are disjoint. Then, from
  $\pp^{(m)}\subset\g^{(2m)}$, it follows that each
  $K$-sheet is contained in a unique $G$-sheet.  So, (i) is
  consequence of (ii) and
  Lemma~\ref{final1}.\\
  Under the hypothesis in (iii), $e$ belongs to $S_K$, hence
  $S_K=S_K(\sS)$ is the unique $K$-sheet containing $e$.
  Therefore, $$e+Y\subset e+X(\sS)\cap\pp\subset
  S_K(\sS)\cap(e+\pp^f)=e+Y.$$ The assertion in (iii) then
  follows from the definition of $S_K(\sS)$.
\end{proof}

\begin{remark}
  One can be more precise about the number of irreducible
  components of $S_G\cap\pp$, see~\S\ref{lastcomments}(4).
\end{remark}

Fix a sheet $S_{G}$ intersecting $\pp$. One can compute the
dimension of $S_G \cap \pp$ in terms of the partitions
associated to the nilpotent orbit $\Od \subset S_{G}$.  Let
$\bolda=(\lambda_{1}\geqslant\dots\geqslant
\lambda_{\delta_{\Od}})$ and $\tilde{\bolda} =
(\tilde{\lambda}_1 \geqslant \dots \geqslant
\tilde{\lambda}_{\delta_{\lf}})$ be the partitions of $N$
defined in~\ref{settings}.  Pick $e\in\Od\cap\pp$ and recall
that if $\sS=(e,h,f)$ is a normal \Striplet we set
$S_K(K.e):= \overline{K.(e+\XP(\sS))}^{\bullet}$.

\begin{prop}\label{dimSK}
  Under the previous notation one has
$$
\dim S_{G}\cap\pp = \dim S_K(K.e)= \lambda_1+\frac
12\Bigl(N^2-\sum_{i=1}^{\lambda_1} \tilde{\lambda}_i^2\Bigr)
$$ 
in types {\rm AI} and {\rm AII}, and
\[
\dim S_{G}\cap\pp=\dim S_K(K.e) =
\sum_{i=1}^{\delta_\Od}\left\lfloor\frac{\lambda_i-\lambda_{i+1}}2
\right\rfloor+\frac 12\Bigl(N^2-\sum_{i=1}^{\lambda_1}
\tilde{\lambda}_i^2\Bigr).
\]
in type {\rm AIII}.
\end{prop}

\begin{proof}
  Recall that $\dim G.e=N^2-\sum_{i=1}^{\lambda_1}
  \tilde{\lambda}_i^2$, see~\cite{CM}, and $\dim
  K.e=\frac{1}{2}\dim G.e$. By Theorem~\ref{final} and
  Remark~\ref{dimGY} one has
  \[
  \dim S_G \cap \pp = \dim S_K(K.e) = \dim K.e + \dim
  \XP(\sS).
  \]

  We know that $\XP(\sS) = X(\sS)$ in types AI and AII,
  cf.~Theorem~\ref{unif}.  Therefore,
  Remark~\ref{epsilonfinite} and equation~\eqref{alpha}
  yield $\dim S_G \cap \pp = \dim K.e + \dim X(\sS)= \dim
  K.e+\dim \tf= \dim K.e+ \lambda_1$. Hence:
$$
\dim S_G \cap \pp = \lambda_1+\frac
12\Bigl(N^2-\sum_{i=1}^{\lambda_1}
\tilde{\lambda}_i^2\Bigr).
$$ 

Since the morphism $\varepsilon$ is quasi-finite, see
Remark~\ref{epsilonfinite}, one has $\dim \XP(\sS)= \dim
\cc$ in type~AIII by Proposition~\ref{corcc}. It then
follows from \eqref{dimd} that
$$
\dim S_G \cap \pp =\dim K.e + \dim \cc = \dim K.e+
\sum_{i=1}^{\delta_\Od} \textstyle{\left\lfloor
    \frac{\lambda_i-\lambda_{i+1}}2\right\rfloor}.
$$
Thus
$$
\dim S_G \cap \pp = \sum_{i=1}^{\delta_\Od}
\left\lfloor\frac{\lambda_i-\lambda_{i+1}}2 \right\rfloor+
\frac 12\Bigl(N^2-\sum_{i=1}^{\lambda_1}
\tilde{\lambda}_i^2\Bigr)
$$
as desired.
\end{proof}

\subsection{Remarks and comments}
\label{lastcomments}
We collect here various remarks and comments about the
results obtained in the previous sections. To keep the
length of the exposition reasonable we will not give full
details of the proofs, leaving them to the interested
reader.
 
If not otherwise specified, we assume that $(\g,\theta)
\cong (\gl_N,\theta)$ is of type~AI-II-III; we then retain
the notation of Section~\ref{AAA} and \S\ref{mainthm}. In
particular, $S_G \subset \g$ is a $G$-sheet which intersects
$\pp$, $\Od=G.e$, $e \in S_G \cap \pp$, is the nilpotent
orbit contained in $S_G$, $\bolda = (\lambda_{1}, \dots,
\lambda_{\delta_{\Od}})$ is the associated partition of $N$,
$\mathbf{v}$ is the basis of $V$ introduced in
\S\ref{settings}, $e+X= e+X(\sS)$, with $\sS =(e,h,f)$, is a
Slodowy slice of $S_G$, $\XP=\XP(\sS)= X \cap \pp$, $\cc
\subset \tf$ is such that $\varepsilon(e + \cc) = e + \XP$
in case AIII
(cf.~\eqref{defc}), etc.  

For simplicity, we will sometimes assume that $\g=\sln_{N}$.
When this is the case, the above notation refers to their
intersection with $\sln_{N}$.

\smallskip
 
(1) Theorems \ref{unif} and~\ref{final} show that $e+\XP$ is
``almost'' a slice for $S_{G}\cap\pp$, or for a $K$-sheet
contained in $S_{G}$ and containing~$e$, meaning that the
$G$-orbit of any element of $S_{G}\cap\pp$ intersects $e +
\XP$. But, contrary to the Lie algebra case, $e +\XP$ does
not necessarily intersect each $K$-orbit contained in the
given $K$-sheet.
As it is implicitly noticed in \cite{KR}, this phenomenon already occurs, in some cases, for the regular sheet; however, $e+\XP$ is a ``true'' slice when one considers the $G^{\theta}$-action instead of the $K$-action \cite[Theorem 11]{KR}.
On can show that the previous result holds in general for types AI, AII. But, in case AIII, it may happen that $\Aut(\g,\kk).(e+\XP)\subsetneq \overline{K.(e+\XP)}^{\bullet}$ for some $K$-sheets.
This mainly explains why we need to work with the closure of $K.(e+\XP)$ in the whole paper.

\smallskip

(2) Suppose that $(\g,\theta)=(\g,\kk)$ is an arbitrary
reductive symmetric Lie algebra. Recall \cite[39.4]{TY} that
a $G$-sheet containing a semisimple element is called a
\textit{Dixmier sheet}. Similarly, we will say that a
$K$-sheet which contains a semisimple element is a
\textit{Dixmier $K$-sheet}.
 
If $\g$ is semisimple of type~A, all $G$-sheets are Dixmier
sheets, cf. \cite[2.3]{Kr}.
This implies that, for each sheet $S_G$ and \Striplet $\sS
=(e,h,f)$ as in \S\ref{Katsylo}, the set
$e+X(S_{G},\sS)=e+X(\sS)$ contains a semisimple element. For
symmetric pairs of type AI or AII, the $K$-sheets are all of
the form $S_{K}(\sS)= S_K(K.e):=\overline{K.(e+X(\sS))}^{\bullet}$ (cf. Theorems~\ref{unif}
and~\ref{final}); thus, in these cases, any $K$-sheet is a
Dixmier $K$-sheet.
 
In type~AIII there exist $K$-sheets containing no semisimple
element and one can characterize them in terms of the
partition $\bolda$ associated to the nilpotent element $e
\in S_G \cap \pp$ as follows. 

\begin{claim}
  \label{claim1}
  In type {\rm AIII}, a $K$-sheet is Dixmier if and only if
  the partition $\bolda$ satisfies:
  $\lambda_{i}-\lambda_{i+1}$ is odd for at most one
  $i\in[\![1,\delta_{\Od}]\!]$ (where we set
  $\lambda_{\delta_{\Od}+1} :=0$). \qed
\end{claim}
This can be proved by using Propositions \ref{corcc} 
and \ref{cormiha}, Corollary~\ref{corlm} and a study of semisimple elements in $e+\cc$.

Observe that the condition for a $K$-sheet to be Dixmier
depends only on the nilpotent orbit $G.e$ and that
$S_K(K.e)$ is Dixmier if and only if $S_K(K.g.e)$, $g \in
\mathsf{Z}$, is Dixmier.

\smallskip

(3) Recall from Section~\ref{Ksheet} that a nilpotent orbit
of $\g$ is \emph{rigid} when it is a sheet of $[\g,\g]$.
When $\g$ is of type A the only rigid nilpotent orbit is
$\{0\}$.  In other cases it may happen that a rigid orbit
$\Od_1$ contains a non-rigid orbit $\Od_2$ in its closure
(see the classification of rigid nilpotent orbits in
\cite{CM}).  Observe that, since the nilpotent cone is
closed, a sheet containing $\Od_2$ cannot be contained in
the closure of $\Od_1$. One gets in this way some sheets
whose closure is not a union of sheets.  One can ask if
similar facts occur for symmetric pairs $(\g,\kk)$, in
particular when $\g$ is of type~A.
 
Let $(\g,\kk,\pp)$ be a symmetric Lie algebra; a nilpotent
$K$-orbit in $\pp$ which is a $K$-sheet in $\pp\cap[\g,\g]$
will be called rigid.  We remarked in (2) that, in types AI
and AII, each $K$-sheet contains a semisimple element; thus,
$\{0\}$ is the only rigid nilpotent $K$-orbit in these
cases.  Assume that $(\g,\kk,\pp)$ is of type AIII,
$\zz(\g)\subset\kk$, and recall from the proof of
Proposition~\ref{dimSK} (using Remark~\ref{epsilonfinite})
that $S_{K}(K.e)=K.e$ if and only if $\dim \cc=0$. The
arguments given in (2) about $K$-sheets can be adapted to
prove:

\begin{claim}
  \label{claim2}
  The orbit $K.e$ is rigid if and only if the partition
  $\bolda$ satisfies: $\lambda_{i}-\lambda_{i+1}\leqslant 1$
  for all $i\in[\![1,\delta_{\Od}]\!]$. \qed
\end{claim}

\noindent Note that the previous result depends only on the
partition $\bolda$ and not on the $ab$-diagram of $e$.  In
particular, $K.e$ is rigid if and only if each $K$-orbit
contained in $G.e\cap\pp$ is rigid.

\begin{exemple} Consider the symmetric pair
  $(\gl_{6},\gl_{3}\oplus\gl_{3})$ and a rigid $K$-orbit
  $\Od_1$ associated to the partition $\bolda=(3,2,1)$.
  This orbit contains in its closure a nilpotent $K$-orbit
  $\Od_2$ with partition $(3,1,1,1)$, cf.~\cite{Oh2}, but
  $\Od_{2}$ is not rigid.
  \\
  In type AIII, we can construct in this way $K$-sheets
  whose closures are not a union of sheets. \qed
\end{exemple}

\smallskip
 
(4) We have shown in Theorem~\ref{final} that the
irreducible components of $S_G \cap \pp$ are $K$-sheets and
are of the form $S_K(\Od_K)$, where $\Od_K$ is a (nilpotent)
$K$-orbit contained in $\Od:=G.e$. The number of these
irreducible components thus depends on the analysis of the
equality $S_{K}(\Od_{K}^1) = S_{K}(\Od_{K}^2)$ where
$\Od_{K}^1, \Od_{K}^2$ are nilpotent $K$-orbits.  An obvious
necessary condition is $\Od=G.\Od_{K}^1=G.\Od_{K}^2$. 

In cases AI and AII, $S_G\cap\pp$ is irreducible and $G.\Od_{K}^1=G.\Od_{K}^2$ is also a sufficient condition for having $S_{K}(\Od_{K}^1) = S_{K}(\Od_{K}^2)$.
This follows in case AII from $K=G^{\theta}$, hence $\Od \cap\pp=\Od_{K}^1$ (Proposition~\ref{JclassAIbis}), and in case AI from the fact that all sheets are Dixmier.

%
%
The situation in type AIII is more complicated and one can
find $G$-sheets having a nonirreducible intersection with
$\pp$.  The characterization of the equality
$S_{K}(\Od_{K}^1) = S_{K}(\Od_{K}^2)$ is given in
Claim~\ref{SKAIII}. 
We first have to define
the notion of ``rigidified $ab$-diagram''. Let $\Gamma$ be
an $ab$-diagram coresponding to a nilpotent $K$-orbit $\Od_K
\subset \pp$; remove from $\Gamma$ the maximum number of
pairs of consecutive columns of the same length.  The new
$ab$-diagram
obtained in this way is uniquely determined and is called
the the \emph{rigidified} $ab$-diagram deduced from
$\Gamma$, or associated to $\Od_K$. The terminology can be
justified by the following remark: a rigidified $ab$-diagram
corresponds to a rigid nilpotent $K$-orbit in some other
symmetric pair of type AIII.

\begin{claim} \label{SKAIII} The two orbits $\Od_{K}^1$ and
  $\Od_{K}^2$ are contained in the same $K$-sheet,
  {i.e.}~$S_{K}(\Od_{K}^1)=S_{K}(\Od_{K}^2)$, if and only if
  their associated rigidified $ab$-diagrams are equal. \qed
\end{claim}

\begin{exemple}
  Let $(\g,\kk)=(\gl_{8},\gl_{4}\oplus\gl_{4})$ and $\Od$ be
  the nilpotent $G$-orbit with associated partition
  $\bolda=(4,3,1)$.  The set $\Od\cap\pp$ splits into four
  $K$-orbits $\Od_{K}^j$, $1 \le j \le 4$, whose respective
  $ab$-diagrams are
$$
\Gamma(\Od_{K}^1)=\begin{array}{l}abab\\aba\\b\end{array};\;\;
\Gamma(\Od_{K}^2)=\begin{array}{l}abab\\bab\\a\end{array};\;\;
\Gamma(\Od_{K}^3)=\begin{array}{l}baba\\aba\\b\end{array};\;\;
\Gamma(\Od_{K}^4)=\begin{array}{l}baba\\bab\\a\end{array}.
$$
The associated rigidified $ab$-diagrams are, respectively:
$$
\begin{array}{l}ab\\a\\b\end{array};\qquad\qquad\quad
\begin{array}{l}ab\\a\\b\end{array};\qquad\qquad\quad
\begin{array}{l}ba\\a\\b\end{array};\qquad\qquad\quad
\begin{array}{l}ba\\a\\b\end{array}.
$$
The previous result implies that $S_{G}\cap\pp$ is the
disjoint union of $S_{K}(\Od_{K}^1)=S_{K}(\Od_{K}^2)$ and
$S_{K}(\Od_{K}^3)=S_{K}(\Od_{K}^4)$. \qed
\end{exemple}

\smallskip

(5) A natural problem is, using section \S\ref{Ksheet}, to
generalize the results obtained in type~A to other types.
The action of $\varepsilon$ is well described in \cite{IH}
for classical Lie algebras and one may ask if conditions
\eqref{heart}, \eqref{diamond} or \eqref{club} hold in this
case.  Concerning \eqref{heart}, the author made some
calculations when $(\g,\kk)$ is of type~CI.  Im-Hof,
cf.~\cite{IH}, splits this type in three cases that we label
CI-I, CI-II and CI-III.  It is likely that \eqref{heart}
remains true for the first two cases.
In case CI-III one finds the following
counterexample. Consider $(\g,\kk):=(\spn_{6},\gl_{3})$ and
the sheet $S_{G}$ with datum $(\lf,0)$ where $\lf$ is
isomorphic to $\gl_{2}\oplus\spn_{2}$.
Let $e$ and $e'$ be nilpotent elements in $S_{G}\cap\pp$
with respective $ab$-diagrams $\Gamma(e)=
\sscolumn{abab}{ab\phantom{ba}}$ and $\Gamma(e')=
\sscolumn{abab}{ba\phantom{ba}}$.
Embed $e$, resp.~$e'$, in an \Striplet $\sS$,
resp.~$\sS'$. One can show that $\dim \XP(S_{G},\sS)=1$,
$\dim \XP(S_{G},\sS')=2$ and we then get
$G.(e+\XP(S_{G},\sS))\subsetneq G.(e'+\XP(S_{G},\sS'))$,
showing that \eqref{heart} is not satisfied.  Moreover, we
see that the similarity observed in the case $\g=\gl_{N}$
between properties of $\XP(S_{G},g.\sS)$ and
$\XP(S_{G},\sS)$, when $g\in\mathsf{Z}$, is no longer valid.





{
  
}
\vfill


\begin{thebibliography}{999}
    {\selectlanguage{english}
    \bibitem [AK] {AK} A.~Altman and S~Kleiman,
      \emph{Introduction to Grothendieck Duality Theory},
      Lecture Notes in Math., {\bf 146}, Springer-Verlag,
      1970.
    \bibitem[Ar]{Ar} S.~Araki, On root systems and an
      infinitesimal classification of irreducible symmetric
      spaces, \emph{J.~Math. Osaka City Univ.}, {\bf13}
      (1962), 1-34.
    \bibitem[BC]{BC} P.~Bala and R.~W.~Carter, Classes of
      unipotent elements in simple algebraic groups.~II,
      \emph{Math. Proc. Camb. Phil. Soc}, {\bf 80} (1976),
      1-18.
    \bibitem[BK]{BK} W.~Borho and H.~Kraft, \"Uber Bahnen
      und deren Deformationen bei linear Aktionen
      reducktiver Gruppen, \emph{Math. Helvetici.}, {\bf54}
      (1979), 61-104.
    \bibitem[Boh]{Boh} W.~Borho, \"Uber schichten
      halbeinfacher Lie-algebren, \emph{Invent. Math.},
      {\bf65} (1981), 283-317.
    \bibitem [Bou]{Bou} N.~Bourbaki, \emph{Groupes et
        alg\`ebres de Lie. Chapitres 7 et 8}, Hermann, 1975.
    \bibitem[Bro]{Bro} A.~Broer, Decomposition varieties in
      semisimple Lie algebras, \emph{Can. J. Math.}, {\bf50}
      \No5 (1998), 929-971.
   \bibitem[Ca]{Ca} R.~W.~Carter, \emph{Finite Groups of
        Lie Type}. A~Wiley-Interscience Publication. John
      Wiley~\&~Sons, Inc., New York, 1989.
    \bibitem[CM]{CM} D.~H.~Collingwood and W.~M.~McGovern,
      \emph{Nilpotent orbits in semisimple Lie algebras},
      Van Nostrand Reinhold Mathematics Series, New York,
      1993.
    \bibitem [Di1]{Dix1} J.~Dixmier, Polarisations dans les
      alg\`ebres de Lie semi-simples complexes,
      \emph{Bull. Sc. Math.}, {\bf99} (1975), 45-63.
    \bibitem [Di2]{Dix2} J.~Dixmier, Polarisations dans les
      alg\`ebres de Lie II, \emph{Bull. Soc. Math. France},
      {\bf104} (1976), 145-164.
   \bibitem [Dy] {Dy} E.~B.~Dynkin, Semisimple subalgebras
      of semisimple Lie algebras,
      \emph{Amer. Math. Soc. Trans. (2)}, {\bf6} (1957),
      p.111-244.

    \bibitem [GW]{GW} R.~Goodman and N.~R.~Wallach, An
      algebraic group approach to compact symmetric spaces,
      \url{http://www.math.rutgers.edu/~goodman/pub/symspace.pdf}.
      \emph{In}: Representations and Invariants of the
      Classical Groups, \emph{Cambridge University Press},
      1998.
    \bibitem[Har]{Ha} R.~Hartshorne, \emph{Algebraic
        geometry}, Graduate Texts in Mathematics {\bf 52},
      Springer-Verlag, 1977.
    \bibitem[He1]{He1} S.~Helgason, \emph{Differential
        geometry, Lie groups, and symmetric spaces}, Pure
      and applied mathematics, Academic press, 1978.
    \bibitem[He2]{He2} S.~Helgason, Fundamental solutions of
      invariant differential operators on symmetric spaces,
      \emph{Amer. J. Math.}, {\bf86} (1978), 565-601.
    \bibitem[He3]{He3} S.~Helgason, Some results on
      invariant differential operators on symmetric spaces,
      \emph{Amer. J. Math.}, {\bf114} (1992), 789-811.
    \bibitem[IH]{IH} A.~E.~Im Hof, The sheets of classical
      Lie algebra,
      \url{http://pages.unibas.ch/diss/2005/DissB_7184.pdf}
      (2005).
    \bibitem[Iv]{Iv} B.~Iversen, A fixed point formula for
      action of tori on algebraic varieties,
      \emph{Invent. Math.}, {\bf16} (1972), 229-236.

    \bibitem[Kat]{Kat} P.~I.~Katsylo, Sections of sheets in
      a reductive algebraic Lie algebra, \emph{Math. USSR
        Izvestiya}, {\bf20} \No 3 (1983), 449-458.
    \bibitem[Ko]{Ko} B.~Kostant, Lie groups representations
      on polynomial rings, \emph{Amer. J. Math.}, {\bf85}
      (1963), 327-404.
    \bibitem[KR]{KR} B.~Kostant and S.~Rallis, Orbits and
      representations associated with symmetric spaces,
      \emph{Amer. J. Math.}, {\bf93} (1971), 753-809.
    \bibitem[Kr]{Kr} H.~Kraft, Parametrisierung von
      Konjugationklassen in $\sln_n$, \emph{Math. Ann.},
      {\bf234} (1978), 209-220.
    \bibitem[Le]{Le} M. Le Barbier, The variety of reductions for a reductive symmetric pair, preprint,
     arXiv:1005.0746 (2010).
    \bibitem[LS]{LS} G.~Lusztig and N.~Spaltenstein, Induced
      unipotent classes, \emph{J. London Math. Soc. (2)},
      {\bf19} \No3 (1979), 41-52.
    \bibitem[Oh1]{Oh1} T.~Ohta, The singularities of the
      closure of nilpotent orbits in certain symmetric
      pairs, \emph{Tohoku Math. J.}, {\bf38} (1986),
      441-468.
    \bibitem[Oh2]{Oh2} T.~Ohta, The closure of nilpotent
      orbits in the classical symmetric pairs and their
      singularities, \emph{Tohoku Math. J.}, {\bf43} (1991),
      161-211.
    \bibitem[Oh3]{Oh3}T.~Ohta, Induction of nilpotent orbits
      for real reductive groups and associated varieties of
      standard representations, \emph{Hiroshima Math. J.},
      {\bf29} (1999), 347-360.
 \bibitem [Pa]{Pa4} D.~I.~Panyushev, On invariant theory
      of $\theta$-groups, \emph{J.~of Algebra}, {\bf 283}
      (2005), 655-670.
    \bibitem [PY] {PY} D.~I.~Panyushev and O.~Yakimova,
      Symmetric pairs and associated commuting varieties,
      \emph{Math. Proc. Cambr. Phil. Soc.}, {\bf143} (2007),
      307-321.
    \bibitem [Pe]{Pe} D.~Peterson, Geometry of the adjoint
      representation of a complex semisimple Lie Algebra,
      \emph{Ph. D. thesis}, Harvard university (1978).
    \bibitem[PV]{PV} V.~L.~Popov and E.~B.~Vinberg, {\it
        Invariant Theory}.  \emph{In}: Algebraic
      Geometry~IV, (Eds: A.~N.~Parshin and
      I.~R.~Shafarevich), Springer-Verlag,
      Berlin/Heidelberg/New York, 1991.
    \bibitem [Ri1] {Ri1} R.~W.~Richardson, Commuting
      varieties of semisimple Lie algebras and algebraic
      groups, \emph{Compositio Math.}, {\bf38} \No3, (1979),
      311-327.
    \bibitem[Ri2]{Ri2} R.~W.~Richardson, Normality of
      G-stable subvarieties of a semisimple Lie
      algebra. {\em In}: Algebraic groups, Utrecht 1986,
      243--264, Lecture Notes in Math., {\bf 1271},
      Springer, Berlin, 1987.
    \bibitem[Ru] {Ru} H.~Rubenthaler, Param\'etrisation
      d'orbites dans les nappes de Dixmier admissibles,
      \emph{M\'em. Soc. Math. Fr.}, {\bf 15} (1984),
      255-275.
    \bibitem [SY] {SY2} H.~Sabourin and R.~W.~T.~Yu, On the
      irreducibility of the commuting variety of the
      symmetric pair $(\so_{p+2},\so_{p}\times \so_{2})$,
      \emph{J. Lie Theory}, {\bf16} (2006), 57-65.
    \bibitem[Sl]{Sl} P.~Slodowy, \emph{Simple singularities
        and simple algebraic groups}, Lecture Notes in
      Math., {\bf 815}, Springer-Verlag, 1980.
    \bibitem[St]{St} R.~Steinberg, Torsion in Reductive
      Groups, \emph{Adv. in Math.}, {\bf 15} (1995), 63-92.
    \bibitem[TY]{TY} P.~Tauvel and R.~W.~T.~Yu, \emph{Lie
        algebras and algebraic groups}, Springer Monographs
      in Mathematics, Springer-Verlag, 2005.
}
  \end{thebibliography}
\end{document}